\documentclass{aims}
\usepackage{amsmath}
  \usepackage{paralist}
  \usepackage{graphics} 
  \usepackage{epsfig} 
\usepackage{graphicx}  \usepackage{epstopdf}
 \usepackage[colorlinks=true]{hyperref}
   
\usepackage{caption}
\usepackage{subcaption}
\usepackage{mathrsfs}
\usepackage{cancel}

\usepackage{mathtools}
\DeclarePairedDelimiter\floor{\lfloor}{\rfloor}

\usepackage{tikz, tikz-cd}
\usetikzlibrary{shapes,arrows}
\tikzstyle{block} = [rectangle, draw, text width=5em, text centered, rounded corners, minimum height=4em]
\usetikzlibrary{calc,patterns,angles,quotes}
\usetikzlibrary{arrows, automata}
\usetikzlibrary{positioning}
\usepackage{tkz-euclide}
\usetikzlibrary{calc}
   
\hypersetup{urlcolor=blue, citecolor=red}

\newtheorem{theorem}{Theorem}[section]
\newtheorem{corollary}{Corollary}

\newtheorem{lemma}[theorem]{Lemma}
\newtheorem{proposition}{Proposition}
\newtheorem{conjecture}{Conjecture}

\theoremstyle{definition}
\newtheorem{definition}[theorem]{Definition}
\newtheorem{remark}{Remark}

\newtheorem{assumption}{Assumption}
\newtheorem{example}{Example}

\title[Hybrid Invariants] 
      {Invariant Forms in Hybrid and Impact Systems and a Taming of Zeno}

\author[W. Clark and A. Bloch]{}


 \email{wac76@cornell.edu}
 \email{abloch@umich.edu}

\thanks{The first author is supported by NSF grants DMS-1645643 and DMS-2103026. The second
	author is supported by NSF grant DMS-1613819 and AFOSR grant 77219283.}

\begin{document}
\maketitle

\centerline{\scshape William Clark}
\medskip
{\footnotesize
 \centerline{Department of Mathematics, Cornell University}
   \centerline{301 Tower Road, Ithaca, NY, USA}
} 

\medskip

\centerline{\scshape Anthony Bloch}
\medskip
{\footnotesize
 \centerline{Department of Mathematics, University of Michigan}
   \centerline{530 Church Street, Ann Arbor, MI, USA}
}

\bigskip


\begin{abstract}
Hybrid (and impact) systems are dynamical systems experiencing both continuous and discrete transitions. In this work, we derive necessary and sufficient conditions for when a given differential form is invariant, with special attention paid to the case of the existence of invariant volumes. Particular attention is given to impact systems where the continuous dynamics are Lagrangian and  subject to nonholonomic constraints. {A celebrated result for volume-preserving dynamical systems is Poincar\'{e} recurrence. In order to be recurrent, trajectories need to exist for long periods of time, which can be controlled in continuous-time systems through e.g. compactness. For hybrid systems, an additional mechanism can occur which breaks long-time existence: Zeno (infinitely many discrete transitions in a finite amount of time). 
We demonstrate that the existence of a smooth invariant volume severely inhibits Zeno behavior; hybrid systems with the ``boundary identity property'' along with an invariant volume-form have almost no Zeno trajectories (although Zeno trajectories can still exist). This leads to the result that many billiards (e.g. the classical point, the rolling disk, and the rolling ball) are recurrent independent on the shape of the compact table-top.}
\end{abstract}


\section{Introduction}
The standard billiard problem models the billiard ball as a point particle and studies its trajectory assuming that it moves in straight lines between impacts and under goes specular reflection (angle of incidence equals the angle of reflection) during impact, e.g. \cite{alksch,bunimovich,markov_billiards,KMS} and the references therein to name only a few. However, real billiard balls are not point particles; billiard balls are balls. An actual rolling ball is an example of a \textit{nonholonomic} system (it is subject to velocity-dependent constraints). Treating the billiard ball as an actual ball complicates the dynamics: the ball is no longer required to travel in straight lines as the ball's spin influences its trajectory and the symmetric nature of impacts is destroyed (linear momentum can be converted to angular momentum and vice versa). The goal of this work is to understand this situation and is split up into three main goals: (a) determine the proper impacts for nonholonomic systems, (b) understand and find conditions for when an arbitrary differential form is preserved under impacts, and (c) to use the power and utility of invariant forms to provide information on the Zeno phenomenon in hybrid systems. {The emphasis of this work is to deal with (b) and (c) while (a) will be used to supply a rich plethora of examples. We point out that there does not seem to be any results concerning goals (b) or (c) (with the exception of some partial and limited results discussed below).}

Classical mechanical systems that are not subject to nonholonomic (nonintegrable) constraints satisfy the principle of least action and, equivalently, their trajectories satisfy the Euler-Lagrange equations. Following this same paradigm, mechanical impacts are assumed to also satisfy the principle of least action which results in the well-known Weierstrass-Erdmann corner conditions (these conditions specify that energy is conserved at impact and the change in velocity is perpendicular to the {normal of the} impacting surface). Nonholonomic systems, by contrast, do not satisfy the principle of least action (they are \textit{not} variational). Rather they satisfy the Lagrange-d'Alembert principle which results in the ``dynamic nonholonomic equations of motion,'' \cite{bloch2008nonholonomic}. For nonholonomic systems with impacts, the impact equations are no longer given by the Weierstrass-Erdmann conditions, but rather a modified version arising from the Lagrange-d'Alembert principle. The first goal of this work is to derive this impact map and is the focus of Section \ref{sec:mechanical_impact}. {This problem has been studied, albeit differently, in e.g. \cite{cortesvinogradov2006,pasquero2005,pasquero2006,pasquero2018}.}

For a general continuous-time dynamical system generated by a differential equation $\dot{x}=X(x)$ where $x\in M$, a smooth function is a constant of motion if $\mathcal{L}_Xf=0$ where $\mathcal{L}$ is the usual Lie derivative. Moreover, an arbitrary differential $k$-form, $\alpha\in\Omega^k(M)$, is preserved under the flow if $\mathcal{L}_X\alpha=0$. For unconstrained mechanical systems, it is known that the energy (a 0-form) and the symplectic form (a 2-form) are always preserved under the flow. This question becomes more difficult for hybrid/impact systems as we lose continuity at the point of impact and differentiability falls away. {In the context of this work, a hybrid/impact system will be of the form
\begin{equation*}
	\begin{cases}
		\dot{x} = X(x), & x\in M\setminus S \\
		x^+ = \Delta(x^-), & x\in S
	\end{cases}
\end{equation*}
where $S\subset M$ is a codimension 1 embedded submanifold and $\Delta:S\to M$ is the reset map.}
A smooth function is a constant {of motion} if both $\mathcal{L}_Xf=0$ and $f\circ\Delta = f$. An important consequence of this is that energy is still a constant of motion for impact systems. 
{It is known, through more direct calculations, that the symplectic form is preserved in impact systems, \cite{impactSymp,asynchronousVHTG}. Be that as it may, a general test for invariant forms does not seem to exist and is rather subtle as $\Delta^*\alpha=\alpha$ is not the correct approach. The correct test is the topic of Section \ref{sec:h_diff_forms}.}

The culmination of our study of invariant differential forms is to examine invariant volumes and invariant measures to build measurable dynamics for hybrid/impact systems. The Krylov-Bogolyubov theorem \cite{krylov1937} still holds for hybrid systems \cite{collins} which guarantees that invariant measures always exist. Even though these measures exist, they can be singular which does not convey much information about the underlying system. As such, we apply the machinery for invariant differential forms to construct a ``hybrid cohomology equation'' which, if solvable, produces a \textit{smooth} hybrid-invariant measure.  Section \ref{sec:vol_mech} shows that there exists a smooth invariant volume for unconstrained mechanical systems with impacts and also lays out conditions for when nonholonomic systems with impacts possess an invariant volume. It turns out that if an invariant volume exists for a nonholonomic system with impacts, it exists \textit{independently} of the choice of impact surface. This implies that the billiard problem with an actual ball (or vertical rolling coin) undergoes Poincar\'{e} recurrence for \textit{any} bounded table-top {- with the important caveat of the possibility of Zeno trajectories.}

The final main endeavor of this work is to examine how measure-preservation influences the Zeno phenomenon in hybrid systems. Roughly speaking, a trajectory in a hybrid system is \textit{Zeno} if it undergoes infinitely many impacts in a finite amount of time. This issue has received a considerable amount of attention, cf. e.g.\cite{4434891,1656623,brogliato2016,mattKv,orstability}. 
{The standard example of hybrid systems with Zeno trajectories are impact systems with inelastic collisions \cite{suffZeno}. This class of examples possess both energy and volume dissipation which makes it difficult to ascertain which mechanism is responsible for Zeno, cf. Example \ref{ex:intro_bball} and Figure \ref{fig:intro_example_bball} below. By abandoning physical impact laws it is possible to get energy-preserving systems with Zeno, cf. Example \ref{ex:intro_cross} and Figure \ref{fig:intro_example_cross} below. Although energy is preserved, volume is still contracting.}
Section \ref{sec:Zeno} examines how the presence of an invariant measure inhibits
the existence of too many Zeno trajectories. Zeno can still occur, but (under a few extra regularity assumptions which mechanical impact systems satisfy) Zeno can \textit{almost never occur}. 

%

%

\begin{example}[Bouncing Ball]\label{ex:intro_bball}
	Consider the inelastic bouncing ball. The continuous dynamics are given by
	\begin{equation*}
		\dot{x} = \frac{1}{m}p, \quad \dot{p} = -mg,
	\end{equation*}
	where $x$ its the ball's height, $p$ its momentum, $m$ its mass, and $g$ is the acceleration due to gravity. When the ball strikes the table, the state changes according to the restitution law,
	\begin{equation*}
		\Delta(x,p) = (x,-c\cdot p),
	\end{equation*}
	where $0<c<1$.
	Between impacts, both volume and energy are conserved while both diminish at impacts as shown in Figure \ref{fig:intro_example_bball}. Both the energy and volume collapse to zero at the Zeno time. This seems to imply a connection between energy dissipation/volume contraction and Zeno. It turns out that energy conservation implies volume conservation (Proposition \ref{prop:unconst_hamilt_volume}) and volume conservation implies that Zeno \textit{almost never occurs} (Theorem \ref{thm:noZeno}).
\begin{figure}
	\begin{subfigure}[t]{.48\columnwidth}
		\centering
		\includegraphics[width=\linewidth]{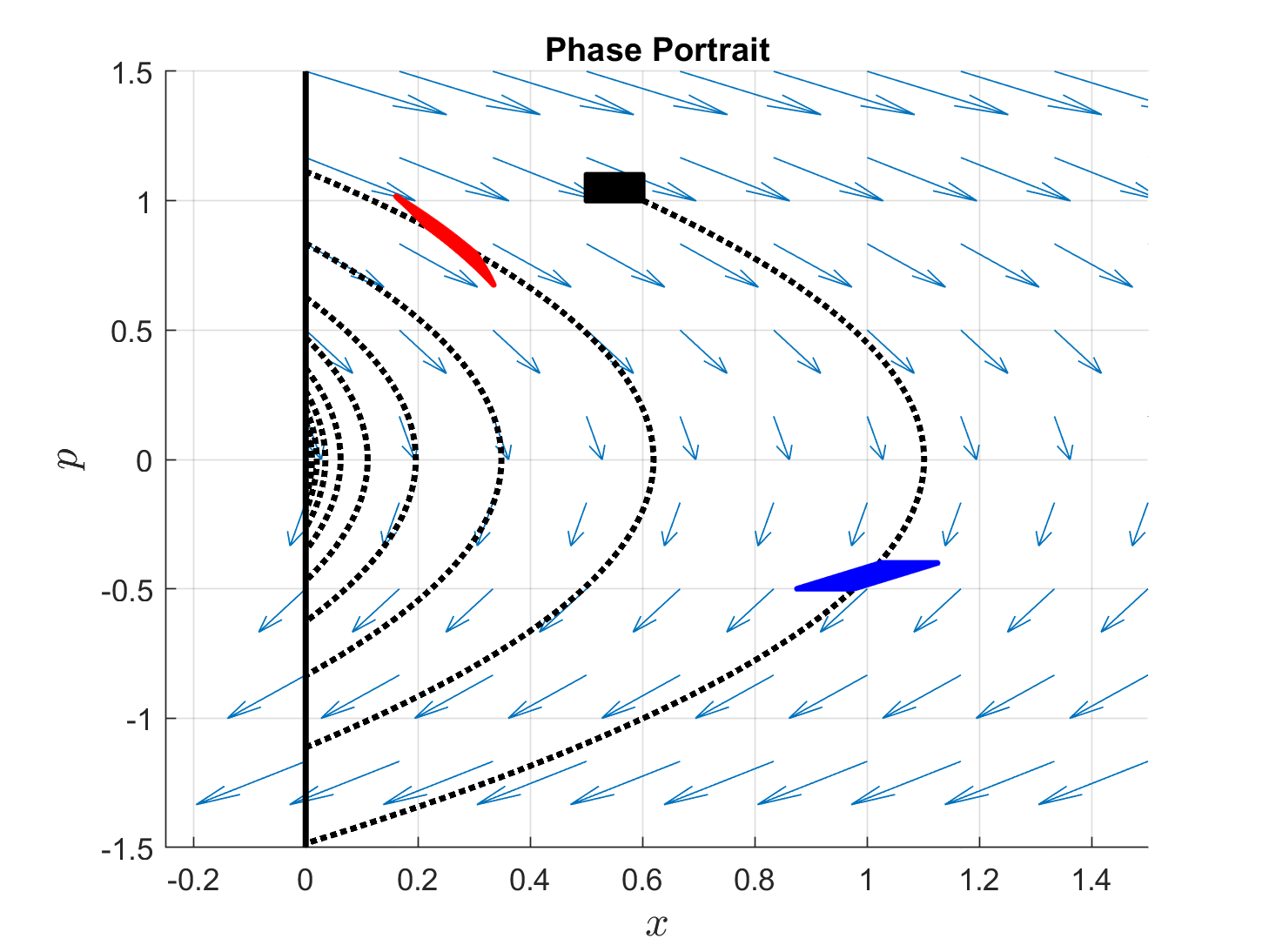}
	\end{subfigure}
	\hfill
	\begin{subfigure}[t]{.48\columnwidth}
		\centering
		\includegraphics[width=\linewidth]{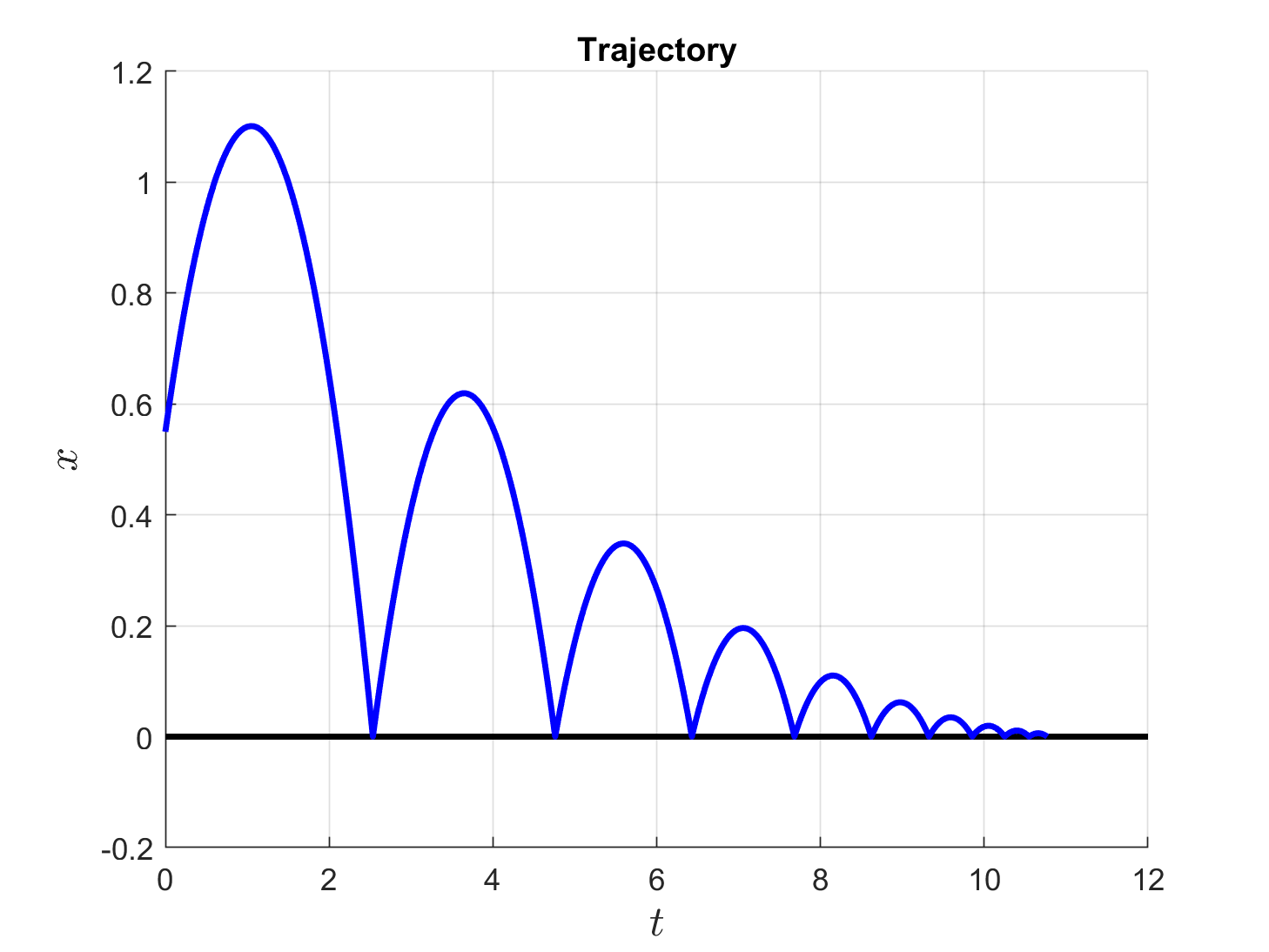}
	\end{subfigure}
	\medskip
	\begin{subfigure}[t]{0.48\columnwidth}
		\centering
		\includegraphics[width=\linewidth]{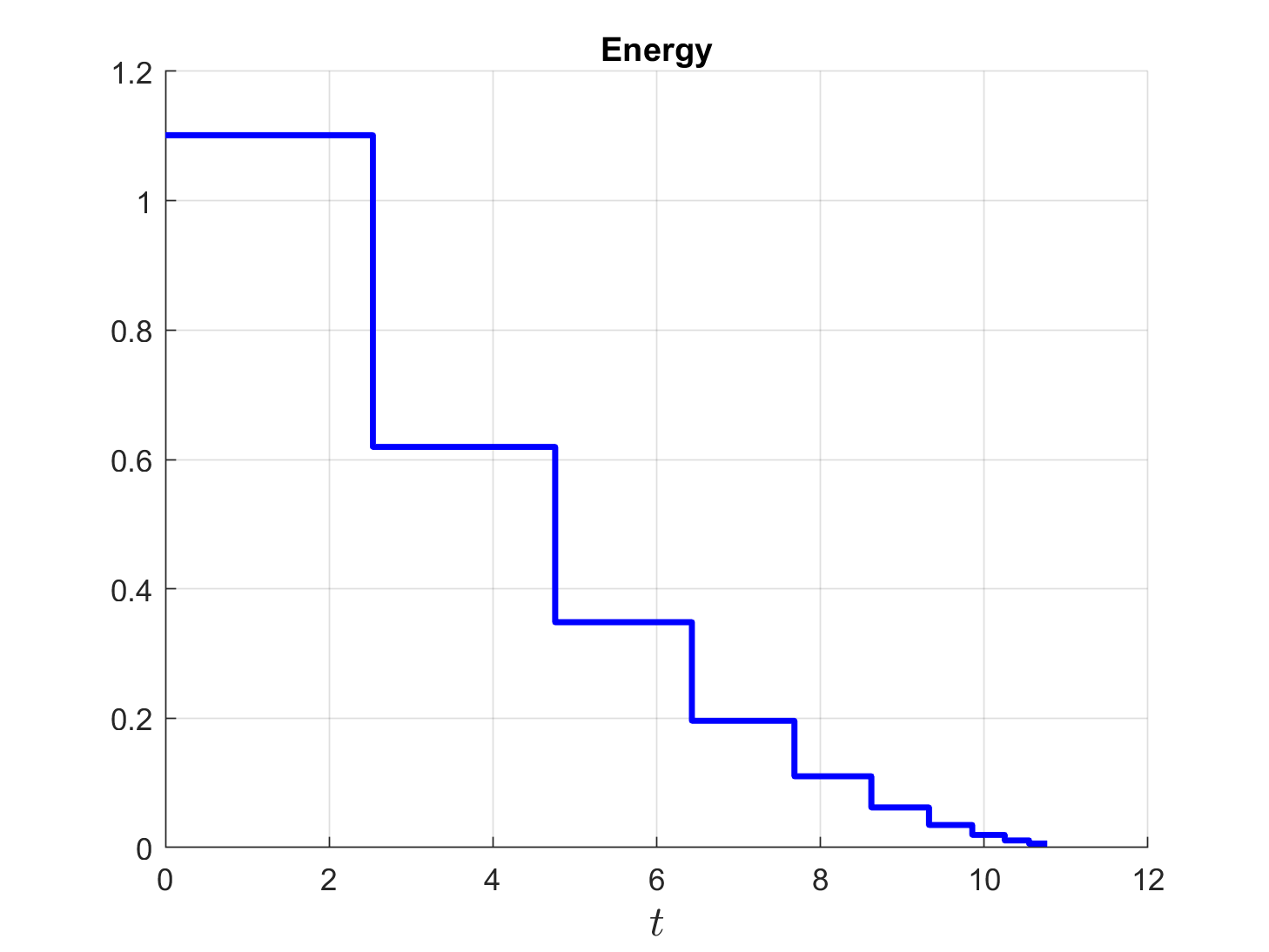}
	\end{subfigure}
	\hfill
	\begin{subfigure}[t]{0.48\columnwidth}
		\centering
		\includegraphics[width=\linewidth]{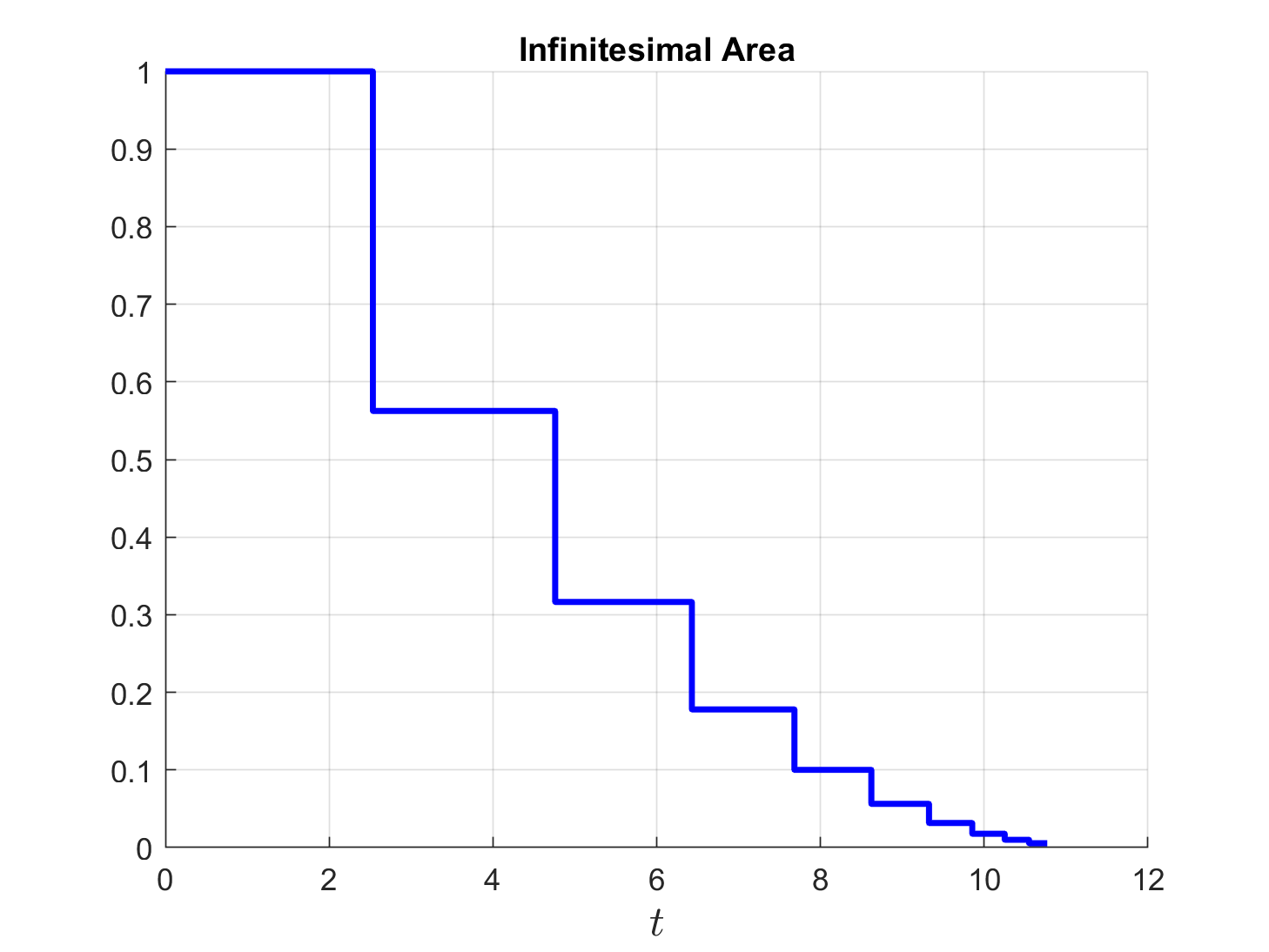}
	\end{subfigure}
	\caption{Plots of the bouncing ball with dissipative impacts, $m=g=1$ and $c=3/4$. Upper left: The phase portrait of the system. The blue arrows designate the continuous-time dynamics while the dashed curve is a trajectory. The black/blue/red regions illustrate propagation of phase volume. Due to continuous-time volume conservation of Hamiltonian systems, the black and blue regions have equal area. However, the red region has strictly smaller area due to the contraction of the impact. Upper right: The position vs time trajectory. Lower left: The energy of the ball vs time. Lower right: The infinitesimal volume vs time. 
	}
	\label{fig:intro_example_bball}
\end{figure}
\end{example}

\begin{example}[Energy preserving Zeno]\label{ex:intro_cross}
	Consider the planar system with dynamics
	\begin{equation*}
		\ddot{x} = 0, \quad \ddot{y} = 0,
	\end{equation*}
	with reset law
	\begin{equation*}
		\Delta(x,y,p_x,p_y) = (x,y,p_x,-p_y),
	\end{equation*}
	occurring when $x=y$ or $x=-y$. This system preserves energy but still results in Zeno when the trajectory enters the ``cross.''
\begin{figure}
	\begin{subfigure}[t]{.48\columnwidth}
		\centering
		\includegraphics[width=\linewidth]{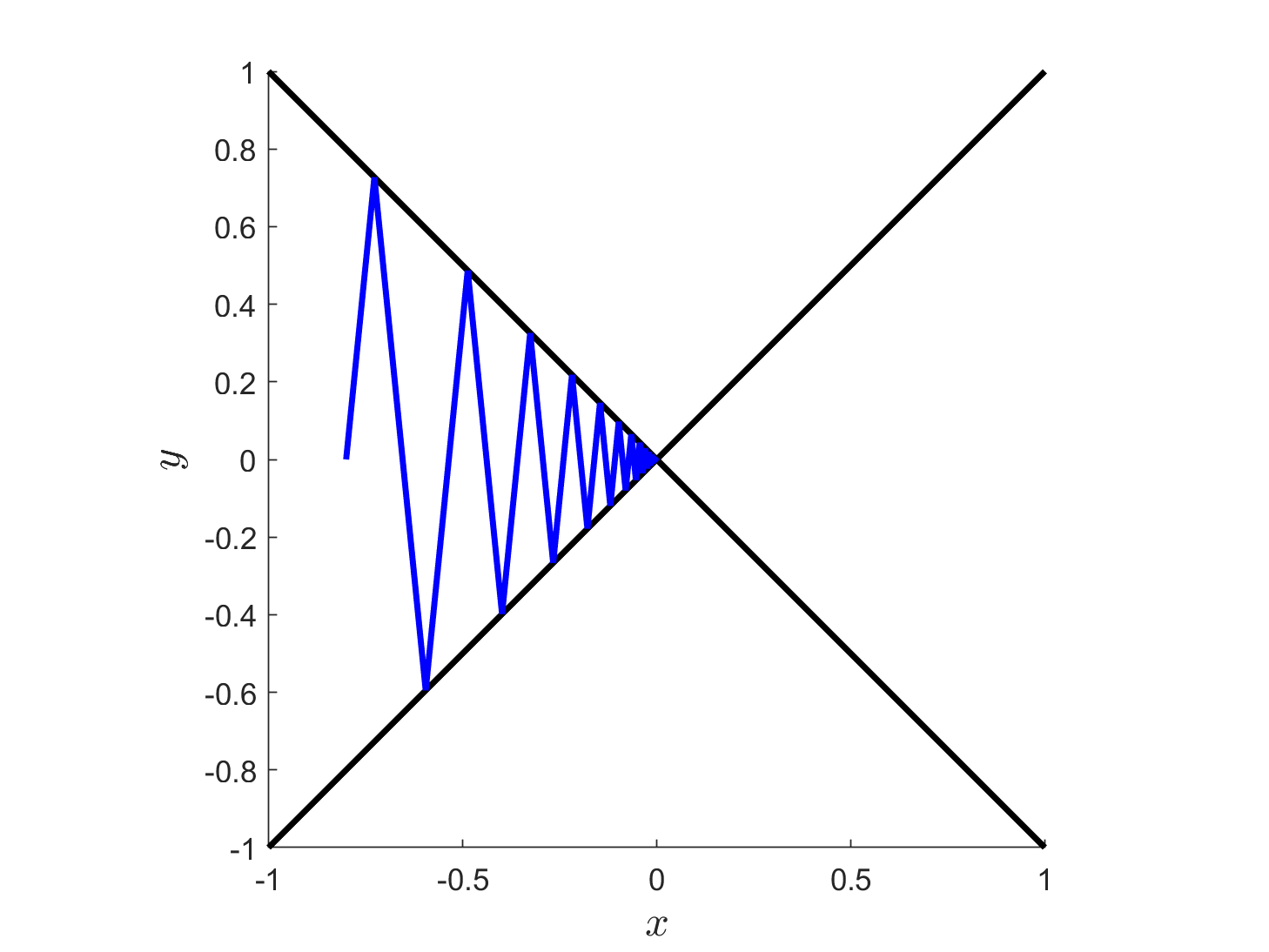}
	\end{subfigure}
	\hfill
	\begin{subfigure}[t]{.48\columnwidth}
		\centering
		\includegraphics[width=\linewidth]{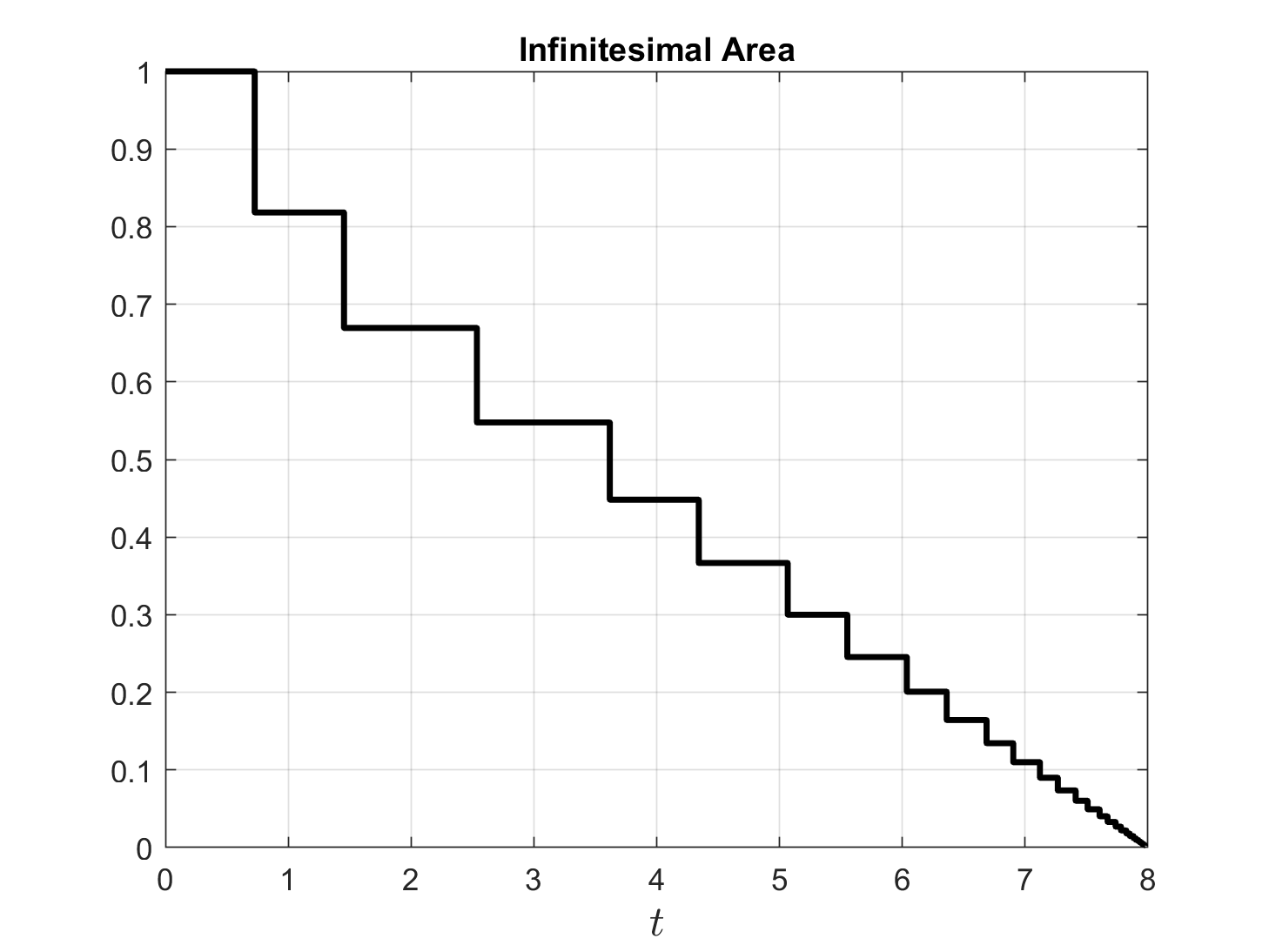}
	\end{subfigure}
	\caption{Left: The (Zeno) trajectory. Right: The infinitesimal volume around the initial condition.
	}
	\label{fig:intro_example_cross}
\end{figure}
\end{example}
This paper is organized as follows: Section \ref{sec:prelim} outlines basics of both hybrid dynamical systems and geometric mechanics. Section \ref{sec:mechanical_impact} derives the impact map for both unconstrained and nonholonomic impact systems. Section \ref{sec:h_diff_forms} derives conditions for whether or not differential forms are hybrid-invariant and section \ref{sec:vol_mech} specializes to the case of invariant volume forms in mechanical systems. Section \ref{sec:Zeno} explores how volume-preservation influences the existence of Zeno states {and presents an example of a volume-preserving impact system with a Zeno trajectory}. {Section \ref{sec:filippov} is a short section applying the results developed to Filippov systems.} Section \ref{sec:examples} contains two examples: the Chaplygin sleigh and the (vertical) rolling disk. Some future research directions are presented in section \ref{sec:future}.
\section{Preliminaries}\label{sec:prelim}
We review some notation and definitions from both hybrid systems and geometric mechanics.
\subsection{Hybrid systems}\label{ssec:hybrid}
This subsection is devoted to defining notation and presenting our version of a hybrid system and, as such, we will not be concerned with the minutia of defining the solution concept for hybrid systems. For more details on foundations of hybrid systems, see e.g. \cite{hdsgoebe,haddad2006}.

A hybrid dynamical system is a dynamical system that experiences both continuous and discrete transitions. There exist many different, nonequivalent, ways to formalize this idea. However, as we are concerned with modeling impact mechanics as hybrid systems, we will use the following definition for a hybrid system which depends on four pieces of data \cite{brogliato2016,fslip,haddad2006,MoGr2005}. Throughout, smooth will mean $C^\infty$ (although most results can be relaxed to $C^1$).
\begin{definition}\label{def:HDS}
	A hybrid dynamical system (HDS) is a 4-tuple, $\mathcal{H} = (M,S,X,\Delta)$, such that
	\begin{enumerate}
		\item[(H.1)] $M$ is a smooth (finite-dimensional) manifold,
		\item[(H.2)] $S\subset M$ is a smooth embedded submanifold with co-dimension 1,
		\item[(H.3)] $X:M\to TM$ is a smooth vector field,
		\item[(H.4)] $\Delta:S\to M$ is a smooth map {whose image is an embedded submanifold}, and
		\item[(H.5)] $S\cap\Delta(S)=\emptyset$ and $\overline{S}\cap\overline{\Delta(S)}$ has co-dimension at least 2.
	\end{enumerate}
\end{definition}
The manifold $M$ is called the state-space, $S$ the impact surface, $X$ the continuous dynamics, and $\Delta$ the impact map, discrete dynamics, or the reset map. {There are no assumptions on the rank of $\Delta$.}
\begin{remark}
	The axiom (H.5) is in place to disallow \textit{beating}. Beating is when {repeated resets} happen instantaneously{; this phenomenon will be ignored in this work as it is a detriment to differentiability, particularly this assumption is critical to Theorem \ref{thm:hybrid_smooth}.} 
	{It is important to point out that beating is not when multiple (different) impacts occur simultaneously, e.g. multi-legged walking when multiple legs strike the ground simultaneously. In this case $S$ would fail to be differentiable and violates (H.2).}
	This {axiom} will also come in useful in \S\ref{sec:Zeno} where we prove that if there exists a smooth invariant volume, then Zeno solutions almost never happen.
\end{remark}
The hybrid dynamics can be informally described as
\begin{equation}\label{eq:firsthybrid}
	\begin{cases}
		\dot{x} = X(x),& x\not\in S,\\
		x^+ = \Delta(x^-),& x^-\in S.
	\end{cases}
\end{equation}
That is, the dynamics follow the continuous dynamics $\dot{x}=X(x)$ away from $S$ and get reinitialized by $\Delta$ when the set $S$ is reached.
\subsubsection{Regularity of Solutions}
We end our preliminary discussion of hybrid systems with a section on regularity of their solutions. This is needed as we will be considering differential forms which requires a notion of differentiability. We start with the solution concept: \textit{the hybrid flow}.
\begin{definition}
	Let $\mathcal{H}=(M,S,X,\Delta)$ be an HDS. Let $\varphi:\mathbb{R} \times M\to M$ be the flow for the continuous dynamics $\dot{x}=X(x)$. Additionally, let $\varphi^\mathcal{H}:\mathbb{R}\times M\to M$ be the flow for the hybrid dynamics \eqref{eq:firsthybrid}, i.e. $\varphi^\mathcal{H}$ satisfies
	\begin{equation*}
		\varphi^\mathcal{H}(t,x) = \varphi(t,x)
	\end{equation*}
	if for all $s\in [0,t]$, $\varphi(s,x)\not\in S$, and if $\varphi(s,x)\not\in S$ for all $s\in [0,t)$ but $\varphi(t,x)\in S$, then
	\begin{equation*}
		\lim_{s\to t^+} \, \varphi^\mathcal{H}(s,x) = \Delta\left( \lim_{s\to t^-}\, \varphi^\mathcal{H}(s,t) \right).
	\end{equation*}
\end{definition}
For more details on the solution concept, cf. e.g. \cite{hdsgoebe}.

Obviously, $\varphi^\mathcal{H}$ will not be differentiable (as it is not continuous at the impact surface). However, it can satisfy the weaker property of being quasi-smooth, which is a similar idea to being quasi-continuous \cite{dishliev1991,haddad2006}.
\begin{definition}
	Consider a hybrid dynamical system $\mathcal{H}=(M,S,X,\Delta)$ with hybrid flow $\varphi^\mathcal{H}$. $\mathcal{H}$ has the quasi-smooth dependence property if for every $x\in M\setminus S$ and $t\in \mathbb{R}$ such that $\varphi^\mathcal{H}(t,x)\not\in S$, there exists an open neighborhood $x\in U$ such that $U\cap S=\emptyset$ and the map $\varphi^\mathcal{H}(t,\cdot):U\to M$ is smooth.
\end{definition}
The quasi-smooth dependence property follows, essentially, from the continuous flow and the impact surface being transverse, cf. Chapter 2 in \cite{clarkthesis}.
\begin{theorem}[\cite{clarkthesis}]\label{thm:hybrid_smooth}
	Let $\mathcal{H}=(M,S,X,\Delta)$ be a hybrid dynamical system satisfying (H.1)-(H.5). In addition, suppose that $\mathcal{H}$ satisfies
	\begin{itemize}
		\item[(A.1)] If $\varphi^\mathcal{H}(t,x)\in \overline{S}\setminus S$, then there exists $\varepsilon>0$ such that for all $0<\delta<\varepsilon$ we have $\varphi^\mathcal{H}(t+\delta,x)\not\in S$, and
		\item[(A.2)] For all $x\in S$, we have $T_xM = T_xS\oplus X(x)\mathbb{R}$.
	\end{itemize} 
	Then, $\mathcal{H}$ has the quasi-smooth dependence property and
	\begin{itemize}
		\item[(A.3)] If $\varphi^\mathcal{H}(t,x)\in S$, then there exists $\varepsilon > 0$ such that for all $0<\delta<\varepsilon$ we have $\varphi^\mathcal{H}(t+\delta,x)\not\in S$.
	\end{itemize}
\end{theorem}
\begin{definition}\label{def:smooth}
	A hybrid dynamical system satisfying (H.1)-(H.5) and (A.1)-(A.2) is called \textit{smooth}.
\end{definition}
\begin{remark}
	The condition (A.1) prohibits the trajectory from entering $S$ through $\overline{S}$. Condition (A.2) is that the continuous dynamics are transverse to $S$. Finally, (A.3) requires trajectories entering $S$ to immediately leave $S$. In the language of mechanical impact systems, (A.1) prohibits grazing impacts and (A.3) states that impacts must move the particle away from the obstacle. {These regularity assumptions are important to control Zeno and to allow for the quasi-smooth dependence property.}
	Plastic impacts are not smooth and are not considered in this work.
\end{remark}
\subsection{Geometric mechanics}\label{ssec:geomech}
Although we present criteria for invariant differential forms which apply for any smooth hybrid dynamical system, the focus of the results and examples will all be mechanical impact systems. These systems will be presented as Lagrangian/Hamiltonian systems. We review these systems here as well as nonholonomic constraints. Our overview will be brief; for more information cf., e.g. \cite{abraham2008foundations} and \cite{bloch2008nonholonomic}.
\subsubsection{Lagrangian mechanics}
For a mechanical system, the space of all possible positions is given by a (smooth) manifold $Q$ called the configuration space. Lagrangian mechanics is defined by a function on the tangent bundle $L:TQ\to\mathbb{R}$ called the Lagrangian function. The equations of motion are given by the Euler-Lagrange equations:
\begin{equation}\label{eq:Lagrange}
	\frac{d}{dt}\frac{\partial L}{\partial \dot{q}} - \frac{\partial L}{\partial q} = 0.
\end{equation}
For most physical examples, the Lagrangian is the difference between the system's kinetic and potential energy. Lagrangians of this form are called \textit{natural}.
\begin{definition}
	A Lagrangian $L:TQ\to\mathbb{R}$ is called natural if 
	\begin{equation*}
		L(q,\dot{q}) = \frac{1}{2}g_q(\dot{q},\dot{q}) - V(q),
	\end{equation*}
	where $g$ is a Riemannian metric on $Q$ and $V:Q\to\mathbb{R}$ is the potential energy.
\end{definition}
Throughout, all Lagrangians will be assumed to be natural.
\subsubsection{Hamiltonian mechanics}
While Lagrangian mechanics evolves on the tangent bundle, Hamiltonian systems evolve on the cotangent bundle. Given a Hamiltonian function, $H:T^*Q\to\mathbb{R}$, the dynamics are given by Hamilton's equations:
\begin{equation}\label{eq:Hamilton}
	\dot{q} = \frac{\partial H}{\partial p}, \quad \dot{p} = -\frac{\partial H}{\partial q}.
\end{equation}
Or, equivalently, by
\begin{equation*}
	dH = i_{X_H}\omega,
\end{equation*}
where $\omega = dq^i\wedge dp_i$ is the standard symplectic form on $T^*Q$ and $i_X\omega = \omega(X,\cdot)$ is the contraction.

Lagrangian and Hamiltonian systems are intimately related through the Legendre transform.
We first define the fiber derivative:
\begin{equation*}
	\mathbb{F}L:TQ\to T^*Q, \quad \mathbb{F}L(v)(w) = \left.\frac{d}{dt}\right|_{t=0}\, L(q,v + tw).
\end{equation*}
When $L$ is natural, the fiber derivative is a diffeomorphism and we have $\mathbb{F}L(v) = g(v,\cdot)$. As long as the fiber derivative is invertable, we can define the Legendre transform via
\begin{equation*}
	H(q,p) = p(\dot{q}) - L(q,\dot{q}), \quad p = \mathbb{F}L(\dot{q}).
\end{equation*}
With this association, the equations of motion \eqref{eq:Lagrange} and \eqref{eq:Hamilton} are equivalent (cf. 3.6.2 in \cite{abraham2008foundations}).

A famous property of Hamiltonian systems is that they preserve the symplectic form (and, consequently, the induced volume form), i.e. if $\varphi_t$ is the flow of a Hamiltonian system \eqref{eq:Hamilton}, then $\varphi_t^*\omega = \omega$ and $\varphi_t^*\omega^n = \omega^n$. We will show in Section \ref{sec:vol_mech} that this property will still hold true when impacts are present.
\subsubsection{Nonholonomic mechanics}
We end our preliminary discussion with a brief overview of nonholonomic systems. Constraints in Lagrangian systems manifest as specifying a submanifold $\mathcal{D}\subset TQ$ such that the dynamics are required to evolve on $\mathcal{D}$. Throughout this work, we will assume that $\mathcal{D}$ is a \textit{distribution}, i.e. the constraints are linear in the velocities. For more information on nonholonomic systems, cf. e.g. \cite{bloch2008nonholonomic,neimark1972dynamics}.

For our purposes, the constraint manifold $\mathcal{D}$ will be given by the joint kernels of differential 1-forms, i.e.
\begin{equation*}
	\mathcal{D} = \bigcap_{\alpha=1}^m \, \ker\eta^\alpha = 
	\left\{ (q,\dot{q})\in TQ : \eta^\alpha_q(\dot{q})=0\right\}, \quad \eta^\alpha\in \Omega^1(Q).
\end{equation*}
With these constraints, the equations of motion according to the Lagrange-d'Alembert principle are
\begin{gather*}
	\frac{d}{dt}\frac{\partial L}{\partial\dot{q}} - \frac{\partial L}{\partial q} = \lambda_\alpha {\tau_Q^*}\eta^\alpha, \quad {\tau_Q:TQ\to Q} \\
	\eta^\alpha(\dot{q}) = 0,
\end{gather*}
in the Lagrangian formalism. These can be equivalently described in Hamilton's formalism via
\begin{gather*}
	i_{X_H^\mathcal{D}}\omega = dH + \lambda_\alpha\pi_Q^*\eta^\alpha \\
	P(W^\alpha)(q,p) := \langle p, W^\alpha(q)\rangle = 0,
\end{gather*}
where $W^\alpha = \mathbb{F}L^{-1}\eta^\alpha$ are dual vector fields  and $\pi_Q:T^*Q\to Q$ is the canonical cotangent projection.
A useful matrix that will appear in many of the computations throughout this work is the \textit{constraint mass matrix}.
\begin{definition}
	Let $\mathscr{C} = \{\eta^1,\ldots,\eta^m\}$ be a collection of 1-forms describing the constraint manifold $\mathcal{D}\subset TQ$. Let $W^\alpha := \mathbb{F}L^{-1}\eta^\alpha$ be the corresponding vector fields with a natural Lagrangian. Then, the constraint mass matrix, $\left(m^{\alpha\beta}\right)$ is given by
	\begin{equation*}
		m^{\alpha\beta} = g(W^\alpha,W^\beta) = \eta^\alpha(W^\beta).
	\end{equation*}
	The inverse will be denoted by $\left(m_{\alpha\beta}\right) = \left(m^{\alpha\beta}\right)^{-1}$.
\end{definition}

\begin{remark}\label{rmk:GNVF}
	The nonholonomic vector field $X_H^\mathcal{D}$ is a vector field on the constraint submanifold $\mathcal{D}^* := \mathbb{F}L(\mathcal{D})$. However, we can extend this to a \textit{global} vector field $\Xi_H^\mathscr{C}$ on the entire space $T^*Q$ such that $\Xi_H^\mathscr{C}|_{\mathcal{D}^*} = X_H^\mathcal{D}$. In \cite{nhvolume}, the global nonholonomic vector field is shown to be given by
	\begin{equation*}
		i_{\Xi_H^\mathscr{C}}\omega = dH -m_{\alpha\beta}\left\{ H,P(W^\alpha)\right\}\pi_Q^*\eta^\beta.
	\end{equation*}
	This will be helpful in the computations in Section \ref{sec:vol_mech} and we will call $\nu_H^\mathscr{C} := i_{\Xi_H^\mathscr{C}}\omega$ the \textit{nonholonomic 1-form}. 
\end{remark}

One goal of this work is to understand invariant volumes in hybrid systems. 
For unconstrained systems, Liouville's theorem states that they preserve the symplectic form and, consequently, the induced volume form as well. However, nonholonomic systems need not be volume-preserving. Below, we state a nonholonomic version of Liouville's theorem as proved in \cite{nhvolume} (a similar result can be found in \cite{federovnaranjo}). Recall that $\eta^\beta\in\Omega^1(Q)$ define our constraints, $W^\beta = \mathbb{F}L^{-1}\eta^\beta$ are their corresponding vector fields, and for a bundle $M\to Q$, $\Gamma(M)$ is the space of sections.

\begin{theorem}[\cite{nhvolume}]\label{th:NHVol}
	Let $L:TQ\to\mathbb{R}$ be a natural Lagrangian and $\mathcal{D}\subset TQ$ be a regular distribution. Then there exists an invariant volume with density depending only on the configuration variables if and only if there exists $\rho\in\Gamma(\mathcal{D}^0)$ such that $\vartheta_\mathscr{C}+\rho$ is exact where
	\begin{equation*}
		\vartheta_\mathscr{C} = m_{\alpha\beta} \cdot \mathcal{L}_{W^\alpha}\eta^\beta.
	\end{equation*}
	Here, $\mathcal{D}^0\subset T^*Q$ is the annihilator of $\mathcal{D}\subset TQ$, $\{\eta^\beta\}$ is a frame for $\mathcal{D}^0$, $W^\alpha = \mathbb{F}L^{-1}\eta^\alpha$, and $m^{\alpha\beta} = \eta^\alpha(W^\beta)$.
	In particular, suppose that $\vartheta_\mathscr{C}+\rho = dg$. Then the following volume form is preserved:
	\begin{equation*}
		\exp\left(\pi_Q^*g\right)\cdot \mu_\mathscr{C},
	\end{equation*}
	where $\mu_\mathscr{C}$ is the nonholonomic volume form described in \cite{nhvolume}.
\end{theorem}
In particular, we will be interested in whether or not nonholonomic systems with an invariant volume prescribed via  Theorem \ref{th:NHVol} continue to preserve this volume when impacts are present.
\section{Mechanical impact systems}\label{sec:mechanical_impact}
This section is devoted to fusing the ideas of \S\ref{ssec:hybrid} and \S\ref{ssec:geomech}. Hybrid systems built from mechanical systems have the form $\mathcal{H} = (M,S,X,\Delta)$ where $M = TQ$ or $T^*Q$ depending on Lagrangian/Hamiltonian and $X$ is either \eqref{eq:Lagrange} or \eqref{eq:Hamilton}. The set $S$ is the location of impact and we make an abuse of notation where $S\subset Q$ rather than $S\subset M$ as impacts will depend on \textit{location only}. The final piece of information we need to construct a mechanical hybrid system is the map $\Delta$. In order to construct a meaningful impact map, we make the following assumption (cf. \S 3.5 in \cite{brogliato2016}):
\begin{assumption}\label{ass:impact}
	A mechanical impact is the identity on the base and satisfies variational/Lagrange-d'Alembert principles on the fibers. In particular, the impact map will have the form $\Delta = (Id,\delta)$, e.g. for Hamiltonian systems we have $\pi_Q\circ\Delta = \pi_Q$ where $\pi_Q:T^*Q\to Q$ is the cotangent projection.
\end{assumption}
Before we discuss the construction of the map $\Delta$, we first clear up the notation surrounding $S$. As $S\subset Q$ is an embedded codimension 1 submanifold, it can be (locally) described as the level-set of a smooth function $h:Q\to\mathbb{R}$. This allows us to define the following five sets. 
\begin{enumerate}
	\item $S = \{q\in Q : h(q) = 0\}\subset Q$, 
	\item $\hat{S} = \{(q,\dot{q})\in TQ: h(q)=0,\, dh(\dot{q})<0\}\subset TQ$,
	\item $S^* = \{(q,p)\in T^*Q: h(q) = 0,\, P(\nabla h) < 0\} \subset T^*Q$,
	\item $\hat{S}_\mathcal{D} = \hat{S}\cap \mathcal{D}\subset \mathcal{D}$, and
	\item $S^*_\mathcal{D} = S^*\cap \mathcal{D}^* \subset \mathcal{D}^* = \mathbb{F}L(\mathcal{D})$.
\end{enumerate}
These sets have the following classification: $S$ is the location of impact, $\hat{S}$ (resp. $S^*$) is the impact surface for unconstrained Lagrangian (resp. Hamiltonian) systems, and $\hat{S}_\mathcal{D}$ (resp. $S^*_\mathcal{D}$) is the impact surface for nonholonomic Lagrangian (resp. Hamiltonian) systems. 

For nonholonomic impacts, issues can arise when $dh\in\mathcal{D}^0$ which has the impact surface as a constraint. To circumnavigate this issue, we make the following assumption.
\begin{assumption}[Nontrivial impact condition]
	Suppose that $S\subset Q$ is given by $S = h^{-1}(0)$. Then $dh|_S\not\in \mathcal{D}^0|_S$.
\end{assumption}
\subsection{Holonomic impacts}
To derive $\Delta$ for unconstrained mechanical systems, we begin with the observation that the Euler-Lagrange equations are variational. With this, we make the assumption that the impact map is as well (cf. Assumption \ref{ass:impact}). This is realized by the Weierstrass-Erdmann corner conditions, cf. \S3.5 of \cite{brogliato2016} or \S4.4 of \cite{kirkOptimal}:
\begin{equation}\label{eq:lagrange_impact}
	\begin{split}
		\mathbb{F}L^+ - \mathbb{F}L^- &= \varepsilon \cdot dh, \\
		L^+ - \langle \mathbb{F}L^+, \dot{q}^+ \rangle &= L^- - \langle \mathbb{F}L^-,\dot{q}^- \rangle,
	\end{split}
\end{equation}
where the multiplier $\varepsilon$ is chosen such that both equations are satisfied. These equations have a cleaner interpretation on the Hamiltonian side:
\begin{equation}\label{eq:hamiltonian_impact}
	\begin{split}
		p^+ &= p^- + \varepsilon \cdot dh, \\
		H^+ &= H^-,
	\end{split}
\end{equation}
i.e. energy is conserved and the change in momentum is perpendicular to the impact surface.

In the case where $L$ is natural, the corner conditions can be explicitly solved. Recall that $\nabla h = dh^\sharp$, or $dh = g(\nabla h,\cdot)$. 
\begin{theorem}
	Given a natural Lagrangian $L(q,\dot{q}) = \frac{1}{2}g_q(\dot{q},\dot{q}) - V(q)$, the impact map $\Delta:\hat{S}\to TQ$ with $(q,\dot{q})\mapsto (q,\delta(q,\dot{q}))$ is given by
	\begin{equation*}
		\delta(q,\dot{q}) = \dot{q} - 2\frac{dh(\dot{q})}{g(\nabla h,\nabla h)} \nabla h.
	\end{equation*}
\end{theorem}
\subsection{Nonholonomic impacts}
Unlike unconstrained systems, nonholonomic systems are no longer variational. As such, the Weierstrass-Erdmann conditions no longer apply. However, we can instead utilize the Lagrange-d'Alembert principle. This leads to a modified version of \eqref{eq:lagrange_impact}, \cite{bpenny}:
\begin{equation}\label{eq:NH_corner}
	\begin{split}
		\mathbb{F}L^+ - \mathbb{F}L^- &= \lambda_k\cdot\eta^k + \varepsilon\cdot dh, \\
		L^+ - \langle \mathbb{F}L^+,\dot{q}^+\rangle &= L^- - \langle \mathbb{F}L^-,\dot{q}^-\rangle, \\
		\eta^k(\dot{q}^+) ^+ &= 0.
	\end{split}
\end{equation}
Again, when the Lagrangian is natural, the nonholonomic corner conditions can be explicitly solved.
\begin{theorem}
	Suppose that $W^\alpha = \mathbb{F}L^{-1}\eta^\alpha$. Then the nonholonomic impact map $\Delta^\mathcal{D}:\hat{S}_\mathcal{D}\to \mathcal{D}$ with $(q,\dot{q})\mapsto (q,\delta(q,\dot{q}))$ is given by
	\begin{equation}\label{eq:restricted_NH_form}
		\delta(q,\dot{q}) = \dot{q} + \lambda_k \cdot W^k + \varepsilon\cdot \nabla h,
	\end{equation}
	where
	\begin{equation}\label{eq:restricted_NH_impact}
		\begin{split}
			\varepsilon &= \frac{-2\cdot dh(\dot{q})}{ dh(\nabla h) - m_{\alpha\beta}\cdot dh(W^\alpha)dh(W^\beta)}, \\
			\lambda_k &= \frac{2\cdot m_{k\ell}\cdot dh(W^\ell)dh(\dot{q})}{ dh(\nabla h) - m_{\alpha\beta}\cdot dh(W^\alpha)dh(W^\beta) }.
		\end{split}
	\end{equation}
\end{theorem}
\begin{remark}\label{rmk:global_impact}
	Strictly speaking, \eqref{eq:restricted_NH_impact} is only defined on $\mathcal{D}\subset TQ$. However, we can define a \textit{global} version (which mimics the global aspect of \cite{nhvolume}). The constraints are given by specifying the submanifold $\mathcal{D}$ which is the joint zero level-set of the $\eta^\beta$. However, $\mathcal{D}$ does not \textit{uniquely} determine the 1-forms $\eta^\beta$. 
	We refer to this arbitrary choice of writing the constraints as a \textit{realization}. In this sense, we let $\mathscr{C} = \{\eta^1,\ldots,\eta^m\}$ be our choice for representing the constraints. Then we can define a \textit{global} impact map $\Delta^\mathscr{C}:\hat{S}\to TQ$ with form \eqref{eq:restricted_NH_form} but with multipliers
	\begin{equation}\label{eq:global_NH_impact}
		\begin{split}
			\varepsilon &= \frac{2m_{\alpha\beta}\cdot dh\left(W^\beta\right)\eta^\alpha(\dot{q}) - 2\cdot dh(\dot{q})}
			{dh(\nabla h) - m_{\alpha\beta}\cdot dh(W^\alpha)dh(W^\beta)} \\
			\lambda_k &= -2m_{k\ell}\cdot dh(W^\ell) \cdot
			\frac{m_{\alpha\beta}\cdot dh(W^\beta)\eta^\alpha(\dot{q}) - dh(\dot{q})}
			{dh(\nabla h) - m_{\alpha\beta}\cdot dh(W^\alpha)dh(W^\beta)}.
		\end{split}
	\end{equation}
	Notice that upon restriction, $\Delta^\mathscr{C}|_{\hat{S}_\mathcal{D}} = \Delta^\mathcal{D}$. This will result in two different versions of nonholonomic impact systems: the local version $\mathcal{H}^\mathcal{D} = (\mathcal{D},\hat{S}_{\mathcal{D}}, X^\mathcal{D}, \Delta^\mathcal{D})$, and a global version $\mathcal{H}^\mathscr{C} = (TQ,\hat{S},\Xi_H^\mathscr{C},\Delta^\mathscr{C})$. See Remark \ref{rmk:GNVF} and \cite{nhvolume} for the global nonholonomic vector field $\Xi_H^\mathscr{C}$.
\end{remark}
\subsection{Intrinsic Formulation of Impacts}
Below, we present an intrinsic view for both the holonomic and nonholonomic impact relations, \eqref{eq:hamiltonian_impact} and \eqref{eq:NH_corner}. Both versions will be stated from a Hamiltonian point of view.
\subsubsection{Holonomic Impacts}
Let $(T^*Q,S^*,X_H,\Delta)$ be an impact Hamiltonian system. 
\begin{theorem}
	The corner conditions \eqref{eq:hamiltonian_impact} are equivalent to
	\begin{equation}\label{eq:intrinsic_impact}
		\left( \mathrm{Id}\times\Delta\right)^*\vartheta_H = \iota^*\vartheta_H,
	\end{equation}
	where $\iota:\mathbb{R}\times S^*\hookrightarrow \mathbb{R}\times T^*Q$ is the inclusion and $\vartheta_H$ is the action form,
	\begin{equation*}
		\vartheta_H = p_i\cdot dx^i - H\cdot dt \in \Omega^1\left(\mathbb{R}\times T^*Q\right).
	\end{equation*}
\end{theorem}
\begin{proof}
	This follows directly from choosing local coordinates such that impact occurs when the last coordinate vanishes.
\end{proof}
This seems to imply that the action form is preserved across impacts; as will be seen in Proposition \ref{prop:unconst_hamilt_volume}, the action form is only a relative invariant.

We obtain the following intrinsic description of impact Hamiltonian mechanics.
\begin{equation*}
	\begin{cases}
		i_X\omega - dH = 0, & (x,p)\not\in S^* \\
		\left(\mathrm{Id}\times\Delta\right)^*\vartheta_H - \iota^*\vartheta_H = 0, & (x,p)\in S^*.
	\end{cases}
\end{equation*}
\begin{remark}
	The intrinsic corner condition \eqref{eq:intrinsic_impact} also describes the case of moving impacts. If $\tilde{S}\subset \mathbb{R}\times Q$ is a time-dependent surface described via $S = \{h(t,x)=0\}$, then \eqref{eq:intrinsic_impact} states that
	\begin{equation*}
		\begin{split}
			\left( p_i^+ - p_i^- \right) &= \varepsilon\cdot \frac{\partial h}{\partial x^i} \\
			-\left( H^+ - H^- \right) &= \varepsilon\cdot \frac{\partial h}{\partial t}.
		\end{split}
	\end{equation*}
\end{remark}
\subsubsection{Nonholonomic Impacts}
We first present an intrinsic form on the continuous equations of motion in a nonholonomic system, cf. \cite{monforte2004geometric}.
\begin{equation*}
	\left.\left(i_{X_H^\mathcal{D}}\omega - dH\right)\right|_{\mathcal{D}^*} \in \mathcal{F}^\circ,
\end{equation*}
where $\mathcal{F}^\circ\subset T^*(T^*Q)|_{\mathcal{D}^*}$ is the Chetaev bundle given by
\begin{equation*}
	\mathcal{F}^\circ = \pi_Q^*\left(\mathcal{D}^\circ\right) = \mathcal{C}^*\left( (T\mathcal{D}^*)^\circ\right),
\end{equation*}
where $\mathcal{C}:T(T^*Q)\to T(T^*Q)$ is the $\mathbb{F}L$-related almost-tangent structure, \cite{nhvolume}. We specialize the Chetaev bundle to be along impacts via its restriction
\begin{equation*}
	\mathcal{F}_S^\circ = \mathcal{F}^\circ \cap \left.T^*\left(T^*Q\right)\right|_{S^*_\mathcal{D}}.
\end{equation*}
\begin{theorem}
	The nonholonomic corner conditions, \eqref{eq:NH_corner}, are equivalent to
	\begin{equation}
		\begin{split}
			\left(\mathrm{Id}\times\Delta\right)^*\vartheta_H - \iota^*\vartheta_H &\in \tau^*\mathcal{F}^\circ_S, \\
			\eta\left(\dot{q}^+\right) &= 0,
		\end{split}
	\end{equation}
	where $\tau:\mathbb{R}\times T^*Q\to T^*Q$ is the projection into the second component.
\end{theorem}
This provides us with the intrinsic description of impact nonholonomic Hamiltonian mechanics.
\begin{equation*}
	\begin{cases}
		\left.\left( i_{X_H^\mathcal{D}}\omega - dH\right)\right|_{\mathcal{D}^*} \in \mathcal{F}^\circ, & (x,p)\not\in S^*_\mathcal{D} \\
		\left(\mathrm{Id}\times\Delta\right)^*\vartheta_H - \iota^*\vartheta_H \in \tau^*\mathcal{F}_S^\circ, & (x,p)\in S^*_\mathcal{D},
	\end{cases}
\end{equation*}
along with satisfying the constraints.
\subsection{Regularity of mechanical hybrid systems}
This subsection proves that mechanical hybrid systems are smooth as per Definition \ref{def:smooth}. We will prove that only unconstrained mechanical systems are smooth as the nonholonomic case follows similarly.
\begin{proposition}
	Let $\mathcal{H} = (T^*Q,S^*,X_H,\Delta)$ be a hybrid Hamiltonian system. Then $\mathcal{H}$ is smooth.
\end{proposition}
\begin{proof}
	We need to show that $\mathcal{H}$ satisfies (H.1)-(H.5) along with (A.1) and (A.2). Conditions (H.1)-(H.4) are immediate. For (H.5), we notice that
	\begin{equation*}
		\begin{split}
			S^* &= \{ (q,p)\in T^*Q : h(q)=0,\, P(\nabla h)<0\}, \\
			\Delta(S^*) &= \{(q,p) \in T^*Q:h(q)=0,\, P(\nabla h) > 0\}, \\
			\overline{S^*}\cap\overline{\Delta(S^*)} &= \{(q,p)\in T^*Q: h(q) = P(\nabla h) = 0\}.
		\end{split}
	\end{equation*}
	Therefore, $S^*\cap\Delta(S^*)=\emptyset$ and $\overline{S^*}\cap\overline{\Delta(S^*)}$ has codimension 2 so (H.5) is satisfied.
	
	For (A.1), assume that $(q(0),p(0))\in \overline{S^*}\cap\overline{\Delta(S^*)}=\overline{S^*}\setminus S^*$ and that there exists $\varepsilon>0$ such that for all $\delta\in (0,\varepsilon)$, we have $(q(\delta),p(\delta))\in S^*$. Since $q(\delta)\in S$ and $P(\delta h)(q(\delta),p(\delta)) <0$, $q(t)$ must intersect $S$ transversely at $\delta$. This leads to a contradiction.
	
	To finish the proof, we need to show that for $(q,p)\in S^*$, we have the direct sum:
	\begin{equation*}
		T_{(q,p)}T^*Q = T_{(q,p)}S^* \oplus X_H\cdot\mathbb{R},
	\end{equation*}
	i.e. $X_H$ is not tangent to $S^*$. This follows from similar reasoning to (A.1); let $\gamma(t)$ be a base curve of $X_H$, then $\gamma$ intersects $S$ transversely. This gives us (A.2) and we are done.
\end{proof}
\subsection{Refraction}
The intrinsic impact condition \eqref{eq:intrinsic_impact} can also describe refraction. Let $Q$ be a smooth manifold with separating hyper-surface $S$ given by the zero level-set of a function $h:Q\to\mathbb{R}$. Partition $Q$ into
\begin{equation*}
	Q^+ = \left\{ q\in Q: h(q) > 0\right\}, \quad Q^- = \left\{ q\in Q : h(q) < 0\right\}.
\end{equation*}
Endow each piece with a distinct Hamiltonian to obtain two Hamiltonian systems, $(T^*Q^+,H^+)$ and $(T^*Q^-,H^-)$. The variational reset map from exiting $Q^+$ and entering $Q^-$ is given by
\begin{equation}\label{eq:refracted}
	\left(\mathrm{Id}\times\Delta\right)^*\vartheta_{H^+} = \iota^*\vartheta_{H^-}.
\end{equation}
\begin{example}[Sphere in the Plane]\label{ex:planar_sphere}
	Suppose that $Q=\mathbb{R}^2$ is the plane. Outside the circle with radius 1/2, the kinetic energy is given by the flat metric while it is spherical within the circle, i.e. $h = x^2+y^2-1/4$ and
	\begin{equation*}
		H^+ = \frac{1}{2}\left(p_x^2+p_y^2\right), \quad H^- = \frac{1}{2}
		\begin{bmatrix}
			p_x & p_y
		\end{bmatrix}
		\cdot M^{-1}\cdot \begin{bmatrix}
			p_x \\ p_y
		\end{bmatrix},
	\end{equation*}
	where
	\begin{equation*}
		M = \left[\begin{array}{cc}
			1 - \dfrac{x^2}{1-x^2-y^2} & \dfrac{xy}{1-x^2-y^2} \\
			\dfrac{xy}{1-x^2-y^2} & 1 - \dfrac{y^2}{1-x^2-y^2}
		\end{array}\right].
	\end{equation*}
\end{example}
Solving \eqref{eq:refracted} produces two solutions. One solution corresponds to actual refraction while the other is reflection. Notice that the reflection solution is invalid as it preserves the \textit{wrong} energy. See Fig \ref{fig:circle_refraction}.
\begin{figure}
	\begin{subfigure}[t]{.48\columnwidth}
		\centering
		\includegraphics[width=\linewidth]{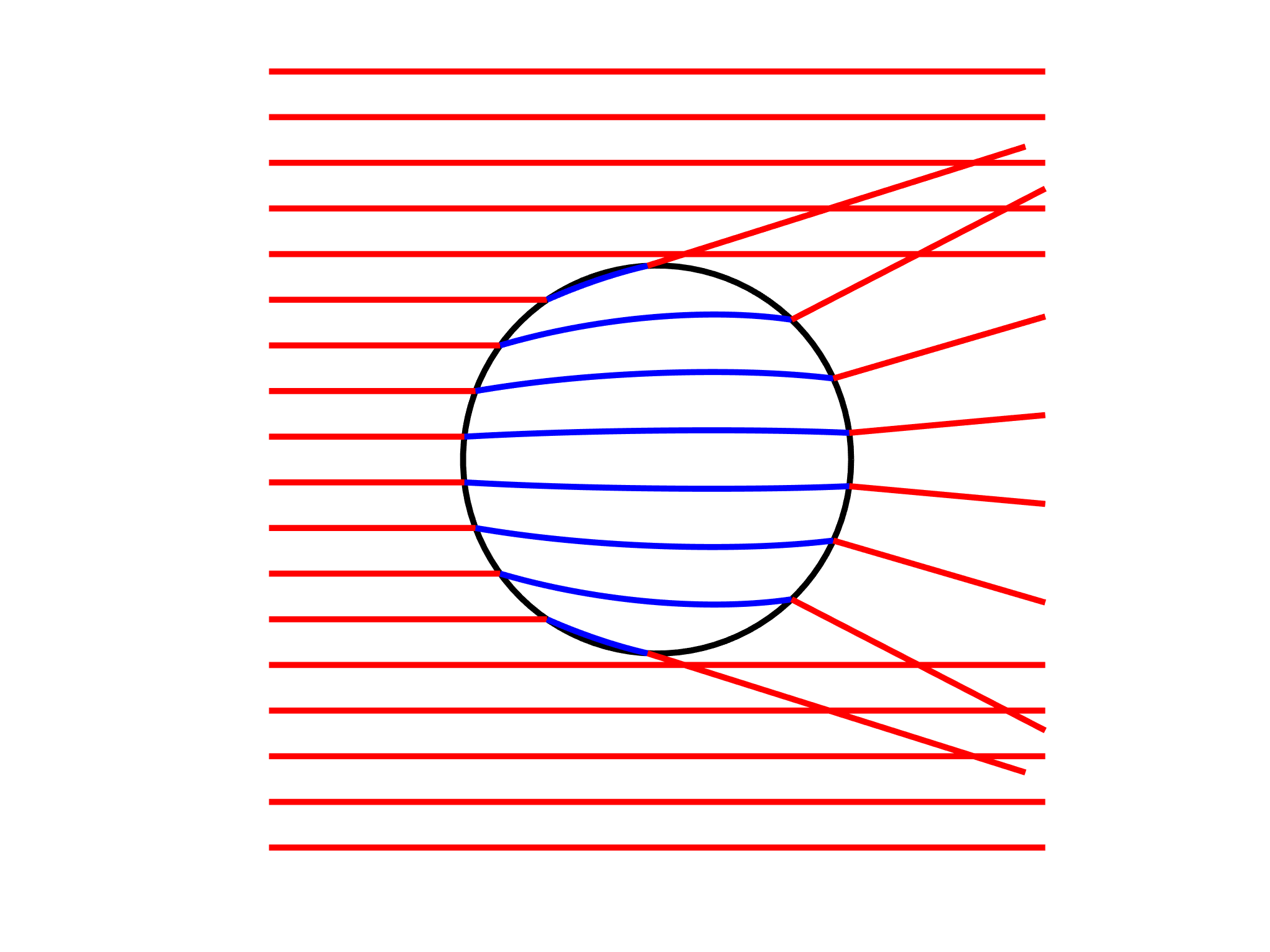}
	\end{subfigure}
	\hfill
	\begin{subfigure}[t]{.48\columnwidth}
		\centering
		\includegraphics[width=\linewidth]{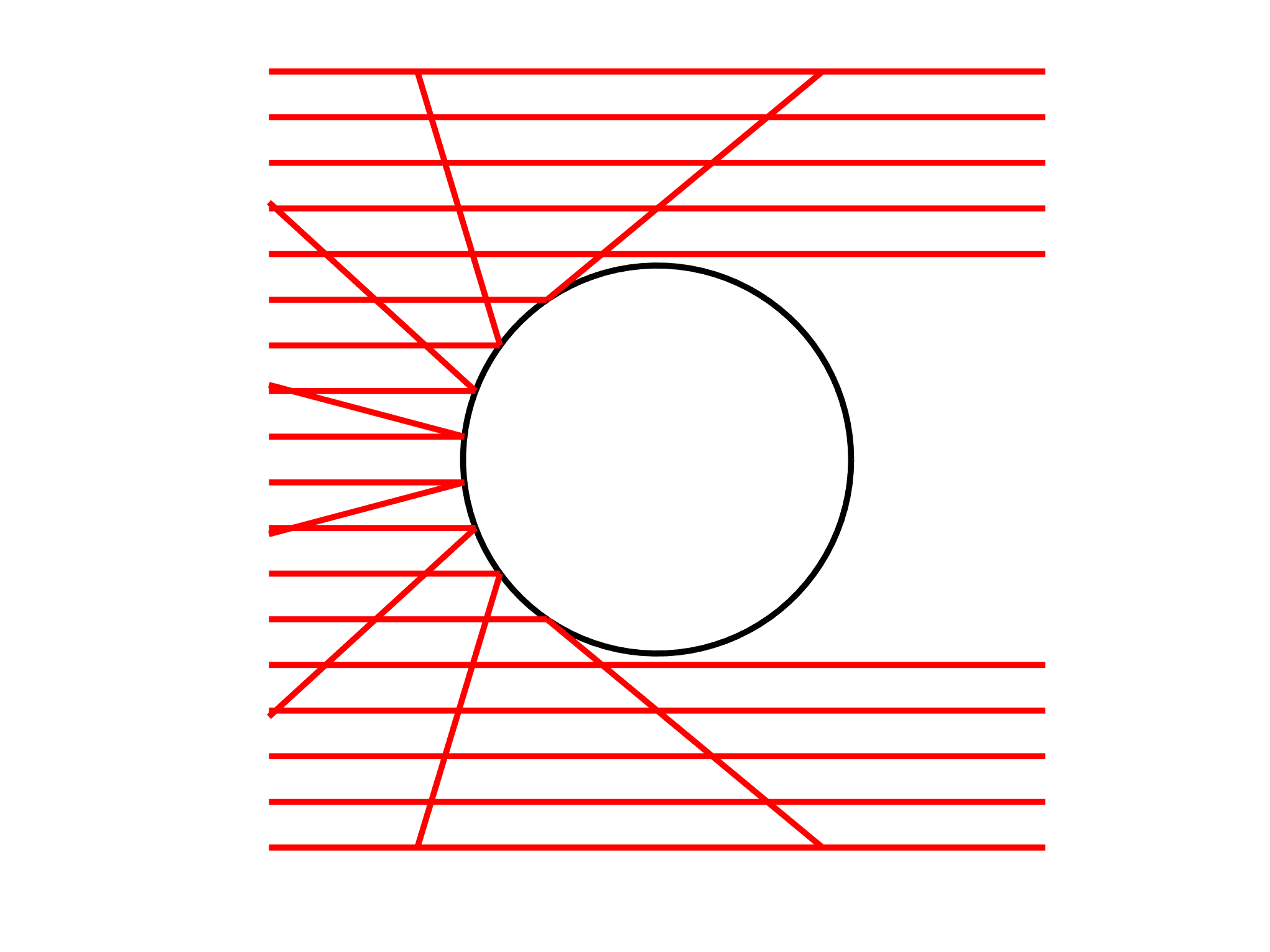}
	\end{subfigure}
	\caption{Left: Trajectories under the ``refraction'' solution to \eqref{eq:refracted} in Example \ref{ex:planar_sphere}. Right: Trajectories under the ``reflection'' solution to \eqref{eq:refracted} in Example \ref{ex:planar_sphere}. Initial conditions are taken to be horizontal; $p_x(0)=1$ and $p_y(0)=0$.
	}
	\label{fig:circle_refraction}
\end{figure}
\section{Hybrid-invariant differential forms}\label{sec:h_diff_forms}
The fundamental goal of this work is to answer the following question: if $\alpha\in\Omega^m(M)$ is a differential form and $\varphi_t^\mathcal{H}$ is the flow of a hybrid system, does $\left(\varphi_t^\mathcal{H}\right)^*\alpha = \alpha$? As this requires differentiability, we will tacitly assume that $\Delta(S)\subset M$ is a smooth embedded submanifold. Also, the results in this section do not require that the hybrid system be mechanical; this is reserved for the following section.

It would seem natural to want $\mathcal{L}_X\alpha = 0$ and $\Delta^*\alpha = \alpha$. However, this does not make sense. The impact map $\Delta$ and the form $\alpha$ do not have the same domains; $\Delta$ is a function on $S$ while $\alpha$ is a form on $M$. This leads to the idea of the \textit{augmented differential}, cf. \cite{fslip}.
\begin{definition}
	Let $\mathcal{H}=(M,S,X,\Delta)$ be a hybrid system and $x\in S$. Then the {linear} map $\Delta_*^X:T_xM\to T_{\Delta(x)}M$ is called the augmented differential where
	\begin{equation*}
		\begin{split}
			\Delta_*^X\cdot u &= \Delta_*\cdot u,\quad u\in T_xS\subset T_xM, \\
			\Delta_*^X\cdot X(x) &= X(\Delta(x)).
		\end{split}
	\end{equation*}
\end{definition}

\begin{remark}
	In order for $\Delta_*^X$ to be defined at a point $x\in S$, $X(x)\not\in T_xS$. Moreover, for $\Delta_*^X$ to be invertible, $X(\Delta(x))\not\in T_{\Delta(x)}\Delta(S)$ (in addition to $\Delta_*:T_xS\to T_{\Delta(x)}\Delta(S)$ being invertible).
\end{remark}
\begin{theorem}\label{th:hybrid_inv_forms}
	Let $\mathcal{H} = (M,S,X,\Delta)$ be a smooth hybrid system with hybrid flow $\varphi_t^\mathcal{H}$. For a given $\alpha\in\Omega^m(M)$, we have $\left(\varphi_t^\mathcal{H}\right)^*\alpha = \alpha$ if and only if $\mathcal{L}_X\alpha = 0$ and
	\begin{equation*}
		\alpha_{\Delta(x)}\left( \Delta_*^X\cdot v_1,\ldots, \Delta_*^X\cdot v_m \right) = \alpha_x\left(v_1,\ldots,v_m\right).
	\end{equation*}
\end{theorem}
\begin{proof} For simplicity of calculations, we will assume that $\alpha$ is a 1-form. 
	Let $x_0\in M$, then the condition that $\left(\varphi_T^\mathcal{H}\right)^*\alpha = \alpha$ means
	\begin{equation*}
		\alpha_{\varphi_T^\mathcal{H}(x_0)}\left( \left(\varphi_T^\mathcal{H}\right)_* v \right) = \alpha_{x_0}\left(v\right).
	\end{equation*}
	Choose $x_0$ and $T$ such that a single impact occurs along the path $\{\varphi_t^\mathcal{H}(x_0):t\in(0,t)\}$ and call this time $t_1$ and location $y_0$, i.e. $y_0 = \varphi_{t_1}(x_0)\in S$. Additionally, call $z_0:=\Delta(y_0)$ and $w_0:=\varphi_{T-t_1}(z_0)=\varphi_T^\mathcal{H}(x_0)$. Because the vector field is transverse to $S$ at $y_0$, we can split up the tangent space at $x_0$ in the following way:
	\begin{equation*}
		T_{x_0}M = T^S_{x_0}M \oplus X(x_0)\cdot\mathbb{R},\quad
		\left(\varphi_{t_1}\right)_*\left(T^S_{x_0}M\right) = T_{y_0}S.
	\end{equation*}
	To compute $\left(\varphi_T^\mathcal{H}\right)_*v$, we split into the cases where $v\in T_{x_0}^SM$ and $v\in X(x_0)\cdot\mathbb{R}$ (which can be taken as $v=X(x_0)$ by linearity). See Figure \ref{fig:impact_differential} for an illustration of this setup.
	
	Let $v\in T_{x_0}^SM$. Therefore, we can choose a curve $\gamma:(-\varepsilon,\varepsilon)\to M$ such that $\varphi_{t_1}\left(\gamma(s)\right)\in S$ for all $s\in(-\varepsilon,\varepsilon)$. Then $\varphi_T^\mathcal{H}(\gamma(s)) = \varphi_{T-t_1}\circ\Delta\circ\varphi_{t_1}(\gamma(s))$. Differentiating this provides
	\begin{equation*}
		\left(\varphi_T^\mathcal{H}\right)_*v = \left(\varphi_{T-t_1}\right)_*\cdot \Delta_* \cdot \left(\varphi_{t_1}\right)_*v.
	\end{equation*}
	Therefore, for $v\in T_{x_0}^SM$, 
	\begin{equation*}
		\alpha_{\varphi_T^\mathcal{H}(x_0)}\left(\left(\varphi_T^\mathcal{H}\right)_*v\right) = \alpha_{\varphi_T^\mathcal{H}(x_0)} \left( \left(\varphi_{T-t_1}\right)_*\cdot \Delta_* \cdot \left(\varphi_{t_1}\right)_*v \right).
	\end{equation*}
	Which, if $\mathcal{L}_X\alpha = 0$, invariance is equivalent to $\alpha_{\Delta(y_0)}\left(\Delta_*\cdot v\right) = \alpha_{y_0}(v)$ for $v\in T_{y_0}S\subset T_{y_0}M$.
	
	Let $v = X(x_0)$. To complete the proof, we need to show that $\left(\varphi_T^\mathcal{H}\right)_*X(x_0) = X(\varphi_T^\mathcal{H}(x_0))$. Let $\gamma:(-\varepsilon,\varepsilon)\to M$ be given by $\gamma(t) = \varphi_t(x_0)$ such that $\varepsilon<t_1$ (so $\gamma$ is also the hybrid flow). Then we have
	\begin{equation*}
		\begin{split}
			\left(\varphi_T^\mathcal{H}\right)_* v &= \left.\frac{d}{dt}\right|_{t=0} \varphi_T^\mathcal{H}\circ\varphi_t^\mathcal{H}(x_0) \\
			&= \left.\frac{d}{dt}\right|_{t=0} \varphi_{T+t}^\mathcal{H}(x_0) \\
			&= X(w_0),
		\end{split}
	\end{equation*}
	which completes the proof.
\end{proof}
\begin{figure}
	\centering
	\begin{tikzpicture}[scale=1.0]
		\coordinate (x0) at (0,0);
		\coordinate (y0) at (3,2);
		\draw (x0) to [out=0,in=-190] (y0);
		\draw (0,5) to [out=0,in=50] (y0) to [out=-180+50,in=-150] (4,-1);
		\node[right] at (4,-1) {$S$};
		\coordinate (z0) at (5,3);
		\draw (6,5) to [out=-90,in=80] (z0) to [out=-180+80,in=90] (7,-1);
		\node[right] at (7,-1) {$\tilde{S}$};
		\coordinate (w0) at (8,2);
		\draw (z0) to [out=-10,in=-160] (w0);
		\draw[red] (2,0.8082) -- (4,3.1918);
		\draw[red] (4.75,1.5822) -- (5.25,4.4178);
		\draw[red] (-0.2,1.5) -- (0.2,-1.5);
		\draw[red] (7.4,1) -- (8.6,3);
		\draw[->,blue] (x0) -- (1,0);
		\draw[->,blue] (y0) -- (3.5,1.9118);
		\draw[->,blue] (z0) -- (5.75,2.8678);
		\draw[->,blue] (w0) -- (9,2.3640);
		\draw [fill] (x0) circle [radius=0.05];
		\draw [fill] (y0) circle [radius=0.05];
		\draw [fill] (z0) circle [radius=0.05];
		\draw [fill] (w0) circle [radius=0.05];
		\node [above] at (4,3.1918) {\textcolor{red}{$T_{y_0}S$}};
		\node [above left] at (y0) {$y_0$};
		\node [below] at (3.5,1.9118) {\textcolor{blue}{$X(y_0)$}};
		\node [above right] at (-0.2,1.5) {\textcolor{red}{$T_{x_0}^SM$}};
		\node [below left] at (x0) {$x_0$};
		\node [below] at (1,0) {\textcolor{blue}{$X(x_0)$}};
		\node [below left] at (5.25,4.4178) {\textcolor{red}{$T_{z_0}\tilde{S}$}};
		\node [left] at (z0) {$z_0$};
		\node [above] at (5.75,2.8678) {\textcolor{blue}{$X(z_0)$}};
		\node [above] at (8.6,3) {\textcolor{red}{$\left(\varphi_{T-t_1}\right)_*\left(T_{z_0}\tilde{S}\right)$}};
		\node [below right] at (w0) {$w_0$};
		\node [below right] at (9,2.3640) {\textcolor{blue}{$X(w_0)$}};
	\end{tikzpicture}
	\caption{Diagram for the proof of Theorem \ref{th:hybrid_inv_forms}. The set $\tilde{S}:=\Delta(S)$ and $T_{x_0}^SM:=\left(\varphi_{-t_1}\right)_*T_{y_0}S$.}
	\label{fig:impact_differential}
\end{figure}
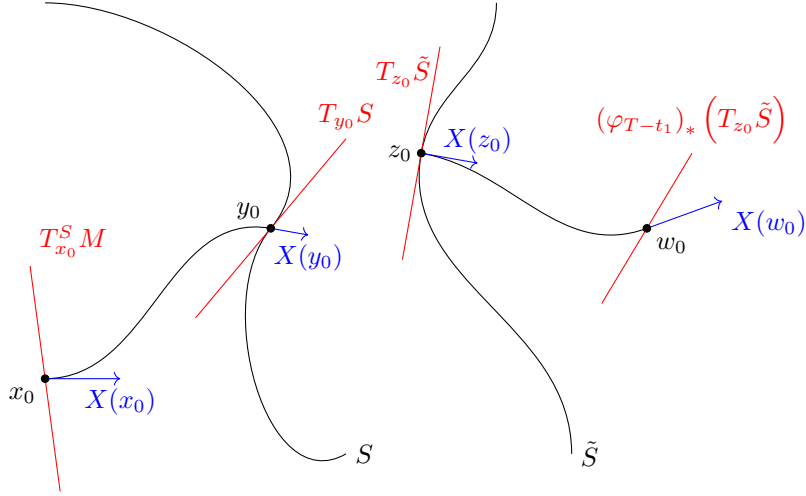
\begin{definition}
	A differential form $\alpha$ is called hybrid-invariant if $\left(\varphi_t^\mathcal{H}\right)^*\alpha=\alpha$. Let the set $\mathscr{A}_{\mathcal{H}}\subset \Omega(M)$ be all the hybrid-invariant forms.
\end{definition}
Computing the augmented differential is tedious in practice as it is a nonorthogonal projection. We can instead decompose $\Delta_*^X$ to avoid its computation. This leads to the following criteria.
\begin{theorem}\label{th:energy_specular}
	A differential form $\alpha$ is hybrid-invariant if and only if $\mathcal{L}_X\alpha = 0$ and
	\begin{gather}
		\Delta^*\iota_{\tilde{S}}^*i_X\alpha = \iota_S^*i_X\alpha, \label{eq:energy_form} \\
		\Delta^*\iota_{\tilde{S}}^*\alpha = \iota_S^*\alpha, \label{eq:specular_form}
	\end{gather}
	where $\tilde{S}=\Delta(S)$ and $\iota_S:S\hookrightarrow M$, $\iota_{\tilde{S}}:\tilde{S}\hookrightarrow M$ are the inclusion maps.
\end{theorem}
\begin{proof}
	Suppose that $\alpha\in\Omega^2(M)$ (the proof is almost identical for forms of different degrees). For $y_0\in S$ and $u,v\in T_{y_0}M$, decompose the vectors in the following way:
	\begin{equation*}
		\begin{split}
			u &= a\cdot X(y_0) + \tilde{u}, \quad \tilde{u}\in T_{y_0}S\\
			v &= b\cdot X(y_0) + \tilde{v}, \quad \tilde{v}\in T_{y_0}S.
		\end{split}
	\end{equation*}
	Under this decomposition, the augmented differential is
	\begin{equation*}
		\begin{split}
			\Delta_*^X\cdot u &= a\cdot X(z_0) + \Delta_*\tilde{u} \\
			\Delta_*^X\cdot v &= b\cdot X(z_0) + \Delta_*\tilde{v}.
		\end{split}
	\end{equation*}
	Therefore, according to Theorem \ref{th:hybrid_inv_forms} invariance is equivalent to
	\begin{equation*}
		\alpha_{y_0}(a\cdot X(y_0) + \tilde{u},b\cdot X(y_0) + \tilde{v}) = 
		\alpha_{z_0}(a\cdot X(z_0) + \Delta_*\tilde{u},b\cdot X(z_0) + \Delta_*\tilde{v}).
	\end{equation*}
	Using the bi-linearity of $\alpha$ results in
	\begin{equation*}
		\begin{split}
			0 &= a\cdot\alpha(X(y_0),\tilde{v}) - a\cdot\alpha(X(z_0),\Delta_*\tilde{v}) \\
			&\quad + b\cdot\alpha(\tilde{u},X(y_0))-b\cdot\alpha(\Delta_*\tilde{u},X(z_0)) \\
			&\quad + \alpha(\tilde{u},\tilde{v}) - \alpha(\Delta_*\tilde{u},\Delta_*\tilde{v}).
		\end{split}
	\end{equation*}
	This condition is equivalent to
	\begin{equation*}
		\begin{split}
			i_{X(y_0)}\alpha(\tilde{v}) &= i_{X(z_0)}\alpha(\Delta_*\tilde{v}) \\
			\alpha(\tilde{u},\tilde{v}) &= \alpha(\Delta_*\tilde{u},\Delta_*\tilde{v})
		\end{split}
	\end{equation*}
	for all $\tilde{u},\tilde{v}\in T_{y_0}S$. These are equivalent to \eqref{eq:energy_form} and \eqref{eq:specular_form}.
\end{proof}
It is interesting to point out that hybrid-invariance requires two additional conditions, not one. For reasons that will be apparent in \S\ref{sec:vol_mech}, condition \eqref{eq:energy_form} will be called the \textit{energy condition} while \eqref{eq:specular_form} will be called the \textit{specular condition}. 

A benefit of using the specular and energy conditions is that we can describe some algebraic properties of the space $\mathscr{A}_\mathcal{H}$.
\begin{corollary}\label{cor:inv_algebra}
	The set of hybrid-invariant forms $\mathscr{A}_\mathcal{H}\subset\Omega(M)$ is a $\wedge$-subalgebra closed under $d$ and $i_X$.
\end{corollary}
\begin{proof}
	If we denote $\mathscr{A}:=\left\{\alpha\in\Omega(M):\mathcal{L}_X\alpha=0\right\}$, then it is already known that $\mathscr{A}\subset\Omega(M)$ is a $\wedge$-subalgebra closed under $d$ and $i_X$ (see Corollary 3.4.5 in \cite{abraham2008foundations}). Therefore, in order to prove the theorem, it suffices only to check \eqref{eq:energy_form} and \eqref{eq:specular_form}. Let $\alpha,\beta\in\mathscr{A}_\mathcal{H}$. We only need to check that $d\alpha$, $i_X\alpha$, and $\alpha\wedge\beta$ obey \eqref{eq:energy_form} and \eqref{eq:specular_form}.
	
	Consider $i_X\alpha$. This satisfies \eqref{eq:specular_form} because $\alpha$ satisfies \eqref{eq:energy_form} and \eqref{eq:energy_form} is satisfied because $i_Xi_X\alpha=0$.
	
	Consider $d\alpha$. Condition \eqref{eq:specular_form} follows from the fact that $d$ commutes with pullbacks:
	\begin{equation*}
		\Delta^*\iota_{\tilde{S}}^*d\alpha = d\left( \Delta^*\iota_{\tilde{S}}^*\alpha\right) 
		= d\left(\iota_S^*\alpha\right) 
		= \iota_S^*d\alpha.
	\end{equation*}
	Condition \eqref{eq:energy_form} requires Cartan's magic formula ($\mathcal{L}_X = di_X + i_Xd$):
	\begin{equation*}
		\begin{split}
			\Delta^*\iota_{\tilde{S}}^*i_Xd\alpha &= \Delta^*\iota_{\tilde{S}}^*\left( \cancel{\mathcal{L}_X\alpha} - di_X\alpha\right)= -\Delta^*\iota_{\tilde{S}}^*di_X\alpha  \\
			&= -d\left(\Delta^*\iota_{\tilde{S}}^*i_X\alpha\right) = 
			-d\left(\iota_S^*i_X\alpha\right) \\
			&= -\iota_S^*di_X\alpha = \iota_S^*i_Xd\alpha.
		\end{split}
	\end{equation*}
	
	Finally, consider $\alpha\wedge\beta$. The condition \eqref{eq:specular_form} holds because pullbacks distribute over the wedge product: $f^*(\alpha\wedge\beta) = f^*\alpha\wedge f^*\beta$. For condition \eqref{eq:energy_form}, we see that
	\begin{equation*}
		\begin{split}
			\Delta^*\iota_{\tilde{S}}^*i_X\left(\alpha\wedge\beta\right) &= \Delta^*\iota_{\tilde{S}}^*\left(i_X\alpha\wedge\beta - \alpha\wedge i_X\beta\right) \\
			&= \left(\Delta^*\iota_{\tilde{S}}^*i_X\alpha\right)\wedge \left(\Delta^*\iota_{\tilde{S}}^*\beta\right) - \left(\Delta^*\iota_{\tilde{S}}^*\alpha\right)\wedge\left(\Delta^*\iota_{\tilde{S}}^*i_X\beta\right) \\
			&= \left(\iota_S^*i_X\alpha\right)\wedge\left(\iota_S^*\beta\right) - \left(\iota_S^*\alpha\right)\wedge\left(\iota_S^*i_X\beta\right) \\
			&= \iota_S^*i_X\left(\alpha\wedge\beta\right),
		\end{split}
	\end{equation*}
	which completes the proof.
\end{proof}
\begin{remark}
	We point out how the conditions for invariant forms manifest for functions, i.e. 0-forms. The energy condition becomes trivial as $i_Xf = 0$ while the specular condition reads that $f\circ\Delta = f$. This is in agreement with the common notion that a function is hybrid-invariant if it is preserved across both the smooth flow ($\mathcal{L}_Xf=0$) and across impacts ($f\circ\Delta = f$). On the other hand, if $\mu\in\Omega^n(M)$ is a volume form, then the specular condition becomes trivial and we are only interested in the energy condition. We investigate this below.
\end{remark}
\subsection{Hybrid-invariant volumes}
Suppose that $\dim M = n$, then a volume form is given by a non-vanishing form $\mu\in\Omega^n(M)$. This volume form is hybrid-invariant if $\mu\in \mathscr{A}_\mathcal{H}$, i.e. we wish to determine whether or not $\mathscr{A}_\mathcal{H}\cap\Omega^n(M)$ is empty. To test for hybrid-invariant volume forms, we have the following refinement of Theorem \ref{th:energy_specular}.
\begin{theorem}
	Let $\mathcal{H} = (M,S,X,\Delta)$ be a smooth hybrid system. A volume form $\mu\in\Omega^n(M)$ is hybrid-invariant if and only if
	\begin{equation*}
		\begin{cases}
			\mathrm{div}_\mu(X) = 0, \\
			\Delta^*\iota_{\tilde{S}}^*i_X\mu = \iota_S^*i_X\mu.
		\end{cases}
	\end{equation*}
\end{theorem}
\begin{proof}
	These conditions match those of Theorem \ref{th:energy_specular} with the exception of the specular term, \eqref{eq:specular_form}. This is because $\mu$ is an $n$-form and $\dim S = n-1$ and therefore $\iota_S^*\mu\equiv 0$ so the specular term is trivially satisfied.
\end{proof}
Recall that the dimension of $\Omega^n(M)$ over $C^\infty(M)$ is one which means that for a given volume-form $\mu\in \Omega^n(M)$, then any other form $\nu\in\Omega^n(M)$ can be written as $\nu = f\mu$ for some $f\in C^\infty(M)$. This fact can be useful in finding hybrid-invariant volumes in the following way: suppose that $\mu\in\Omega^n(M)$ but $\mu\not\in\mathscr{A}_\mathcal{H}$. What conditions can be placed on a function $f\in C^\infty(M)$ to guarantee that $f\mu\in\mathscr{A}_\mathcal{H}$?
\begin{definition}
	Let $\mathcal{H} = (M,S,X,\Delta)$ be a smooth hybrid system and let $\mu\in\Omega^n(M)$ be a volume form. The unique function $\mathcal{J}_\mu(\Delta)\in C^\infty(S)$ such that
	\begin{equation}\label{eq:hybrid_jacobian}
		\Delta^*\iota_{\tilde{S}}^*i_X\mu = \mathcal{J}_\mu(\Delta)\cdot \iota_S^* i_X\mu,
	\end{equation}
	is called the \textit{hybrid Jacobian} of $\Delta$ (with respect to $\mu$).
\end{definition}
The hybrid Jacobian allows for conditions on $f$ to be described via a ``hybrid cohomology equation.''
\begin{proposition}\label{prop:hybrid_coho}
	For a smooth hybrid system $\mathcal{H} = (M,S,X,\Delta)$, there exists a smooth hybrid-invariant volume, $f\mu$, if there exists a smooth function $g\in C^\infty(M)$ such that
	\begin{equation}\label{eq:hybrid_cohom}
		\begin{split}
			dg(X) &= -\mathrm{div}_\mu(X) \\
			g\circ\Delta - g|_S &= -\ln\left(\mathcal{J}_\mu(\Delta)\right).
		\end{split}
	\end{equation}
	Then the density is (up to a multiplicative constant) $f = \exp(g)$.
\end{proposition}
\section{Invariant volumes in mechanical hybrid systems}\label{sec:vol_mech}
It turns out that unconstrained impact systems remain volume preserving while the problem in nonholonomic systems is much more difficult to answer. It is already known that unconstrained impact systems are symplectic (and hence volume-preserving), \cite{impactSymp}. We prove this below using Theorem \ref{th:energy_specular}.
\begin{proposition}\label{prop:unconst_hamilt_volume}
	Let $\mathcal{H} = (T^*Q,S^*,X_H,\Delta)$ be an unconstrained Hamiltonian impact system. Then $\omega^n\in \mathscr{A}_\mathcal{H}$.
\end{proposition}
\begin{proof}
	First off, $\mathcal{L}_{X_H}\omega = 0$ by Liouville's theorem. To show that $\omega$ is hybrid-invariant, we need to show that it satisfies both the energy and specular conditions. The energy conditions follows from conservation of energy:
	\begin{equation*}
		\Delta^*\iota_{\tilde{S}}^*i_{X_H}\omega = \Delta^*\iota_{\tilde{S}}^* dH = \iota_S^*dH = \iota_S^*i_{X_H}\omega.
	\end{equation*}
	To show the specular condition, choose coordinates such that $x^n = h$. In these coordinates,
	\begin{equation*}
		\iota_{\tilde{S}}^*\omega = dx^1\wedge dp_1 + \ldots + dx^{n-1}\wedge dp_{n-1}. 
	\end{equation*}
	According to the first corner condition, \eqref{eq:lagrange_impact}, the impact is the identity on every coordinate with the exception of $p_n$, but this is invisible to the restricted form $\iota_{\tilde{S}}^*\omega$. Therefore, it is preserved across impacts.
	
	We have proven that $\omega\in\mathscr{A}_\mathcal{H}$ and $\omega^n\in\mathscr{A}_\mathcal{H}$ follows from Corollary \ref{cor:inv_algebra}.
\end{proof}
This proposition demonstrates the naming convention of both the energy and specular conditions: the energy condition comes from conservation of energy and the specular condition comes from the impact being specular (the two corner conditions in \eqref{eq:lagrange_impact}).
\subsection{Invariant volumes in nonholonomic hybrid systems}
Proposition \ref{prop:unconst_hamilt_volume} shows that unconstrained hybrid mechanical systems automatically preserve volume (independent of the choice of $S$). In the language of Proposition \ref{prop:hybrid_coho}, $\mathrm{div}_{\omega^n}(X_H)=0$ and $\mathcal{J}_{\omega^n}(\Delta)=1$ which admits a trivial solution to the hybrid cohomology equation. On the contrary, it is no longer generally true that $\mathrm{div}_{\mu_\mathscr{C}}(X_H^\mathcal{D})=0$ where $\mu_\mathscr{C}$ is the nonholonomic volume form (cf. \cite{nhvolume}) and $X_H^\mathcal{D}$ is the nonholonomic vector field. Likewise, it is no longer obvious whether or not $\mathcal{J}_{\mu_\mathscr{C}}(\Delta^\mathcal{D})=1$. In what follows, we compute the hybrid Jacobian and show that it, indeed, does equal 1.
\subsubsection{The hybrid Jacobian}
In order to find invariant volumes for nonholonomic hybrid systems, we need to be able to compute $\mathcal{J}_{\mu_\mathscr{C}}(\Delta^\mathcal{D})$. In order to calculate this, we will first compute the ``global'' version $\mathcal{J}_{\omega^n}(\Delta^\mathscr{C})$ and restrict to $\mathcal{D}^*$ (recall Remark \ref{rmk:global_impact}). The following computation will make use of the nonholonomic volume form (which was used in Theorem \ref{th:NHVol}) which is defined as below.
\begin{definition}
	Let $\mathscr{C} = \left\{\eta^1,\ldots,\eta^m\right\}$ be a collection of constraints realizing the constraint submanifold $\mathcal{D}\subset TQ$. The nonholonomic volume form, $\mu_\mathscr{C}$, is a volume form on $\mathcal{D}^*$ which is constructed as follows: Let $W^\alpha = \mathbb{F}L^{-1}\eta^\alpha$ be the corresponding vector fields and $P(W^\alpha)$ be their momentum functions. Define the $m$-form
	\begin{equation*}
		\sigma_{\mathscr{C}} := dP(W^1)\wedge \ldots \wedge dP(W^m).
	\end{equation*}
	Then the nonholonomic volume form is given by
	\begin{equation*}
		\mu_{\mathscr{C}} = \iota^*\varepsilon, \quad \sigma_{\mathscr{C}}\wedge\varepsilon = \omega^n,
	\end{equation*}
	where $\iota:\mathcal{D}^*\hookrightarrow T^*Q$ is the inclusion.
\end{definition}
It is shown in \cite{nhvolume} that the nonholonomic form, as defined above, is a unique volume form on $\mathcal{D}^*$. Computations with this volume form are difficult as it requires utilizing local coordinates. The following lemma offers a computational trick to sidestep this issue.
\begin{lemma}
	Let $\mathcal{H}^\mathscr{C}=(T^*Q,S^*,\Xi_H^\mathscr{C},\Delta^\mathscr{C})$ be the global version of the nonholonomic hybrid system $\mathcal{H}^\mathcal{D}=(\mathcal{D}^*,S^*_\mathcal{D},X_H^\mathcal{D},\Delta^\mathcal{D})$ and let $\mu_{\mathscr{C}}$ be the nonholonomic volume form. Then
	\begin{equation*}
		\mathcal{J}_{\omega^n}(\Delta^\mathscr{C})|_{\mathcal{D}^*} =  \mathcal{J}_{\mu_{\mathscr{C}}}(\Delta^\mathcal{D}).
	\end{equation*}
\end{lemma}
\begin{proof}
	A computation yields:
	\begin{equation*}
		\begin{split}
			\Delta^{\mathscr{C}*}\iota_{\tilde{S}}^*i_{\Xi_H^\mathscr{C}}\omega^n &=
			\Delta^{\mathscr{C}*}\iota_{\tilde{S}}^*i_{\Xi_H^\mathscr{C}}\left(\sigma_{\mathscr{C}}\wedge\varepsilon\right) \\
			&= \Delta^{\mathscr{C}*}\iota_{\tilde{S}}^*\left( i_{\Xi_H^\mathscr{C}}\sigma_\mathscr{C}\wedge\varepsilon + (-1)^m\sigma_{\mathscr{C}}\wedge i_{\Xi_{H}^\mathscr{C}}\varepsilon \right) \\
			&= (-1)^m\left(\Delta^{\mathscr{C}*}\iota_{\tilde{S}}^*\sigma_{\mathscr{C}}\right)\wedge
			\left(\Delta^{\mathscr{C}*}\iota_{\tilde{S}}^*i_{\Xi_H^\mathscr{C}}\varepsilon\right) \\
			&= (-1)^m \left(\iota_S^*\sigma_{\mathscr{C}}\right)\wedge \left(\Delta^{\mathscr{C}*}\iota_{\tilde{S}}^*i_{\Xi_H^\mathscr{C}}\varepsilon\right),
		\end{split}
	\end{equation*}
	which uses the fact that the constraints are preserved under the flow. That is, $i_{\Xi_H^\mathscr{C}}\sigma_{\mathscr{C}}=0$ and $\Delta^{\mathscr{C}*}\iota_{\tilde{S}}^*\sigma_{\mathscr{C}} = \iota_S^*\sigma_{\mathscr{C}}$. The right side of \eqref{eq:hybrid_jacobian} produces
	\begin{equation*}
		\begin{split}
			\mathcal{J}_{\omega^n}(\Delta^\mathscr{C})\cdot \iota_S^*i_{\Xi_H^\mathscr{C}}\left(\sigma_{\mathscr{C}}\wedge\varepsilon\right) &=
			\mathcal{J}_{\omega^n}(\Delta^\mathscr{C})\cdot\iota_S^*\left(i_{\Xi_H^\mathscr{C}}\sigma_\mathscr{C}\wedge\varepsilon + (-1)^m\sigma_{\mathscr{C}}\wedge i_{\Xi_H^\mathscr{C}}\varepsilon\right) \\
			&= (-1)^m\mathcal{J}_{\omega^n}(\Delta^\mathscr{C})\cdot\left(\iota_S^*\sigma_\mathscr{C}\right)\wedge \left( \iota_S^*i_{\Xi_H^\mathscr{C}}\varepsilon\right).
		\end{split} 
	\end{equation*}
	Combining both of the above gives
	\begin{equation*}
		\left(\iota_S^*\sigma_{\mathscr{C}}\right)\wedge \left(\Delta^{\mathscr{C}*}\iota_{\tilde{S}}^*i_{\Xi_H^\mathscr{C}}\varepsilon\right) = 
		\mathcal{J}_{\omega^n}(\Delta^\mathscr{C})\cdot\left(\iota_S^*\sigma_\mathscr{C}\right)\wedge \left( \iota_S^*i_{\Xi_H^\mathscr{C}}\varepsilon\right).
	\end{equation*}
	The result follows from restricting to $\mathcal{D}^*$.
\end{proof}
Therefore, to calculate $\mathcal{J}_{\omega^n}(\Delta^\mathscr{C})$, we need to understand $\Delta^{\mathscr{C}*}\iota_{\tilde{S}}^*i_{\Xi_H^\mathscr{C}}\omega^n$. Expanding gives
\begin{equation*}
	\Delta^{\mathscr{C}*}\iota_{\tilde{S}}^*i_{\Xi_H^\mathscr{C}}\omega^n = \Delta^{\mathscr{C}*}\iota_{\tilde{S}}^*\left(n\cdot \nu_H^\mathscr{C}\wedge\omega^{n-1}\right) = 
	n\cdot \left(\Delta^{\mathscr{C}*}\iota_{\tilde{S}}^*\nu_H^\mathscr{C}\right)\wedge
	\left(\Delta^{\mathscr{C}*}\iota_{\tilde{S}}^*\omega\right)^{n-1}.
\end{equation*}
Therefore, the hybrid Jacobian is determined by how much the nonholonomic 1-form and the symplectic form fail the specular condition \eqref{eq:specular_form}. We next present a helpful computational lemma which will be useful for computing the above.

\begin{lemma}\label{le:wedge_power}
	Let $x^1,\ldots,x^{n-1},p_1,\ldots,p_n$ be local coordinates and let $A=\left(\alpha_j^i\right)$ be an $(n-1)\times n$ matrix. Then
	\begin{equation}\label{eq:wedge_power}
		\begin{split}
			\left( \sum_{j=1}^{n-1} \sum_{i=1}^n \, \alpha_j^i \cdot dx^j\wedge dp_i\right)^{n-1}
			= (-1)^{\floor{\frac{n-1}{2}}}\cdot(n-1)!\cdot\sum_{k=1}^{n} \, \det(A_k) \cdot \Omega_k, \\
			\Omega_k := dx^1\wedge\ldots\wedge dx^{n-1} \wedge dp_1 \wedge
			\ldots \wedge \widehat{dp_k}\wedge \ldots dp_n,
		\end{split}
	\end{equation}
	where $A_k$ is the $(n-1)\times (n-1)$ matrix obtained from deleting the $k^{th}$-column from $A$ and the caret $\widehat{dp_k}$ means that $dp_k$ is omitted from the wedge product.
\end{lemma}
\begin{proof}
	Recall the multinomial theorem which states that 
	\begin{equation}\label{eq:multinomial}
		\left( \sum_{j=1}^{n-1} \sum_{i=1}^n \, \alpha_j^i \cdot dx^j\wedge dp_i \right)^{n-1} = 
		\sum_{S(d_j^i)=n-1}{n-1\choose d_1^1,\ldots, d_{n-1}^n} \prod_{ij} \left( \alpha_j^i\cdot dx^j\wedge dp_i\right)^{d_j^i},
	\end{equation}
	where
	\begin{equation*}
		S(d_j^i) = \sum_{j=1}^{n-1} \sum_{i=1}^n \, d_j^i, \quad
		{n-1\choose d_1^1,\ldots, d_{n-1}^n} = \frac{n!}{d_1^1!\cdot\ldots\cdot d_{n-1}^n!}.
	\end{equation*}
	Notice that for any $i,j$ we have $\left(\alpha_j^i\cdot dx^j\wedge dp_i\right)^2=0$. This implies that the only nonzero terms in \eqref{eq:multinomial} have $d_j^i\in\{0,1\}$. This simplifies \eqref{eq:multinomial} to
	\begin{equation}\label{eq:multinomial_bool}
		\left( \sum_{j=1}^{n-1} \sum_{i=1}^n \, \alpha_j^i \cdot dx^j\wedge dp_i \right)^{n-1} = 
		(n-1)!\cdot\sum_{\substack{S(d_j^i)=n-1 \\ d_j^i\in\{0,1\}}}\prod_{ij}\, (\alpha_j^i\cdot dx^j\wedge dp_i)^{d_j^i}.
	\end{equation}
	In order to evaluate \eqref{eq:multinomial_bool}, we wish to understand the structure of the matrices $\textbf{d}= \left(d_j^i\right)$ that contribute a nonzero term. In addition to having coefficients in $\{0,1\}$, they also have the following property: if $d_j^i=1$, then $d_j^k=d_k^i=0$ for all $k$. This is due to the fact that $\left(dx^j\wedge dp_i\right)\wedge \left(dx^\ell\wedge dp_k\right)=0$ whenever $j=\ell$ or $i=k$. In other words, the matrix $\textbf{d}$ must have a single nonzero entry in each row and at most one in each column. The matrix $\textbf{d}$ is then given as a column permutation of the matrix
	\begin{equation*}
		\textbf{d}_0 = \begin{bmatrix}
			1 & 0 & \cdots & 0 & 0 \\
			0 & 1 & \cdots & 0 & 0\\
			\vdots & \vdots & \ddots & \vdots & \vdots \\
			0 & 0 & \cdots & 1 & 0
		\end{bmatrix}.
	\end{equation*}
	Let $\mathscr{D}$ be the set of all such matrices and partition it as $\mathscr{D}=\sqcup_{k=1}^n \mathscr{D}_k$ where $\textbf{d}\in\mathscr{D}_k$ if its $k^{th}$-column is identically zero. The expression \eqref{eq:multinomial_bool} becomes
	\begin{equation}\label{eq:power_partition}
		\left( \sum_{j=1}^{n-1} \sum_{i=1}^n \, \alpha_j^i \cdot dx^j\wedge dp_i \right)^{n-1} = (n-1)! \cdot \sum_{k=1}^n \, \sum_{(d_j^i)\in\mathscr{D}_k} \, \prod_{i,j:d_j^i\ne 0} \alpha_j^i\cdot dx^j\wedge dp_i.
	\end{equation}
	By deleting the $k^{th}$-row from $\mathscr{D}_k$ there is a natural isomorphism $S_{n-1}\to \mathscr{D}_k$, where $S_{n-1}$ is the symmetric group of $n-1$ elements. A matrix $(d_j^i)\in\mathscr{D}_k$ if and only if there exists $\sigma\in S_{n-1}$ such that $d_j^i=1$ if and only if $\sigma_k(j)=i$ where
	\begin{equation*}
		\sigma_k(j) = \begin{cases}
			\sigma(j), & \sigma(j) < k \\
			\sigma(j) + 1, & \sigma(j) > k.
		\end{cases}
	\end{equation*}
	(This modified permutation keeps track of the $k^{th}$-column deletion.) Before we finish the calculation of \eqref{eq:wedge_power}, we notice that (see 3.1.3 in \cite{abraham2008foundations})
	\begin{equation*}
		\begin{split}
			\prod_{j=1}^{n-1} dx^j\wedge dp_{\sigma_k(j)} &= 
			(-1)^{\floor{\frac{n-1}{2}}}\cdot dx^1\wedge\ldots\wedge dx^{n-1}\wedge dp_{\sigma_k(1)} \wedge \ldots \wedge dp_{\sigma_k(n-1)} \\
			&= (-1)^{\floor{\frac{n-1}{2}}}\cdot
			\mathrm{sgn}(\sigma)\cdot \underbrace{dx^1\wedge\ldots\wedge dx^{n-1}\wedge dp_1\wedge\ldots\wedge \widehat{dp_k}\wedge\ldots dp_n}_{=:\Omega_k}.
		\end{split}
	\end{equation*}
	Using this, we see that \eqref{eq:power_partition} becomes
	\begin{equation*}
		\begin{split}
			\left( \sum_{j=1}^{n-1} \sum_{i=1}^n \, \alpha_j^i \cdot dx^j\wedge dp_i \right)^{n-1} &= (-1)^{\floor{\frac{n-1}{2}}}\cdot(n-1)! \cdot \sum_{k=1}^n \, \sum_{\sigma\in S_{n-1}} \,
			\mathrm{sgn}(\sigma) \prod_{j=1}^{n-1} \, \alpha_j^{\sigma_k(j)} \cdot \Omega_k \\
			&= (-1)^{\floor{\frac{n-1}{2}}}\cdot(n-1)! \cdot \sum_{k=1}^n \det(A_k) \cdot \Omega_k,
		\end{split}
	\end{equation*}
	which is precisely \eqref{eq:wedge_power}.
\end{proof}
\begin{proposition}\label{prop:pre_jacobian}
	In local coordinates where $h=x^n$, we have
	\begin{equation*}
		\iota_S^*\left(\nu_H^\mathscr{C}\wedge\omega^{n-1}\right) = (-1)^{\floor{\frac{n-1}{2}}}\cdot(n-1)!\cdot \left(\frac{\partial H}{\partial p_n}\right)\cdot dx^1\wedge\ldots\wedge dx^{n-1}\wedge dp_1\wedge\ldots\wedge dp_n.
	\end{equation*}
\end{proposition}
\begin{proof}
	By Lemma \ref{le:wedge_power} we can compute $(\iota_S^*\omega)^{n-1}$ where $\alpha_j^i = \delta_j^i$. This provides
	\begin{equation*}
		\left(\iota_S^*\omega\right)^{n-1} = (-1)^{\floor{\frac{n-1}{2}}}\cdot(n-1)! \cdot \Omega_n.
	\end{equation*}
	Due to the fact that $\Omega_n$ depends on every $dx^j$ and $dp_i$ with the exception of $dp_n$, the only component of $\nu_H^\mathscr{C}$ that wedges with $(\iota_S^*\omega)^{n-1}$ to produce a nonzero term is the $p_n$ term, i.e.
	\begin{equation*}
		\begin{split}
			\nu_H^\mathscr{C} \wedge \left(\iota_S^*\omega\right)^{n-1} &= \left(\frac{\partial H}{\partial p_n}\cdot dp_n\right)\wedge \left(\iota_S^*\omega\right)^{n-1} \\
			&= (-1)^{\floor{\frac{n-1}{2}}}\cdot(n-1)!\cdot\frac{\partial H}{\partial p_n} \cdot \Omega_n\wedge dp_n,
		\end{split}
	\end{equation*}
	where $\Omega_n\wedge dp_n = dp_n\wedge \Omega_n$ because $\Omega_n$ has even degree.
\end{proof}
\begin{corollary}\label{cor:coord_free_restricted}
	In coordinate-free language, we have
	\begin{equation*}
		\iota_S^*\left(\nu_H^\mathscr{C}\wedge\omega^{n-1}\right) = 
		(-1)^{\floor{\frac{n-1}{2}}}\cdot(n-1)!\cdot \pi_Q^*dh\left(\Xi_H^\mathscr{C}\right) \cdot \Omega_h,
	\end{equation*}
	where $\Omega_h$ is a volume on $S^*$ given by
	\begin{equation*}
		\Omega_h = \iota_S^*\varepsilon, \quad dh\wedge\varepsilon = (-1)^{n-1}\cdot dx^1\wedge\ldots\wedge dx^n\wedge dp_1\wedge\ldots\wedge dp_n.
	\end{equation*}
\end{corollary}
We are now ready to proceed with calculating the hybrid Jacobian.
\begin{theorem}\label{th:nonh_hybrid_jac}
	The hybrid Jacobian is given by 
	\begin{equation}\label{eq:hybrid_nh_jacobian}
		\mathcal{J}_{\omega^n}(\Delta^\mathscr{C}) = \frac{(2\cdot\pi_{Q}^*\pi_{\mathcal{D}}^*dh - \pi_Q^*dh)(\Xi_H^\mathscr{C})}{\pi_Q^*dh(\Xi_H^\mathscr{C})}.
	\end{equation}
	In particular, $\mathcal{J}_{\mu_{\mathscr{C}}}(\Delta^\mathcal{D}) = 1$.
\end{theorem}
\begin{proof}
	We will choose local coordinates such that $h=x^n$ in a manner similar to Proposition \ref{prop:pre_jacobian} and later translate to a coordinate-free language as in Corollary \ref{cor:coord_free_restricted}.
	We will first compute $\left(\Delta^{\mathscr{C}*}\iota_{\tilde{S}}^*\omega\right)^{n-1}$. In coordinates where $h=x^n$, this becomes
	\begin{equation*}
		\left(\Delta^{\mathscr{C}*}\iota_{\tilde{S}}^*\omega\right)^{n-1} = 
		\left( dx^1\wedge d(p_1\circ\Delta^\mathscr{C}) + \ldots + dx^{n-1}\wedge d(p_{n-1}\circ\Delta^\mathscr{C})\right)^{n-1}.
	\end{equation*}
	The map $\Delta^\mathscr{C}$ given by Remark \ref{rmk:global_impact} depends on both $x$ and $p$:
	\begin{equation*}
		d\left(p_j\circ\Delta^\mathscr{C}\right) = \alpha_j^i dp_i + \beta_{ij} dx^i.
	\end{equation*}
	With this notation, we have
	\begin{equation*}
		\begin{split}
			\left(\Delta^{\mathscr{C}*}\iota_{\tilde{S}}^*\omega\right)^{n-1} &=
			\left( \alpha_j^i \cdot dx^j\wedge dp_i + \beta_{ij} \cdot dx^j\wedge dx^i \right)^{n-1} \\
			&= \left(\alpha_j^i\cdot dx^j\wedge dp_i \right)^{n-1},
		\end{split}
	\end{equation*}
	where the $\beta_{ij}\cdot dx^j\wedge dx^i$ terms do not contribute because any piece containing them will necessarily have a repeated term. Therefore if we can determine the coefficients $\alpha_j^i$, Lemma \ref{le:wedge_power} shows how to compute the product. The expression \eqref{eq:global_NH_impact} shows that the impacts are linear in the momentum and so the coefficients are
	\begin{equation*}
		\alpha_j^i = \delta_j^i - 2\frac{\left(\pi_{\mathcal{D}}\nabla h\right)^i}
		{dh\left(\pi_{\mathcal{D}}\nabla h\right)} \left(
		\pi_{\mathcal{D}}^*dh\right)_j,
	\end{equation*}
	where $\left(\pi_{\mathcal{D}}\nabla h\right)^i$ is the $i^{th}$-component of the vector $\pi_{\mathcal{D}}\nabla h$ and similarly for $\pi_{\mathcal{D}}^*dh$. 
	
	We must now calculate the determinants of the matrices $A_k$. For the remainder of the proof, we will deal with the $n=4$ case but the general case works in the same way. For ease of notation, let $u:=\pi_{\mathcal{D}}\nabla h$ and $v:=\pi_{\mathcal{D}}^*dh$. Notice that in our choice of local coordinates, 
	$$dh(\pi_{\mathcal{D}}\nabla h) = dx^n(\pi_{\mathcal{D}}\nabla h) = \left(
	\pi_{\mathcal{D}}\nabla h\right)^n = u^n =: \frac{1}{\kappa}.$$
	The matrix $A = (\alpha_j^i)$ is given by
	\begin{equation*}
		A = \begin{bmatrix}
			1 - 2\kappa u^1v_1 & - 2\kappa u^2v_1 & - 2\kappa u^3v_1 & - 2\kappa u^4v_1 \\
			- 2\kappa u^1v_2 & 1 - 2\kappa u^2v_2 & - 2\kappa u^3v_2 & - 2\kappa u^4v_2 \\
			- 2\kappa u^1v_3 & - 2\kappa u^2v_3 & 1 - 2\kappa u^3v_3 & - 2\kappa u^4v_3 \\
		\end{bmatrix}.
	\end{equation*}
	The determinants $\det A_k$ are
	\begin{equation*}
		\begin{split}
			\det A_1 &= -2v_1, \\
			\det A_2 &= 2v_2, \\
			\det A_3 &= -2v_3, \\
			\det A_4 &= 1 - 2\kappa\left(u^1v_1 + u^2v_2 + u^3v_3\right) = 2v^4-1.
		\end{split}
	\end{equation*}
	Lemma \ref{le:wedge_power} asserts that 
	\begin{equation}\label{eq:th_n=4_omega}
		\left(\Delta^{\mathscr{C}*}\iota_{\tilde{S}}^*\omega\right)^3 = 
		(-1)^{\floor{\frac{n-1}{2}}}\cdot(n-1)!\cdot
		2\left[\left(
		-v_1\cdot\Omega_1 + v_2\cdot\Omega_2 - v_3\cdot\Omega_3 + v_4\cdot\Omega_4\right) - \Omega_4\right].
	\end{equation}
	To finish the theorem, we need to compute the wedge product of $\Delta^{\mathscr{C}*}\iota_{\tilde{S}}^*\nu_H^\mathscr{C}$ with \eqref{eq:th_n=4_omega}. It turns out that $\Delta^{\mathscr{C}*}\iota_{\tilde{S}}^*\nu_H^\mathscr{C}\wedge \left(\Delta^{\mathscr{C}*}\iota_{\tilde{S}}^*\omega\right)^3 = \iota_{S}^*\nu_H^\mathscr{C}\wedge \left(\Delta^{\mathscr{C}*}\iota_{\tilde{S}}^*\omega\right)^3$.
	This is because 
	\begin{equation*}
		\begin{split}
			\Delta^{\mathscr{C}*}\iota_{\tilde{S}}^*\nu_H^\mathscr{C} &= 
			\Delta^{\mathscr{C}*}\iota_{\tilde{S}}^* dH - \Delta^{\mathscr{C}*}\iota_{\tilde{S}}^*\left( m_{\alpha\beta}\{H,P(W^\alpha)\}\pi_Q^*\eta^\beta\right) \\
			&=\iota_S^*dH - \Delta^{\mathscr{C}*}\iota_{\tilde{S}}^*\left( m_{\alpha\beta}\{H,P(W^\alpha)\}\pi_Q^*\eta^\beta\right),
		\end{split}
	\end{equation*}
	by conservation of energy. Notice that the second term above has the form $\gamma_i\cdot dx^i$ which pairs to zero when wedged with any $\Omega_k$. Therefore,
	\begin{equation*}
		\begin{split}
			\Delta^{\mathscr{C}*}\iota_{\tilde{S}}^*\left(\nu_H^\mathscr{C}\wedge\omega^{n-1}\right) &= \iota_{S}^*dH \wedge 	\left(\Delta^{\mathscr{C}*}\iota_{\tilde{S}}^*\omega\right)^{n-1} \\
			&= (-1)^{\floor{\frac{n-1}{2}}}\cdot(n-1)!\cdot
			\left[ 2\left( v_j \cdot \frac{\partial H}{\partial p_j} \right) - \frac{\partial H}{\partial p_4} \right] \cdot\Omega\\
			&= (-1)^{\floor{\frac{n-1}{2}}}\cdot(n-1)! \cdot \left[
			2 \cdot \pi_{\mathcal{D}}^*dh\left(\Xi_H^\mathscr{C}\right) - dh\left(\Xi_H^\mathscr{C}\right)\right]\cdot\Omega.
		\end{split}
	\end{equation*}
	The result follows from applying Proposition \ref{prop:pre_jacobian} which says
	\begin{equation*}
		\iota_S^*\left(\nu_H^\mathscr{C}\wedge\omega^{n-1}\right) = 
		(-1)^{\floor{\frac{n-1}{2}}}\cdot(n-1)!\cdot dh\left(\Xi_H^\mathscr{C}\right)\cdot\Omega.
	\end{equation*}
	The quotient of coefficients is \eqref{eq:hybrid_nh_jacobian}.
\end{proof}

Since $\mathcal{J}_{\mu_{\mathscr{C}}}(\Delta^\mathcal{D})=1$, we need an invariant density to be conserved across impacts: if $f\mu_{\mathscr{C}}$ is invariant then $f\circ\Delta^\mathcal{D} = f|_{S^*_\mathcal{D}}$. As it turns out, there is a clear qualitative difference between nonholonomic systems with measures depending on configurations versus those who do not. This is because $\Delta^\mathcal{D}$ is the identity on the configuration variables but is not on the momenta/velocities. If $f$ only depends on the configurations, then $f\circ\Delta=f$ is automatically satisfied. If $f$ depends on the momenta/velocities, then we can always choose some impact surface, $S$, such that $f\circ\Delta\ne f$. This is summarized in the following proposition.
\begin{proposition}\label{prop:hybrid_volume}
	Let $L:TQ\to\mathbb{R}$ be a natural Lagrangian and $\mathcal{D}\subset TQ$ a regular distribution. Suppose that there exists an invariant volume form $f\mu_{\mathscr{C}}$ such that $f = \pi_Q^*h$ for some $h:Q\to\mathbb{R}$ (cf. Theorem \ref{th:NHVol}). Then $f\mu_{\mathscr{C}}\in\mathscr{A}_{\mathcal{H}}$ where $\mathcal{H} = (\mathcal{D}^*,S^*_\mathcal{D},X_H^\mathcal{D},\Delta^\mathcal{D})$ for \textit{any} $S\subset Q$.
\end{proposition}
Before we demonstrate this with examples, we will address the Zeno issue in measure-preserving systems.
\section{The Zeno issue in volume-preserving systems}\label{sec:Zeno}
In dynamical systems, invariant measures are useful for studying recurrent properties e.g. the Poincar\'{e} recurrence theorem and ergodic theory. However, in the context of hybrid systems, the existence of invariant measures has another advantage: they impose strong limitations on the Zeno behavior of the trajectories.

\begin{definition}[Zeno States]
	Let $\varphi_t^\mathcal{H}$ be the hybrid flow of a hybrid system $\mathcal{H} = (M,S,X,\Delta)$. A point $x\in M$ has a Zeno trajectory if there exists an increasing sequence of times $\{t_i\}_{i=0}^\infty$ such that $\varphi_{t_i}^\mathcal{H}(x)\in S$ for all $i$ and $t_i\to t_\infty < \infty$.
\end{definition}

It seems that there are only sufficient conditions for Zeno behavior \cite{4434891}. However, to the best of our knowledge, there are no results on necessary conditions which would provide a way to rule them out all together. We demonstrate that, under a few additional assumptions, the existence of an invariant measure rules our Zeno behavior (almost everywhere). For more results and properties of Zeno states, cf. e.g. \cite{4434891,1656623,brogliato2016,mattKv,orstability}.

To rule out Zeno behavior in hybrid systems, it is important to subdivide this into two classes: \textit{spasmodic} and \textit{steady}.
\begin{definition}
	Let $x\in M$ have a Zeno trajectory. The trajectory is spasmodic if the sequence $\{\varphi_{t_i}^\mathcal{H}\}_{i=0}^\infty$ escapes every compact set as $i\to \infty$. The trajectory is called steady if it is not spasmodic.
\end{definition}
\subsection{Spasmodic Zeno Trajectories}
Our goal is to utilize volume-preservation to show that Zeno almost never happens. However, this is \textit{not} true for spasmodic Zeno trajectories as this subsection demonstrates.
\begin{example}[Super-elastic spasmodic]\label{ex:spasmotic}
	Consider the hybrid mechanical system with $Q = [0,1]$ and $S = \{0,1\}$. Suppose that the impact map is given by $\Delta:\dot{q} \mapsto -\alpha\cdot \dot{q}$ where $\alpha > 0$ (which injects energy into the system). Then the particle bounces off the walls faster and faster until breaking in a finite amount of time.
	\begin{equation*}
		t_{Zeno} = \frac{1}{v_0}\, \sum_{k=0}^\infty \, \alpha^{-k} = \frac{1}{v_0}\cdot \frac{\alpha}{\alpha-1} < \infty.
	\end{equation*}
	We have a finite Zeno time but we also have 
	\begin{equation*}
		\lim_{t\to t_{Zeno}^-} \, \lvert\dot{q}(t)\rvert = \infty.
	\end{equation*}
\begin{figure}
	\centering
	\includegraphics[scale=0.75]{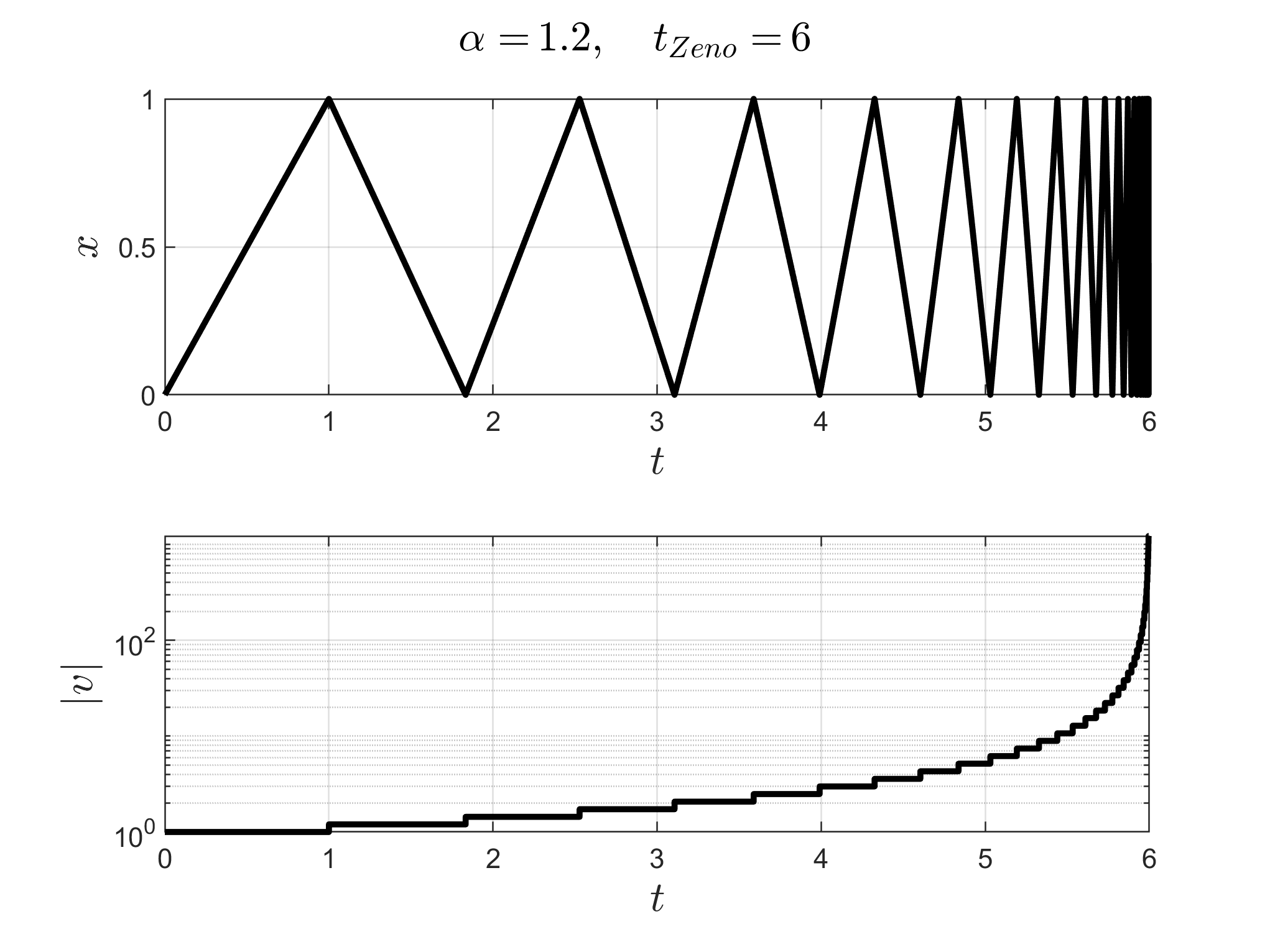}
	\caption{Trajectory of the spasmodic Zeno trajectory discussed in Example \ref{ex:spasmotic}. This has the first 40 impacts.}
\end{figure}
The common idea is that only sub-elastic collisions can lead to Zeno states. However, super-elastic collisions can still pose issues.
\end{example}

\begin{example}[Volume-preserving spasmodic]
	Consider a modification of Example \ref{ex:spasmotic} where $Q = [0,1]\times\mathbb{R}$ and $S = \{0,1\}\times \mathbb{R}$. We have the standard dynamics
	\begin{equation*}
		X = \dot{x}\frac{\partial}{\partial x} + \dot{y}\frac{\partial}{\partial y},
	\end{equation*}
	and suppose the impact map has the form $\delta(\dot{x},\dot{y}) = (-\alpha\dot{x},\beta\dot{y})$. It follows from the previous example that when $\alpha >1$, we have a spasmodic Zeno state. However, can we choose $\beta < 1$ such that this system is volume-preserving? This would lead to Zeno states in volume-preserving systems which is troublesome. Consider the volume form $\mu = dx\wedge dy\wedge d\dot{x}\wedge d\dot{y}$, then we have
	\begin{equation*}
		\iota_{\tilde{S}}^*i_X\mu = \dot{x} dy\wedge d\dot{x}\wedge d\dot{y}.
	\end{equation*}
	This gives us a hybrid Jacobian of
	\begin{equation*}
		\mathcal{J}_\mu(\Delta) = \alpha^2\beta.
	\end{equation*}
	If we take $\beta = \alpha^{-2}$, then $\mathcal{J}_\mu(\Delta)=1$ and we have a volume-preserving system with spasmodic Zeno states.
\end{example}

If the hybrid system possesses a Lyapunov function (i.e. a proper function which is non-increasing along trajectories) then spasmodic Zeno states are prohibited. This is good news as both unconstrained and nonholonomic systems preserve energy which is a Lyapunov function, provided it is proper.
\begin{proposition}
	Let $\mathcal{H} = (T^*Q,S^*,X_H,\Delta)$ be an impact Hamiltonian system with a natural Hamiltonian (this also holds true for nonholonomic systems). If $Q$ is compact, then spasmodic Zeno states do not occur.
\end{proposition}
\begin{proof}
	This follows from the fact that $H^{-1}(e)$ is compact for $e\in\mathbb{R}$.
\end{proof}

The following example illustrates the necessity of having $Q$ compact
in this proposition. 
\begin{example}[Elastic Hamiltonian spasmodic]
	Consider the (mathematical) Hamiltonian on $Q=\mathbb{R}^2$,
	\begin{equation*}
		H = p_x\left(1+x^2\right) + \frac{1}{2}p_y^2x.
	\end{equation*}
	The Hamiltonian vector field is
	\begin{equation*}
		X_H = \left(1+x^2\right)\frac{\partial}{\partial x} + p_yx\frac{\partial}{\partial y} - \left(2xp_x + \frac{1}{2}p_y^2\right)\frac{\partial}{\partial p_x}.
	\end{equation*}
	It is clear that if $x(0) = 0$, then $x(t) = \tan(t)$ which escapes to infinity at $t=\pi/2$. This can be used to make a spasmodic Zeno state by choosing $S = \{0,1\}\times\mathbb{R}$ with elastic impact map
	\begin{equation*}
		\Delta(x,y,p_x,p_y) = (x,y,p_x,-p_y).
	\end{equation*}
	This system is volume-preserving and energy-preserving but still has spasmodic Zeno states.
\end{example}
\subsection{Steady Zeno Trajectories}
Let us assume, henceforth, that any Zeno state will not be spasmodic, i.e. it will be steady.
Any Zeno issues will occur within the set $\mathcal{Z}:= \overline{S}\cap\overline{\Delta(S)}$, which by (H.5) has codimension at least 2 (exactly 2 for mechanical systems). We will therefore focus our attention on trajectories that intersect this set; let $\mathcal{N}$ be all points in $M$ that eventually move to $\mathcal{Z}$,
\begin{equation*}
	\mathcal{N} := \left\{ x\in M: \exists t>0~s.t.~ \lim_{s\to t^-}\, \varphi_s^\mathcal{H}(x)\in\mathcal{Z}\right\}.
\end{equation*}
Our goal is to show that $\mathcal{N}$ has zero measure. A key ingredient in proving this is the following assumption (which holds for mechanical systems).

\begin{proposition}
	Let $x\in M$ have a steady Zeno trajectory with impact times $\{t_i\}$ and let $d$ be the induced distance metric from a Riemannian metric on $M$. Then,
	\begin{equation*}
		\lim_{i\to \infty} \, d\left( \varphi_{t_i}^\mathcal{H}(x),\mathcal{Z}\right) = 0.
	\end{equation*}
\end{proposition}
\begin{proof}
	As $x$ is not spasmodic, there exists a compact set that contains the trajectory. Therefore, we can take $M$ to be compact. Let $g$ be a Riemannian metric on $M$ with induced distance metric $d$. Then, because $M$ is compact, we have
	\begin{equation*}
		\sup_{x\in M} \, \sqrt{ g(X(x),X(x)) } = \delta < \infty.
	\end{equation*}
	If follows that we have the uniform bound:
	\begin{equation*}
		d\left( x, \varphi_t(x) \right) < t\delta, \quad \forall x\in M.
	\end{equation*}
	Since the trajectory is Zeno, successive impact times converge to zero and the inequality above shows that successive impact locations converge, i.e. 
	\begin{equation*}
		\lim_{i\to\infty} \, d(x_i,\varphi_{t_i}(x_i)) = 0.
	\end{equation*}
	Therefore, $x_\infty\in\mathcal{Z}$.
\end{proof}
\begin{remark}
	When $\mathcal{H} = (TQ,\hat{S},X_L,\Delta)$ is a Lagrangian mechanical hybrid system, the Zeno set is $\mathcal{Z} = TS$. This has the interpretation that Zeno states occur when the impact surface is met tangentially, i.e. the impact is a ``scuff.'' As an impact becomes increasingly tangential, the impulse from the impact goes to zero. This phenomenon is key to ruling out Zeno states as the following assumption formalizes.
\end{remark}
\begin{assumption}[Boundary identity property]\label{ass:boundary}
	Consider a smooth hybrid dynamical system $\mathcal{H}=(M,S,X,\Delta)$. Then for any sequence $\{s_n\}\in S$ such that $s_n\to s\in\mathcal{Z}$, we have $\Delta(s_n)\to s$.
\end{assumption}
This assumption is useful because it allows us to ``complete'' the hybrid flow in a manner similar to \cite{1656623}. Essentially, suppose $x_0$ is Zeno so $\lim_{s\to t^-}\varphi_s^\mathcal{H}(x_0) = z_0\in\mathcal{Z}$. Then we define $\varphi_t^\mathcal{H}(x_0) := z_0$ and we can extend it via assumption (A.1). Let $\varepsilon>0$ such that $\varphi_t(z)$ does not intersect $S$ for all $t\in(0,\varepsilon)$. We define the \textit{completed} flow to be
$\varphi_{t+\delta}^\mathcal{H}(x_0) = \varphi_\delta\left(z_0\right)$.
If a hybrid flow is measure-preserving then its associated completed flow is too precisely due to the boundary identity property; we are ignoring any impacts at $z_0$ which comes from continuously extending the impact map from $S$ to $\overline{S}$. We can now state the following theorem.
\begin{theorem}\label{thm:noZeno}
	Suppose that $\mathcal{H} = (M,S,X,\Delta)$ is a smooth compact hybrid dynamical system (cf. Definition \ref{def:smooth}) with the boundary identity property, Assumption \ref{ass:boundary}. If $\mathcal{H}$ preserves a smooth measure $\mu$, then $\mu(\mathcal{N})=0$.
\end{theorem}
\begin{proof}
	The assumption that $\mathcal{H}$ be compact is to disallow spasmodic Zeno states.
	Partition $\mathcal{Z}$ into a countable collection of compact sets, $\{V_\alpha\}$, and partition $\mathcal{N}$ in the following way:
	\begin{equation*}
		\mathcal{N}_{\alpha,\delta} = \left\{ x\in\mathcal{N} : \exists t\in(0,\delta) \; s.t. \; \varphi_t^\mathcal{H}(x)\in V_\alpha\right\}.
	\end{equation*}
	It follows that if each $\mathcal{N}_{\alpha,\delta}$ has zero measure then all of $\mathcal{N}$ has zero measure, since a countable union of null sets is still a null set. In particular, we only need to prove that for all $\alpha$, there exists $\delta$ such that $\mathcal{N}_{\alpha,\delta}$ has zero measure. This is because for $\delta > s$
	\begin{equation*}
		\varphi_s^\mathcal{H}\left( \mathcal{N}_{\alpha,\delta}\setminus \mathcal{N}_{\alpha,s}\right) = \mathcal{N}_{\alpha,\delta-s}.
	\end{equation*}
	Fix an $\alpha$. By (A.1), for each $z\in V_\alpha$, there exists $\varepsilon>0$ such that $\varphi_t(z)\not\in S$ for all $t\in (0,\varepsilon)$. Let $\delta$ be the infimum of all such $\varepsilon$ which is positive due to the compactness of $V_\alpha$. By the measure-preserving property of the flow, we get
	\begin{equation*}
		\mu\left(\mathcal{N}_{\alpha,\delta/4} \right) =
		\mu\left( \varphi^\mathcal{H}_{\delta/2}\left(\mathcal{N}_{\alpha,\delta/4}\right) \right) \leq \mu\left( \mathcal{O}(V_\alpha,\delta)\right),
	\end{equation*}
	where 
	\begin{equation*}
		\mathcal{O}(V_\alpha,\delta) = \bigcup_{t\in (0,\delta)} \, \varphi_t(V_\alpha).
	\end{equation*}
	Because zero impacts occur, the set $\mathcal{O}(V_\alpha,\delta)$ is a manifold with codimension at least 1 which necessarily has zero measure.
\end{proof}
In order for a hybrid dynamical system to be volume-preserving, its divergence must be zero \textit{and} its hybrid Jacobian must be one. However, if a trajectory is Zeno, the continuous component of the flow is finite while the impact component is infinite. This seems to suggest that the hybrid Jacobian controls Zeno states much more than the divergence does.
\begin{theorem}\label{th:zeno_impacts}
	Let $\mathcal{H}=(M,S,X,\Delta)$ be a smooth compact hybrid dynamical system with the boundary identity property and let $\mu\in\Omega^n(M)$ be a volume form. If $\mathcal{J}_\mu(\Delta)= 1$ everywhere on $S$, then $\tilde{\mathcal{Z}}$ has measure zero.
\end{theorem}
\begin{proof}
	This follows via the same construction as in Theorem \ref{thm:noZeno} along with the following bounds on $\mu$. Notice that $\mu_t := \left(\varphi_t^\mathcal{H}\right)^*\mu$ satisfies the following differential equation:
	\begin{equation*}
		\dot{\mu}_t = \mathrm{div}_\mu(X)\cdot \mu_t,
	\end{equation*}
	(notice that nothing occurs at impacts as $\mathcal{J}_\mu(X)=1$).
	Let $M:=\max \mathrm{div}_\mu(X)$ and $m = \min \mathrm{div}_\mu(X)$. Then Gr\"{o}nwall's inequality states that for any finite-time, the translated volume will be equivalent to the original volume. Therefore the set $\mathcal{O}(V_\alpha,\delta)$ will still be a null set.
\end{proof}
\begin{corollary}
	Nonholonomic mechanical systems have almost no Zeno states.
\end{corollary}
Even though nonholonomic systems can experience dissipation during the continuous phase, there is never dissipation occurring during the moment of impact. Therefore, a trajectory in any mechanical system (with elastic impacts) will almost never be Zeno. We will use this to justify that our ignorance of Zeno states is essentially benign.
\subsection{Steady Zeno in an Elastic System}
It is important to notice that Theorem \ref{thm:noZeno} states that Zeno \textit{almost never} happens. Below, we present an example of an elastic impact system which does possess a Zeno trajectory.

We start with the observation that the inelastic bouncing ball is Zeno. Consider the example with continuous dynamics
\begin{equation*}
	\ddot{x} = 0, \quad \ddot{y} = -1,
\end{equation*}
which is the usual falling particle in the plane. Suppose that an impact occurs when $y=0$ and the reset is described via
\begin{equation*}
	\Delta(x,y,\dot{x},\dot{y}) = (x,y,\dot{x},-\alpha\dot{y}),
\end{equation*}
where $0\leq\alpha\leq 1$ is the coefficient of restitution. 

Let $\alpha<1$. The resulting dynamics are Zeno by the following: Suppose that we have the initial conditions $y(0)=0$ and $\dot{y}(0) = v_0>0$. The time of the first impact is given by $t_1 = 2v_0$. Consequently, $t_2-t_1 = 2v_1 = 2\alpha v_0$. The trajectory is Zeno as the limit converges,
\begin{equation*}
	t_{Zeno} = \lim_{k\to\infty} \, t_k = \lim_{k\to\infty} \, \sum_{j=1}^k \, \Delta t_j = \lim_{k\to\infty} \, \sum_{j=1}^k \, 2v_0\alpha^j = \frac{2v_0}{1-\alpha} < \infty.
\end{equation*}
Each jump in the inelastic bouncing ball is shorter than the previous which results in Zeno. For elastic bouncing, all jumps must remain the same height. As such, we will artificially shrink the bounces by raising the table.

Consider the curve
\begin{equation}\label{eq:Zeno_curve}
	y = \alpha^2x\left(\sqrt{2}v_0-\alpha^2x\right),
\end{equation}
for some parameters $v_0>0$ and $0<\alpha<1$. Notice that $y(0)=0$ and its maximum occurs at
\begin{equation*}
	y_{max} = y\left( \frac{v_0}{\sqrt{2}\alpha^2}\right) = \frac{1}{2}v_0^2.
\end{equation*}
Consider the hybrid system with continuous dynamics
\begin{equation*}
	\ddot{x} = 0,\quad \ddot{y} = -1,
\end{equation*}
with impact occurring when \eqref{eq:Zeno_curve} is satisfied with reset
\begin{equation*}
	\begin{array}{ll}
		x \mapsto x, & \dot{x}\mapsto \dot{x} \\
		y\mapsto y, & \dot{y}\mapsto -\dot{y}.
	\end{array}
\end{equation*}
\begin{proposition}
	The trajectory of the above hybrid system with the initial conditions
	\begin{equation*}
		\begin{array}{ll}
			x(0) = 0, & \dot{x}(0) = 1, \\
			y(0) = 0, & \dot{y}(0) = v_0,
		\end{array}
	\end{equation*}
	has a steady Zeno trajectory and the Zeno points happens at
	\begin{equation}\label{eq:Zeno_point}
		x\left( \frac{v_0}{\sqrt{2}\alpha^2}\right) = \frac{v_0}{\sqrt{2}\alpha^2}, \quad y\left(\frac{v_0}{\sqrt{2}\alpha^2}\right) = \frac{v_0^2}{2}.
	\end{equation}
\end{proposition}
\begin{theorem}
	Let $(x_n,y_n)$ be the location of the $n^{th}$ impact for the above Zeno trajectory. Choose a function $h$ such that $y_n = h(x_n)$ and $h'(x_n)=0$ for all $n$. Then $h\not\in C^3$.
\end{theorem}
\begin{proof}
	We will compute $\mathrm{Var}(h'')$ and show that it is not finite. Let $(x_Z,y_Z)$ be the Zeno point \eqref{eq:Zeno_point}. In particular,
	\begin{equation*}
		\lim_{n\to\infty}\, x_n = x_Z, \quad \lim_{n\to\infty} \, y_n = y_Z.
	\end{equation*}
	By the mean value theorem, there exists a point $c\in(x_n,x_{n+1})$ such that
	\begin{equation*}
		h'(c) = \frac{y_{n+1}-y_n}{x_{n+1}-x_n}.
	\end{equation*}
	Applying the mean value theorem a second time, there exists a point $\tilde{c}\in (x_n,c)$ such that
	\begin{equation*}
		h''(\tilde{c}) = \frac{h'(c)-h'(x_n)}{c-x_n} \geq \frac{y_{n+1}-y_n}{(x_{n+1}-x_n)^2}.
	\end{equation*}
	Likewise, there exists some $\bar{c}\in (c,x_{n+1})$ such that
	\begin{equation*}
		h''(\bar{c}) \leq \frac{y_n-y_{n+1}}{(x_{n+1}-x_n)^2}.
	\end{equation*}
	Combining both of these estimates, we see that
	\begin{equation*}
		\mathrm{Var}(h'') \geq 2 \sum_{n=0}^\infty \, \frac{y_{n+1}-y_n}{(x_{n+1}-x_n)^2}.
	\end{equation*}
	The result follows as long as the sum is divergent. Let $(x_n,y_n)$ be an arbitrary impact location. The vertical velocity (post-impact) is
	\begin{equation*}
		v_n = \sqrt{v_0^2 - 2y_n}.
	\end{equation*}
	Therefore, the next impact occurs when
	\begin{equation*}
		-\frac{1}{2}\left(x_{n+1}-x_n\right)^2 + \sqrt{v_0^2-2y_n}(x_{n+1}-x_n) + y_n = \alpha^2x_{n+1}(\sqrt{2}v_0-\alpha^2x_{n+1}),
	\end{equation*}
	where the left side is the vertical trajectory, $y_{n+1}$, as $x=t$ and the right side is the impact location. Solving for $x_{n+1}$, 
	\begin{equation*}
		x_{n+1} = -x_n - \frac{2}{2\alpha^4-1}\left( x_n + \sqrt{v_0^2-2\alpha^2x_n(\sqrt{2}v_0-\alpha^2x_n)}-\alpha^2\sqrt{2}v_0\right).
	\end{equation*}
	This produces
	\begin{equation*}
		\lim_{x_n\to x_Z} \,\frac{y_{n+1}-y_n}{(x_{n+1}-x_n)^2} = \frac{\alpha^2}{\sqrt{2}}\ne 0.
	\end{equation*}
	Therefore, the sum is divergent and $\mathrm{Var}(h'')=\infty$.
\end{proof}
\begin{figure}
	\centering
	\includegraphics[scale=0.75]{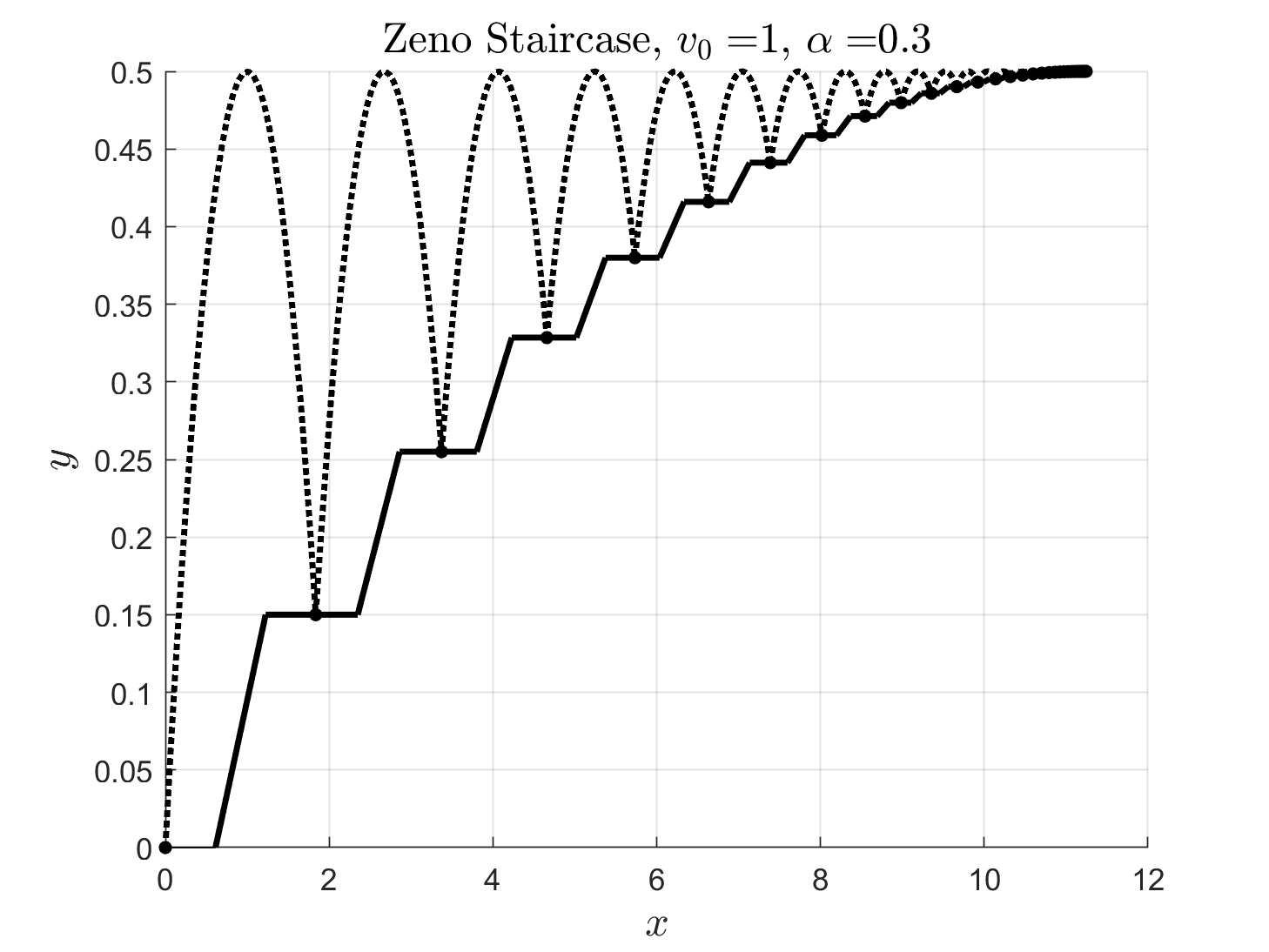}
	\caption{A steady Zeno trajectory in an elastic mechanical system.}
\end{figure}
This theorem prompts the following conjecture.
\begin{conjecture}
	Let $(T^*Q,S^*,X_H,\Delta)$ be a natural and elastic Hamiltonian system. If $S\in C^3$, then there are no Zeno trajectories.
\end{conjecture}
\section{Filippov Systems}\label{sec:filippov}
Throughout this work, it was implicitly assumed that the reset map was not the identity. In this section, we draw attention to the special case where the reset is the identity but the vector field is discontinuous. Systems of this form appear in e.g. nonsmooth stabilization \cite{blockdrakunov}.
 
Filippov systems are continuous-time dynamical systems where the vector field is discontinuous. Let $M$ be a smooth manifold, $h:M\to\mathbb{R}$ be a smooth function with zero as a regular value, and $f,g\in\mathfrak{X}(M)$ be two smooth vector fields. We will call 
\begin{equation}\label{eq:filippov}
	(f,g)_h(x) := \begin{cases}
		f(x), & h(x) > 0, \\
		g(x), & h(x) < 0.
	\end{cases}
\end{equation}
For an in-depth review of these systems see, e.g. \cite{filippov}.

We answer whether or not \eqref{eq:filippov} preserves a volume-form, which was studied in \cite{novaesvarao}. We present the following theorem which addresses whether or not a volume is preserved under a Filippov system and note that we get the same result as in \cite{novaesvarao}.
\begin{theorem}
	Let $\mathrm{Id}_h^{f,g}$ denote the augmented differential for \eqref{eq:filippov}. Then,
	\begin{equation}\label{eq:filippov_differential}
		\mathrm{Id}_h^{f,g} = \mathrm{Id} + \frac{1}{dh(f)}\cdot\left(g-f\right)\otimes dh.
	\end{equation}
	In particular, if $\mu\in\Omega^{\dim M}(M)$ is a volume-form, the hybrid Jacobian is
	\begin{equation*}
		\mathcal{J}_\mu\left( \mathrm{Id}\right) = \det\left( \mathrm{Id}_h^{f,g}\right) = \frac{dh(g)}{dh(f)}.
	\end{equation*}
\end{theorem}
\begin{proof}
	The augmented differential must satisfy
	\begin{equation*}
		\begin{cases}
			\mathrm{Id}_h^{f,g}\cdot v = v, & dh(v) = 0, \\
			\mathrm{Id}_h^{f,g}\cdot f = g,
		\end{cases}
	\end{equation*}
	where we tacitly assume that $dh(f)\ne 0$. It can be seen that \eqref{eq:filippov_differential} satisfies this. The hybrid Jacobian follows from the matrix determinant lemma.
\end{proof}
\begin{corollary}
	Let $\mu\in\Omega^{\dim M}(M)$ be a volume form and define the (discontinuous) density by 
	\begin{equation*}
		\rho(x) = \begin{cases}
			\alpha^+(x), & h(x) > 0 \\
			\alpha^-(x), & h(x) < 0.
		\end{cases}
	\end{equation*}
	Then $\rho\mu$ is invariant if and only if
	\begin{equation*}
		\begin{split}
			\mathcal{L}_f\alpha^+ + \alpha^+\mathrm{div}_\mu(f) &= 0, \quad h(x) > 0, \\
			\mathcal{L}_g\alpha^- + \alpha^-\mathrm{div}_\mu(g) &= 0, \quad h(x) < 0,
		\end{split}
	\end{equation*}
	and $\alpha^+(x)dh_x(f(x)) = \alpha^-(x)dh_x(g(x))$ for all $x\in h^{-1}(0)$.
\end{corollary}
\begin{remark}
	Note the similarity between the discontinuous invariant volume condition and the Rankine-Hugoniot jump conditions for a shock wave.
\end{remark}
\section{Examples}\label{sec:examples}
We present three examples of nonholonomic mechanical systems: the Chaplygin sleigh, the (uniform) rolling ball, and the vertical rolling disk.
\subsection{The Chaplygin sleigh}
The Chaplygin sleigh is a nonholonomic system with configuration space, $Q=\mathrm{SE}_2$, the special Euclidean group and has the following Lagrangian.
\begin{equation*}
	L = \frac{1}{2}\left( m\dot{x}^2 + m\dot{y}^2 + (I+ma^2)\dot{\theta}^2 - 2ma\dot{x}\dot\theta\sin\theta + 2ma\dot{y}\dot{\theta}\cos\theta\right),
\end{equation*}
where $m$ is the mass of the sleigh, $I$ is its rotational moment of inertia, and $a$ is the distance from the center of mass to the contact point (cf. \S 1.7 in \cite{bloch2008nonholonomic}). The constraint is that the sleigh can only slide in the direction in which it is pointed which is described by
\begin{equation*}
	\dot{y}\cos\theta - \dot{x}\sin\theta = 0.
\end{equation*}
It was shown in \cite{nhvolume} that the Chaplygin sleigh has no invariant volumes with density depending only on configuration. Therefore, we do not get an immediate hybrid-invariant volume. In \cite{clarkthesis} it is shown that $p_\theta^{-3}$ is an invariant density for the sleigh. However, this density is \textit{not} preserved across arbitrary impacts. Therefore, \textit{no smooth hybrid-invariant measures exist for the Chaplygin sleigh}.
\subsubsection{Statistical distribution of the sleigh}
To augment the discussion on volume-preservation for the Chaplygin sleigh, we present some numerical calculations for the statistical distribution of its trajectories. For the choice of the impact set $S$, we chose an elliptical billiard table:
\begin{equation*}
	\begin{split}
	E &= \left\{ (x,y)\in\mathbb{R}^2 : \frac{x^2}{a^2} + \frac{y^2}{b^2} \leq 1\right\}, \\
	S &= \left\{ (x,y,\theta) \in \mathrm{SE}_2: (x\pm L\cos\theta,y\pm L\sin\theta)\in\partial E\right\},
	\end{split}
\end{equation*}
where $L$ is the length of the sleigh and the $\pm$ in the definition of $S$ corresponds to the front or the back of the sleigh striking the wall. The full state of the system lies in the set $\mathcal{D}^*\subset T^*\mathrm{SE}_2$, which is 5-dimensional. To simplify the figures, we only track the $(x,y)$-location of the sleigh. This produces a function $\rho:E\to\mathbb{R}$ given by
\begin{equation}\label{eq:density_numeric}
	\int_A\, \rho(x,y) \, dxdy = \lim_{n\to\infty} \frac{1}{n}\, \sum_{k=0}^{n-1}\, \chi_{\tilde{A}}(\varphi^k(z)),
\end{equation}
where $\varphi = \varphi_1^\mathcal{H}$ is the time-1 map of the hybrid dynamics, $\chi_{\tilde{A}}$ is the indicator function for the set $\tilde{A}$, and 
\begin{equation*}
	\tilde{A} = \left\{ (x,y,\theta,p_x,p_y,p_\theta)\in \mathcal{D}^*\subset T^*\mathrm{SE}_2: (x,y)\in A\right\}.
\end{equation*}
Plots of numerical approximations for $\rho$ are shown in Figure \ref{fig:locations_sleigh}.
\begin{figure}
	\begin{subfigure}[t]{.32\columnwidth}
		\centering
		\includegraphics[width=\linewidth]{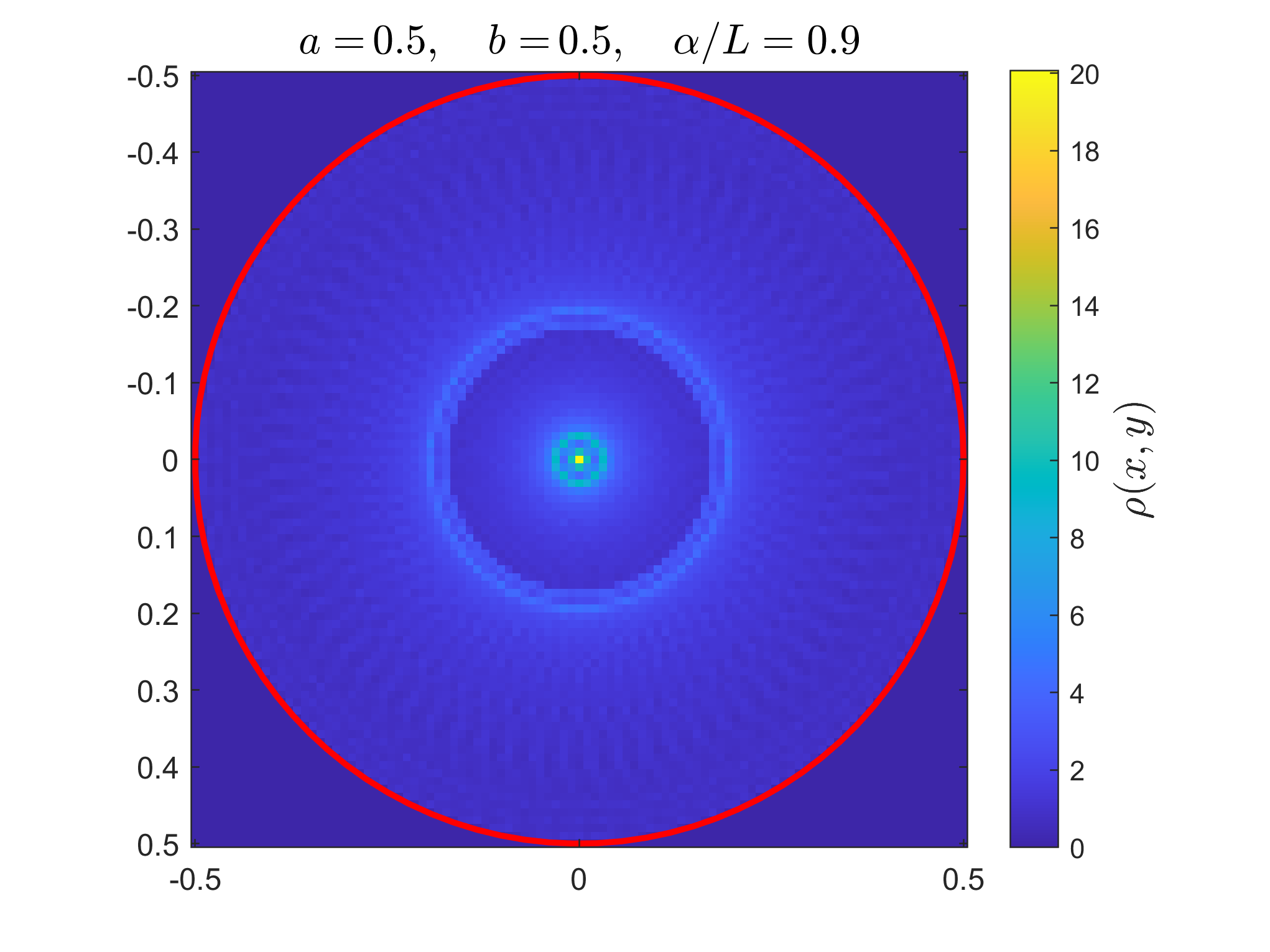}
	\end{subfigure}
	\hfill
	\begin{subfigure}[t]{.32\columnwidth}
		\centering
		\includegraphics[width=\linewidth]{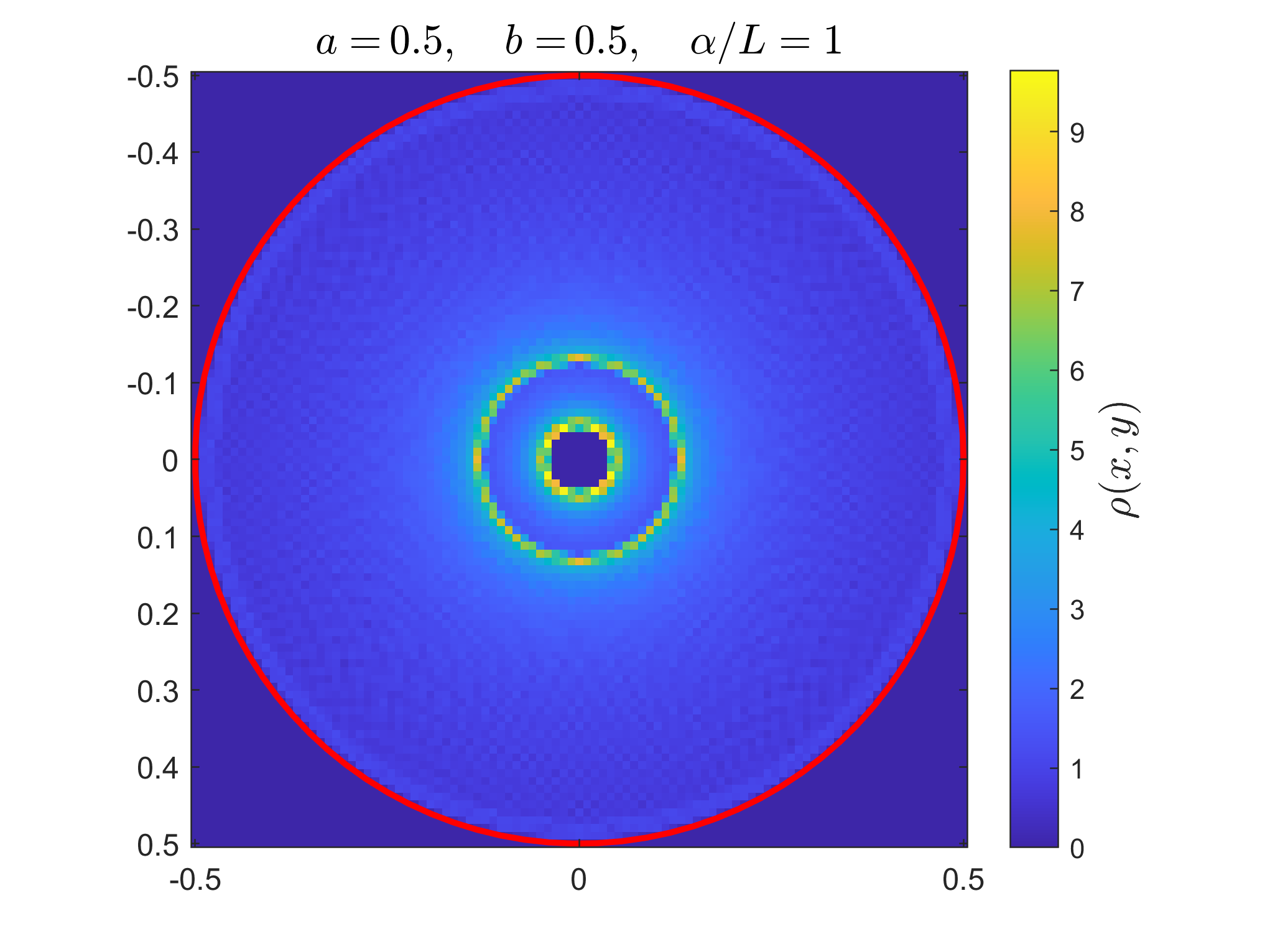}
	\end{subfigure}
	\hfill
	\begin{subfigure}[t]{0.32\columnwidth}
		\centering
		\includegraphics[width=\linewidth]{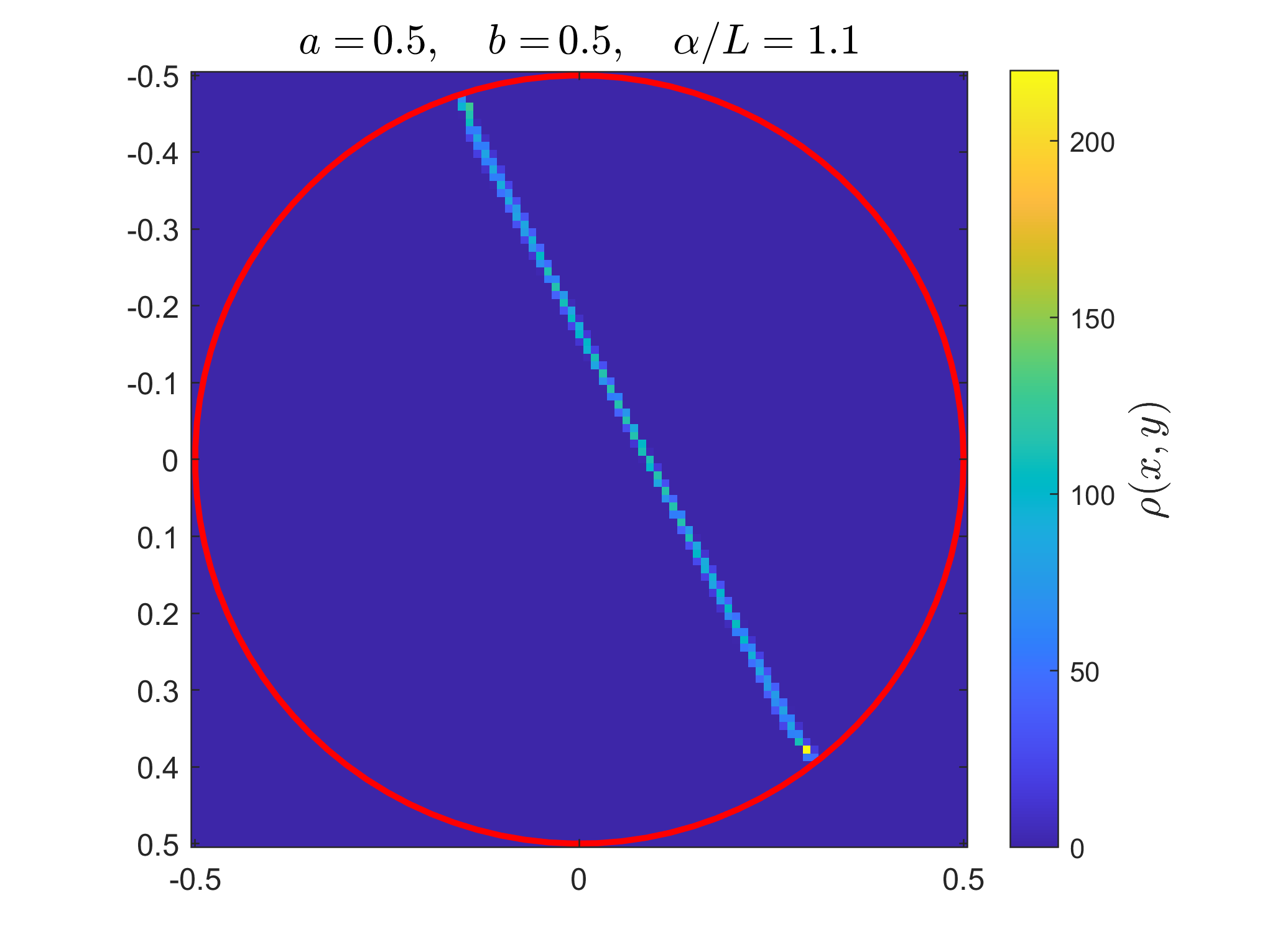}
	\end{subfigure}
	
	\medskip
	
	\begin{subfigure}[t]{0.32\columnwidth}
		\centering
		\includegraphics[width=\linewidth]{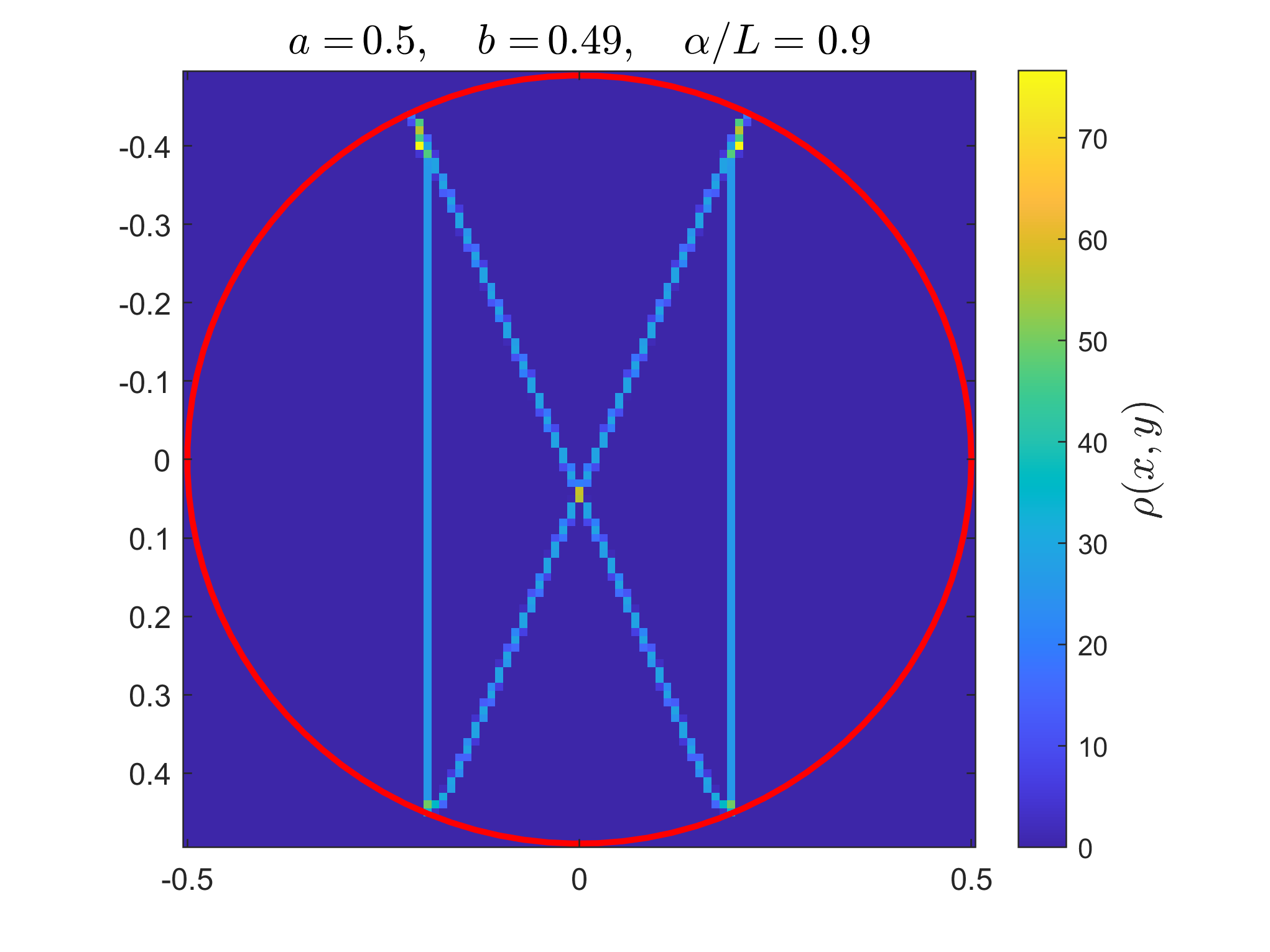}
	\end{subfigure}
	\hfill
	\begin{subfigure}[t]{0.32\columnwidth}
		\centering
		\includegraphics[width=\linewidth]{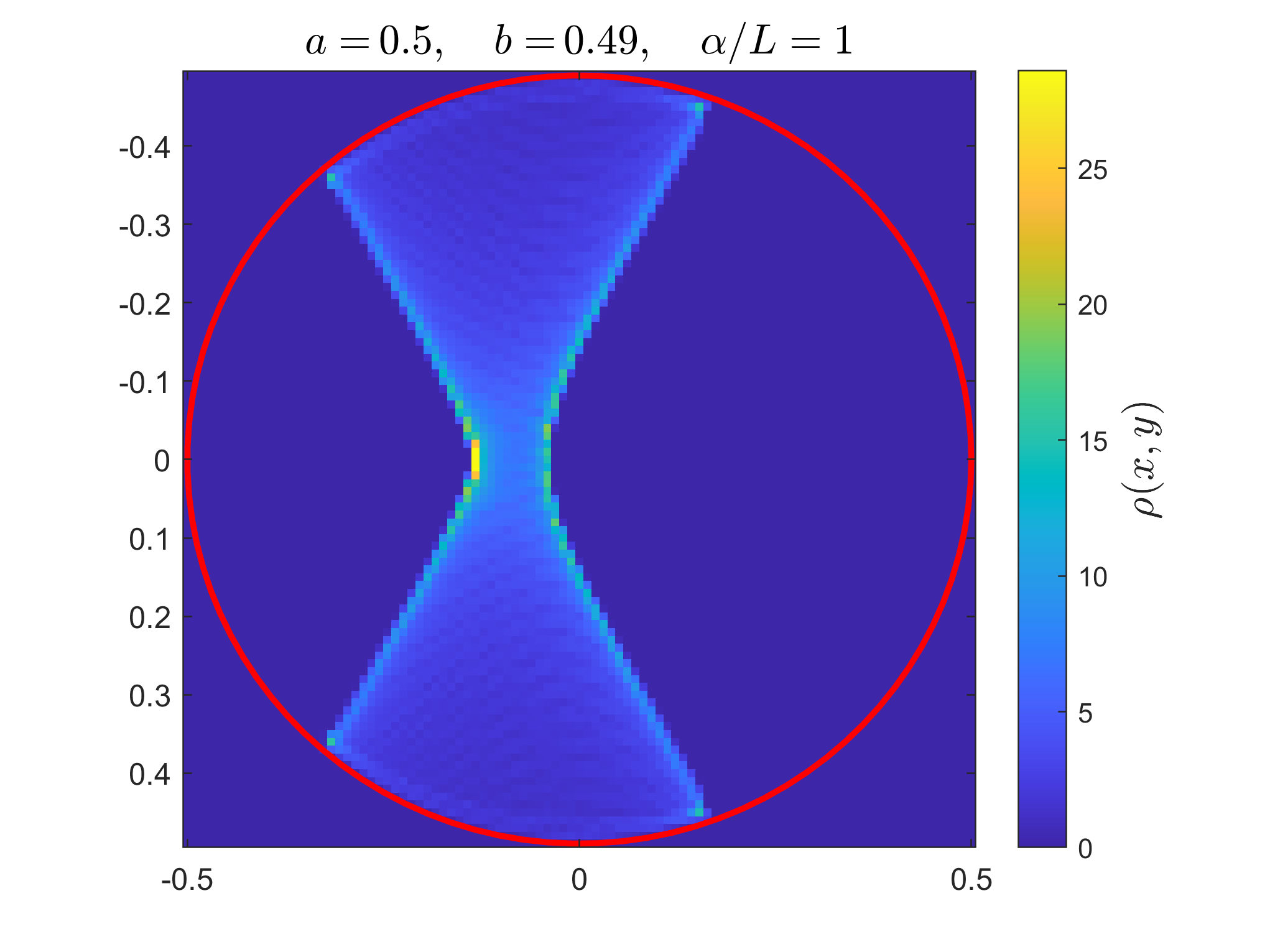}
	\end{subfigure}
	\hfill
	\begin{subfigure}[t]{0.32\columnwidth}
		\centering
		\includegraphics[width=\linewidth]{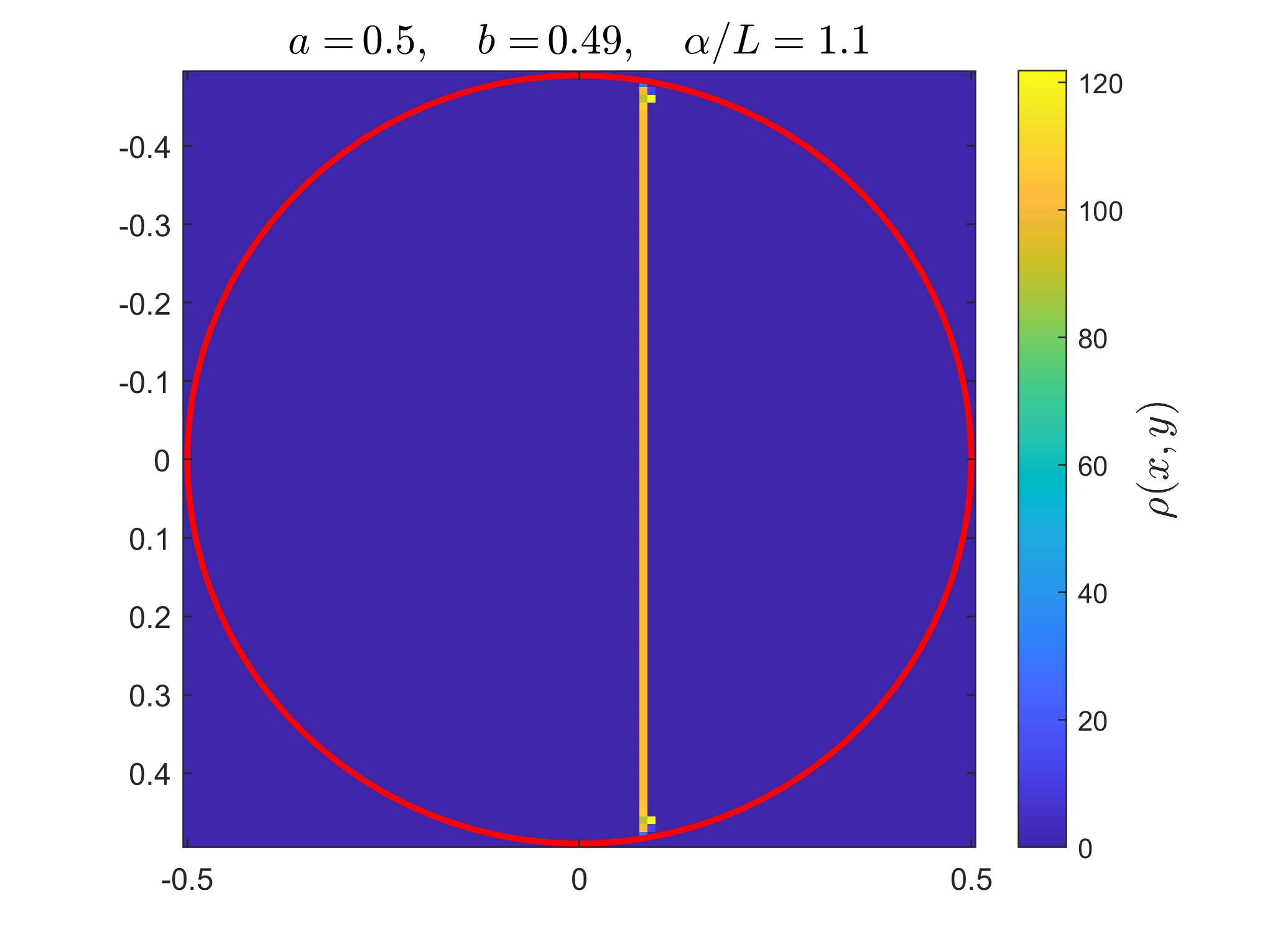}
	\end{subfigure}

	\medskip
	
	\begin{subfigure}[t]{0.32\columnwidth}
		\centering
		\includegraphics[width=\linewidth]{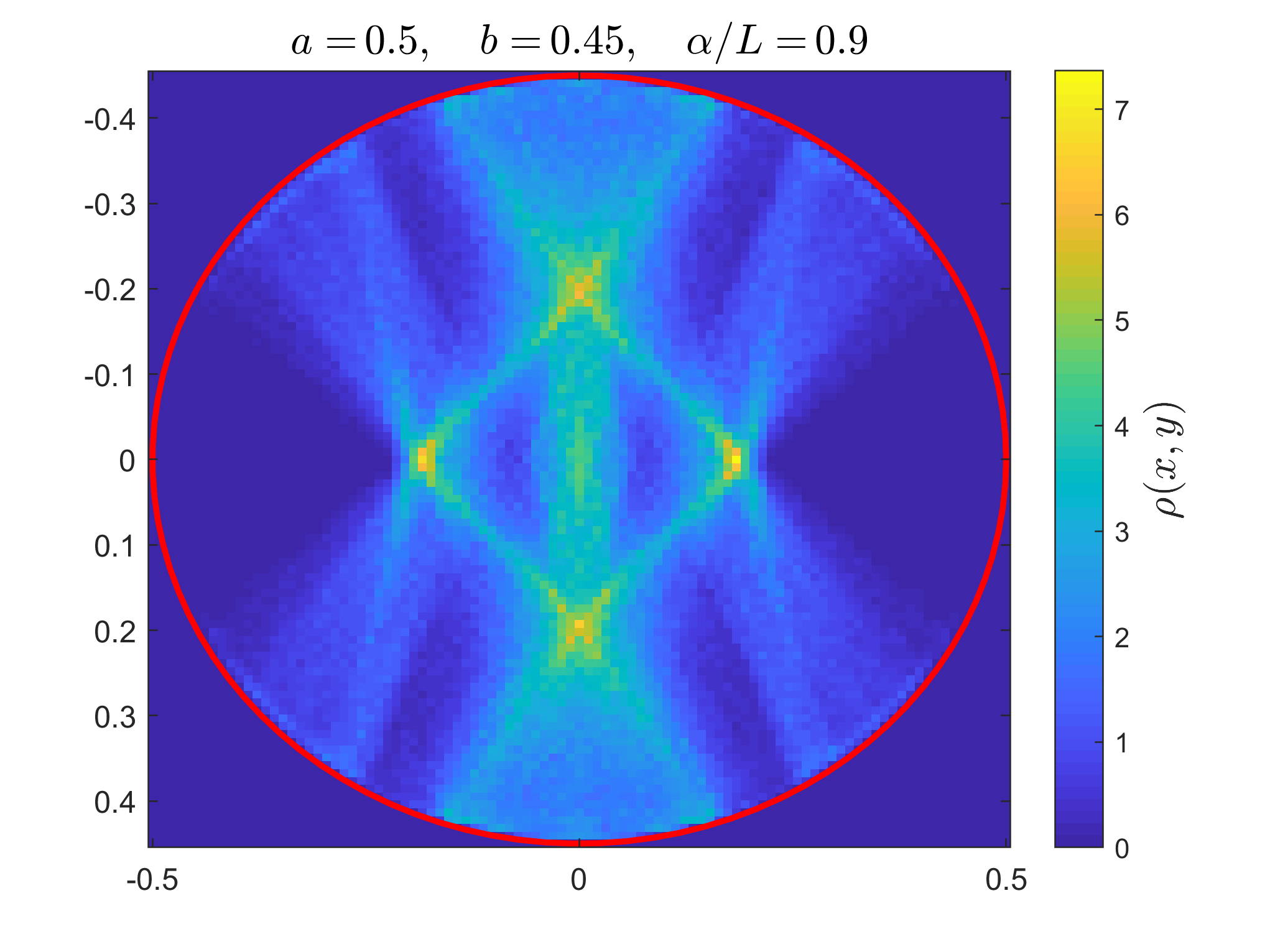}
	\end{subfigure}
	\hfill
	\begin{subfigure}[t]{0.32\columnwidth}
		\centering
		\includegraphics[width=\linewidth]{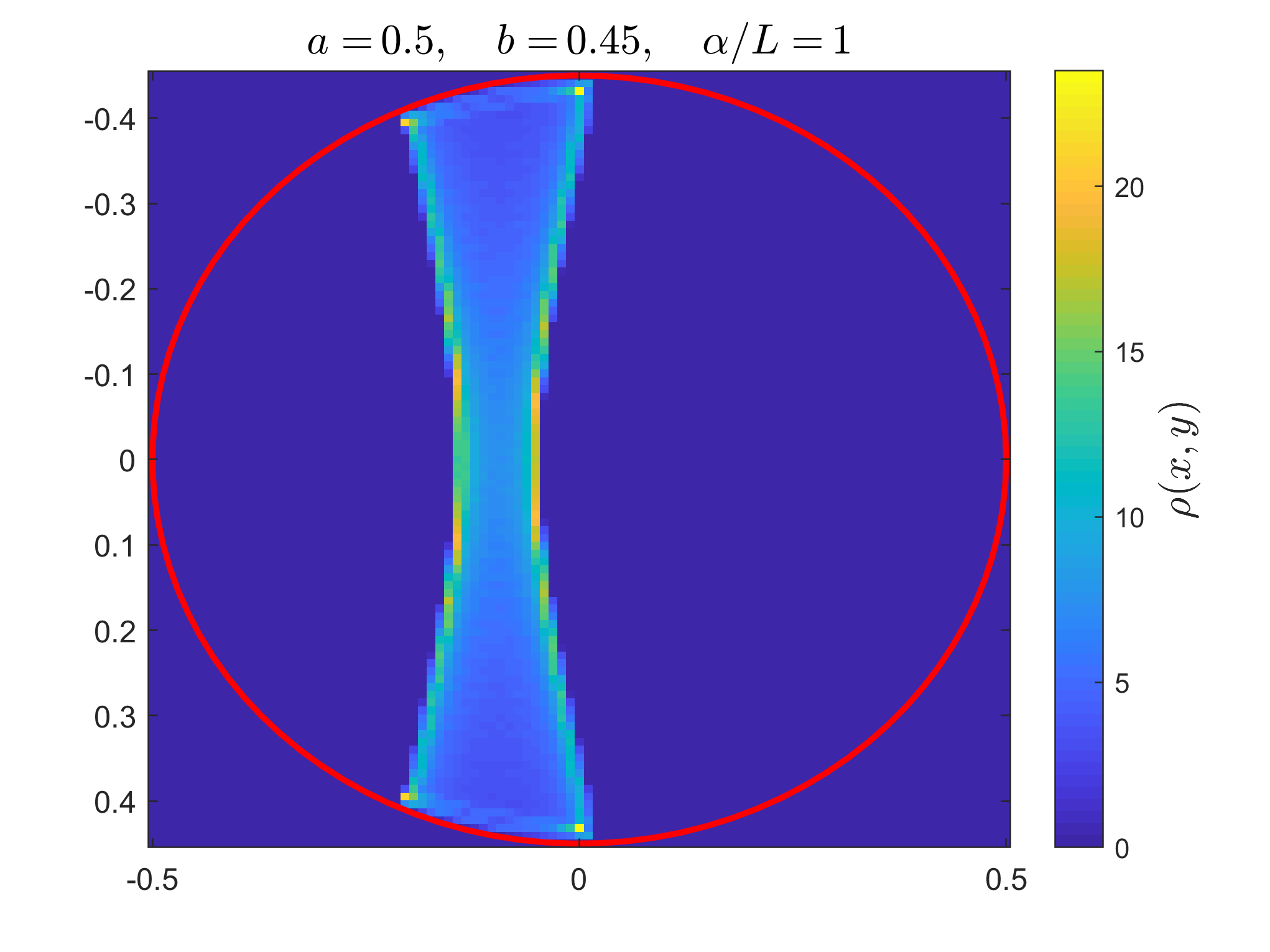}
	\end{subfigure}
	\hfill
	\begin{subfigure}[t]{0.32\columnwidth}
		\centering
		\includegraphics[width=\linewidth]{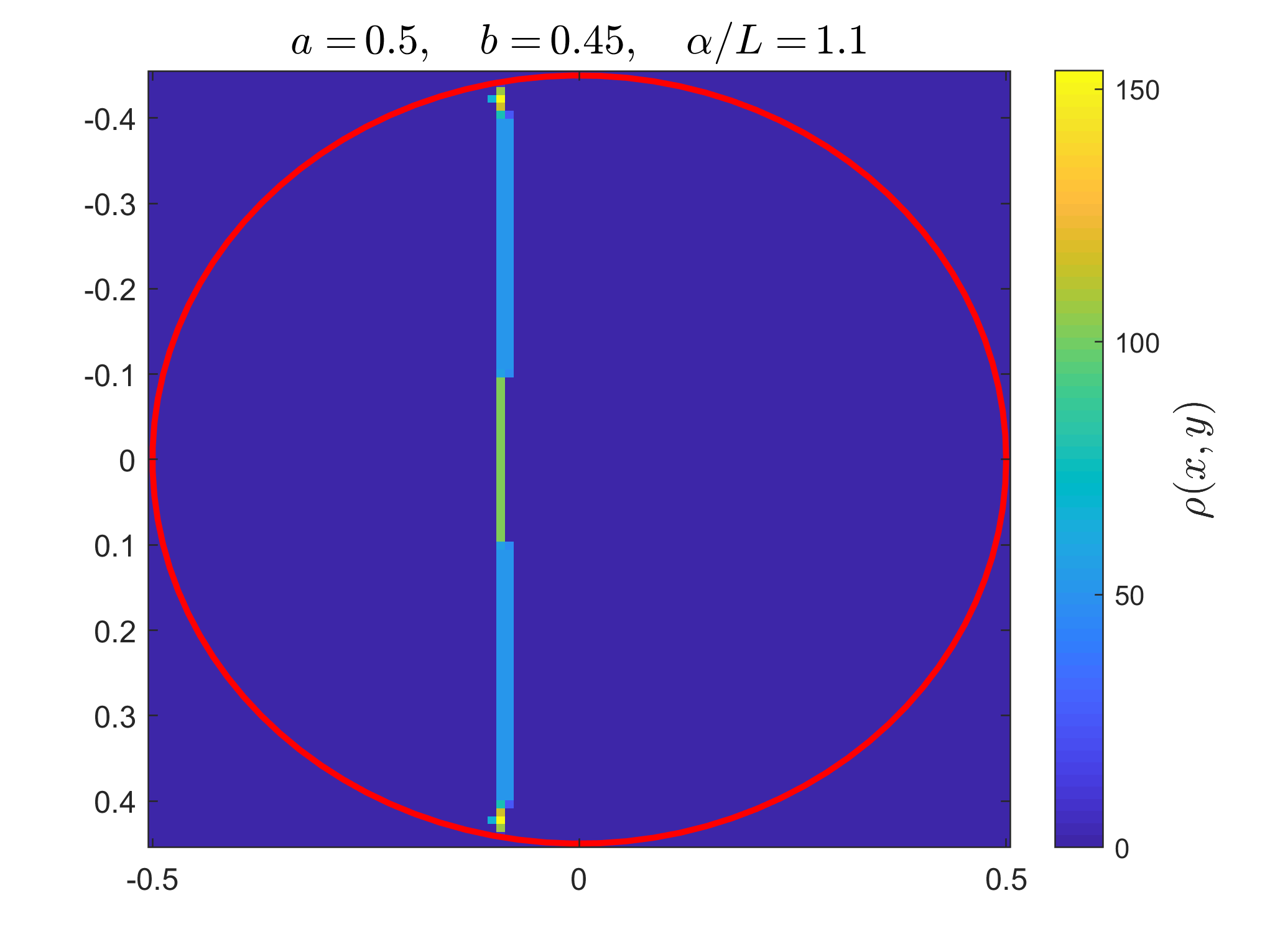}
	\end{subfigure}
	
	\caption{Plots of the statistical asymptotic behavior of the billiard Chaplygin sleigh. Each row corresponds to a different semi-minor axis length while each column corresponds to a different length of the sleigh. The density $\rho$ is approximated by dividing $E$ into 100 by 100 boxes and tracking the locations for $10^6$ time-1 maps. The first thousand iterations were discarded to minimize the influence of transient behavior.}
	\label{fig:locations_sleigh}
\end{figure}

\subsection{The rolling ball}
The next example is that of the (homogeneous) rolling ball; for more detail, cf. e.g. \cite{monforte2004geometric}. The Lagrangian is the kinetric energy and is given by
\begin{equation*}
	L = \frac{1}{2}\left( \dot{x}^2 + \dot{y}^2 + k^2\left( \dot{\theta}^2+\dot{\varphi}^2+\dot{\psi}^2 + 2\dot{\varphi}\dot{\psi}\cos\theta\right) \right).
\end{equation*}
The constraints which prohibit slipping while rolling are given by
\begin{equation*}
	\begin{split}
		\dot{x} &= r\dot\theta\sin\psi - r\dot{\varphi}\sin\theta\cos\psi, \\
		\dot{y} &= -r\dot{\theta}\cos\psi - r\dot{\varphi}\sin\theta\sin\psi.
	\end{split}
\end{equation*}
The rolling ball possesses the invariant volume whose density depends only on the configuration variables, cf. \cite{nhvolume}. Therefore, the invariant volume is also hybrid-invariant by Proposition \ref{prop:hybrid_volume}. The Poincar\'{e} recurrence theorem can be applied to obtain that almost every trajectory is recurrent for any compact table-top $E\subset \mathbb{R}^2$.


\subsection{The vertical rolling disk}
The final example presented here will be the vertical rolling disk. The Lagrangian is, again, the kinetic energy of the disk,
\begin{equation}
	L = \frac{1}{2}m\left(\dot{x}^2+\dot{y}^2\right) + \frac{1}{2}I\dot{\theta}^2 + \frac{1}{2}J\dot{\varphi}^2.
\end{equation}
Here, $m$ is the mass of the disk, $I$ is the moment of inertia of the disk about the axis perpendicular to the plane of the disk, $J$ is the moment of inertia about an axis in the plane of the disk, and $R$ is the radius of the disk. The constraints enforcing rolling without slipping are
\begin{equation*}
	\begin{split}
		\dot{x} &= R\dot\theta\cos\varphi, \\
		\dot{y} &= R\dot\theta\sin\varphi.
	\end{split}
\end{equation*}
More information on the hybrid nonholonomic equations of motion can be found in \cite{bpenny}. Similarly to the rolling ball, the vertical rolling disk possesses an invariant volume whose density depends only on the configuation variables, \cite{nhvolume}, and Proposition \ref{prop:hybrid_volume} can be applied.

\subsubsection{Statistical distribution of the disk}
We present some numerical results on the long term behavior of the disk. Consider the interior of the billiard table (an ellipse)
\begin{equation*}
	\begin{split}
	E &= \left\{ (x,y) \in \mathbb{R}^2: \frac{x^2}{a^2} + \frac{y^2}{b^2} \leq 1\right\}, \\
	S &= \left\{ (x,y,\theta,\varphi)\in\mathbb{R}^2\times S^1\times S^1: (x\pm R\cos\varphi,y\pm R\sin\varphi)\in \partial E\right\},
	\end{split}
\end{equation*}
where the $\pm$ serves the same roll as in the sleigh; either the front or the back of the disk may strike the wall.
The density function $\rho:E\to\mathbb{R}$ is given by the same numerical procedure as \eqref{eq:density_numeric} with the definition $\tilde{A}$ changed accordingly. Numerical results can be found in Figure \ref{fig:locations_penny}.
\begin{figure}
	\begin{subfigure}[t]{.48\columnwidth}
		\centering
		\includegraphics[width=\linewidth]{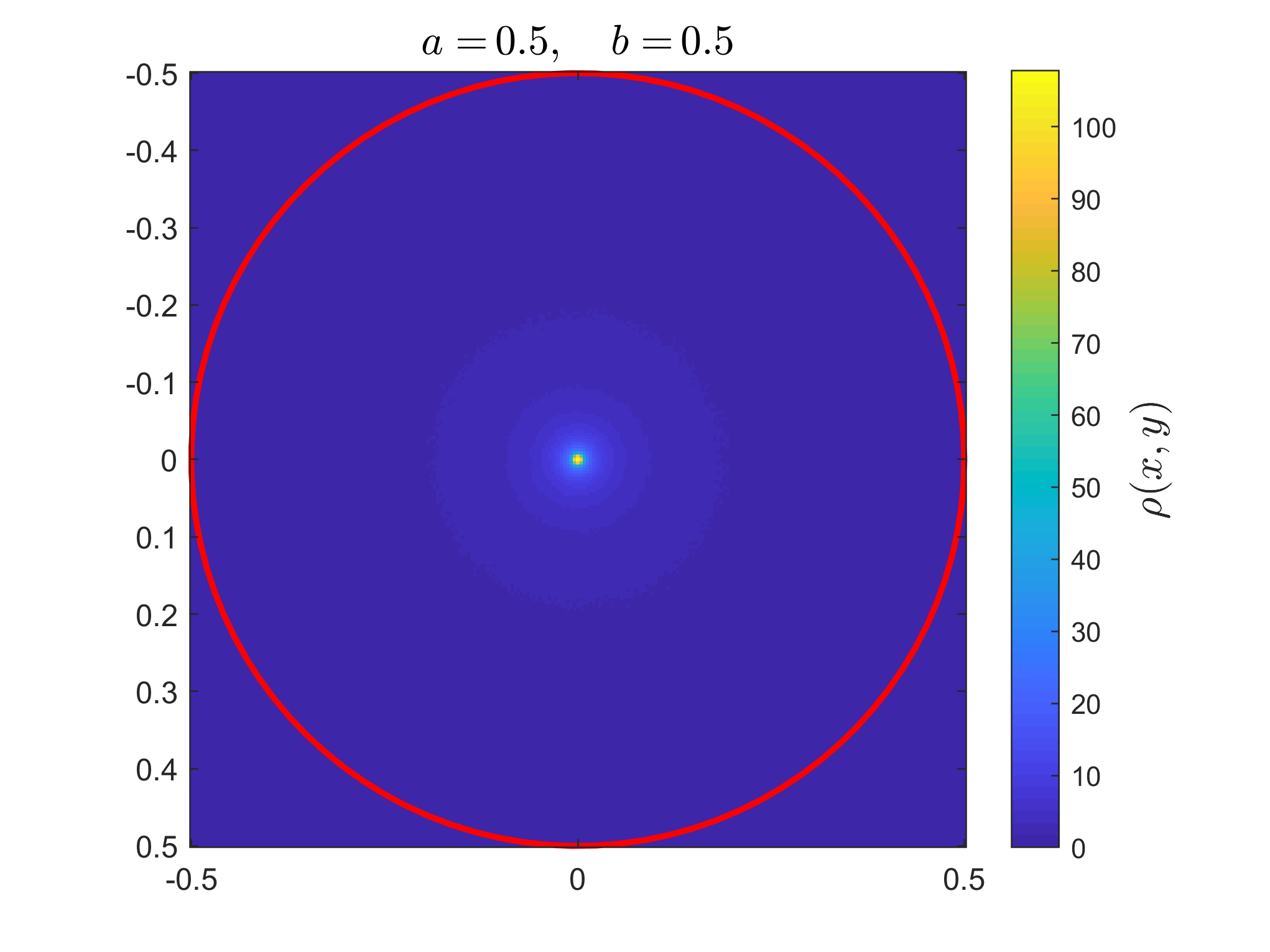}
	\end{subfigure}
	\hfill
	\begin{subfigure}[t]{.48\columnwidth}
		\centering
		\includegraphics[width=\linewidth]{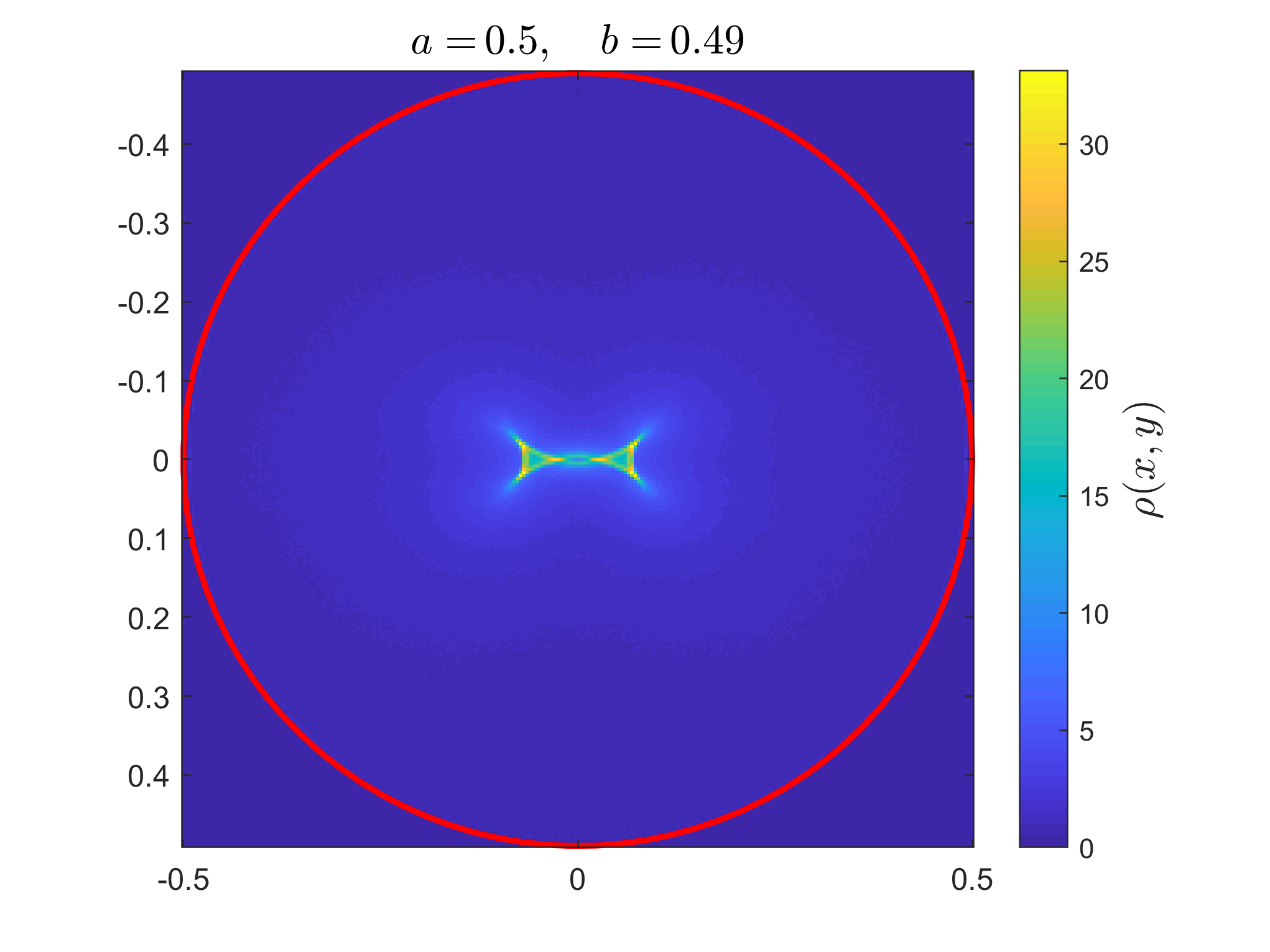}
	\end{subfigure}
	
	\medskip
	
	\begin{subfigure}[t]{0.48\columnwidth}
		\centering
		\includegraphics[width=\linewidth]{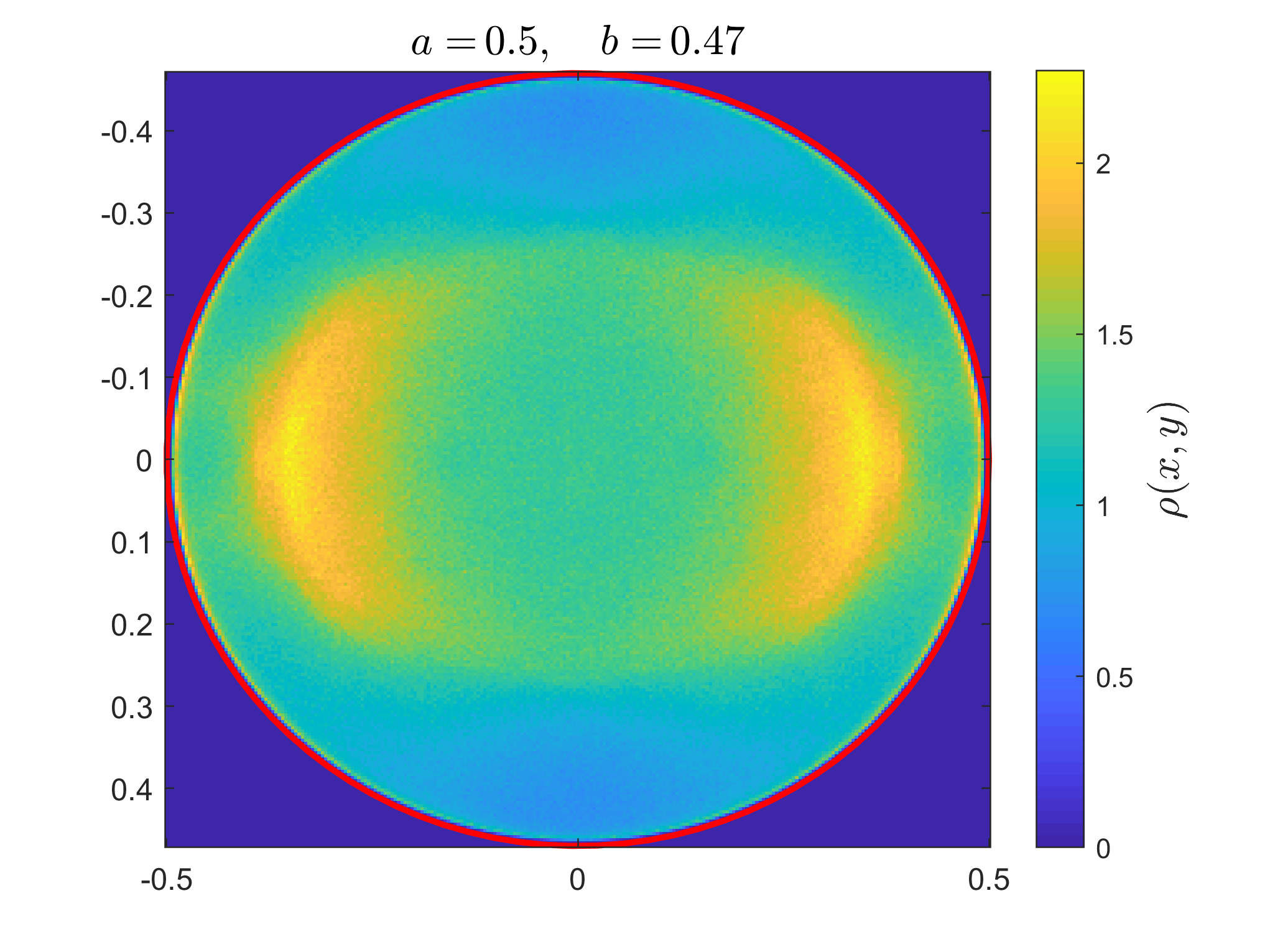}
	\end{subfigure}
	\hfill
	\begin{subfigure}[t]{0.48\columnwidth}
		\centering
		\includegraphics[width=\linewidth]{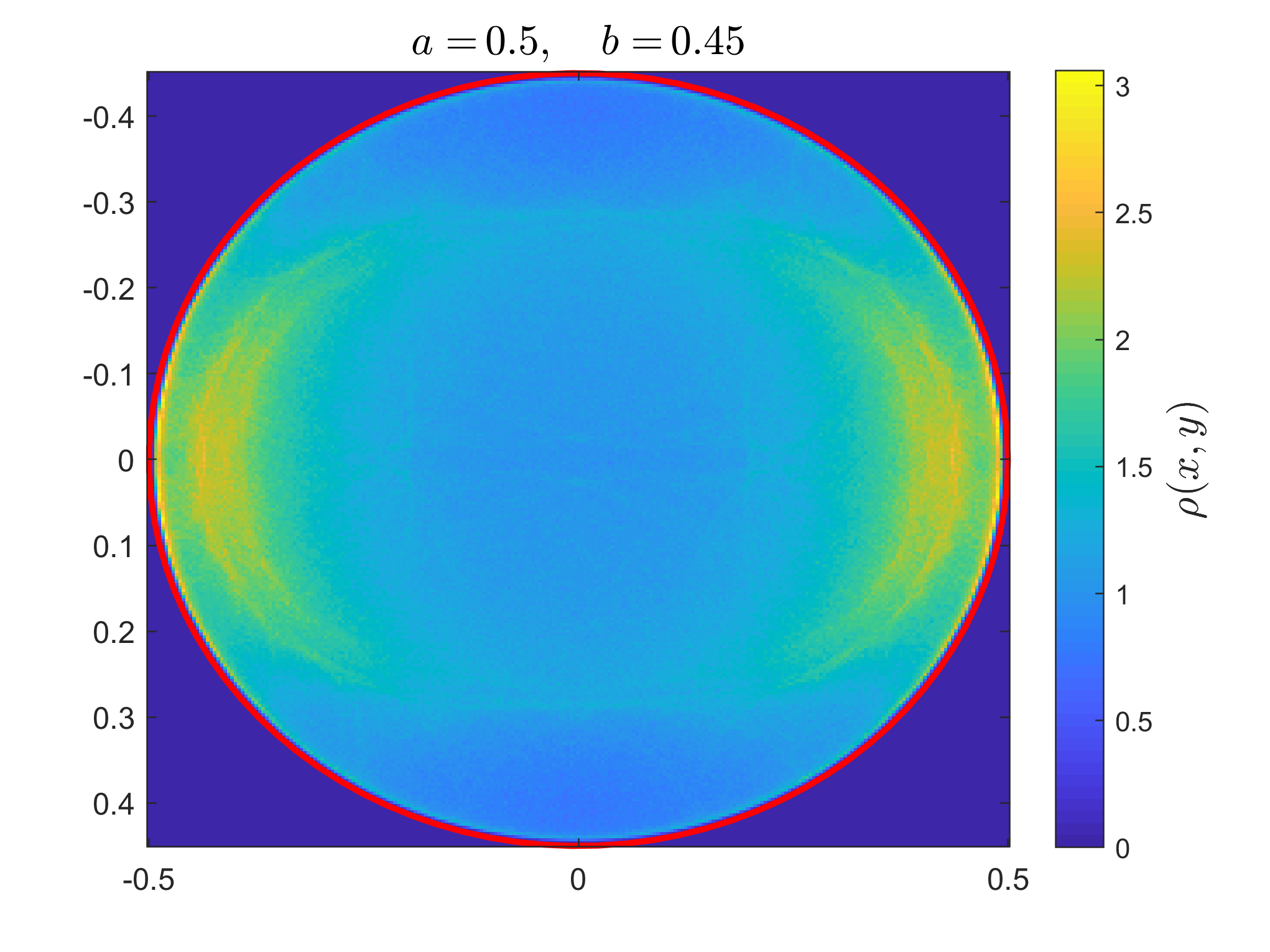}
	\end{subfigure}
	\caption{Plots of the asymptotic locations of the penny for various semi-minor axis lengths. The density $\rho$ is approximated by dividing $E$ into 200 by 200 boxes and tracking the location of $10^6$ time-1 maps.
	}
	\label{fig:locations_penny}
\end{figure}

\section{Conclusion and future work}\label{sec:future}
This work was motivated by the goal of understanding the set $\mathscr{A}_\mathcal{H}$ of hybrid-invariant differential forms. Straight-forward testable conditions were developed (Theorem \ref{th:energy_specular}) which allowed for a trivial proof that unconstrained hybrid mechanical systems are symplectic, and consequently, volume preserving (Proposition \ref{prop:unconst_hamilt_volume}).

When restricted to volume forms, the conditions prescribed by Theorem \ref{th:energy_specular} produce a ``hybrid cohomology'' equation \eqref{eq:hybrid_cohom}. This has a trivial solution for unconstrained mechanical systems. Existence of solutions for nonholonomic systems is completely controlled by the continuous component as the discrete component is trivial (Theorem \ref{th:nonh_hybrid_jac}). As the existence of an invariant volume is independent on the structure of the impacts, existence of an invariant volume for the continuous component results in hybrid-invariant volumes independent of the impacts (Proposition \ref{prop:hybrid_volume}).

Finally, the existence of an invariant volume imposes considerable limitations on Zeno solutions in hybrid systems (Theorem \ref{thm:noZeno}); although this result does not apply to spasmodic Zeno states, only steady Zeno states. Moreover, it is not needed that the continuous component be volume-preserving; as long as the impacts are not dissipative, Zeno trajectories are still controlled (Theorem \ref{th:zeno_impacts}). In particular, all elastic mechanical systems (holonomic or nonholonomic) possess almost no Zeno states.

We draw attention to two possible future extensions of this work: hybrid integrators and hybrid brackets.

For a given $\alpha\in\mathscr{A}_\mathcal{H}$, the goal is to construct a numerical integrator which preserves this form. This idea is explored in \cite{impactSymp} where an integrator is constructed which preserves the symplectic form for hybrid Hamiltonian systems. One way to further this topic is to find a systematic method to generate integrators for \textit{arbitrary} hybrid systems which preserve an \textit{arbitrary} invariant differential form $\alpha\in\mathscr{A}_\mathcal{H}$.

The next direction is concerned entirely with hybrid mechanical systems. It would seem natural to define a hybrid bracket (which mimics the usual Poisson bracket for continuous Hamiltonian systems) in the following way. Let $\mathfrak{D}:C^\infty(T^*Q)\to C^\infty(S^*,T^*Q)$ be given by $\mathfrak{D}(H) = \Delta_H$ such that $\Delta_H$ satisfies \eqref{eq:hamiltonian_impact}. Then we would define the hybrid bracket via
\begin{equation*}
	\left\{ f,g\right\}_\mathcal{H} = \begin{cases}
		\{f,g\}, & x\not\in S^*, \\
		\mathfrak{D}(f)^*g - g, & x \in S^*,
	\end{cases}
\end{equation*}
where $\{f,g\}$ is the usual Poisson bracket. However, this is problematic for (at least) four reasons.
\begin{itemize}
	\item It is not clear that the hybrid bracket is either smooth or continuous.
	\item The hybrid bracket may not be skew.
	\item Certain conditions are needed for $H$ to guarantee a single, nontrivial, solution to \eqref{eq:hamiltonian_impact}.
	\item Impacts might not occur; for example if $S = \{x^2+y^2=1\}$, then impacts never occur with the Hamiltonian $H = xp_y-yp_x$ but do with the Hamiltonian $H = p_x^2+p_y^2$.
\end{itemize}
We intend to address these issues in future work.



\medskip

\medskip

\end{document}